\documentclass[11pt, a4paper]{article}

\usepackage[utf8]{inputenc}
\usepackage[T1]{fontenc}
\usepackage{lmodern}
\usepackage{microtype}
\usepackage[pdfborder={0 0 0}]{hyperref}
\usepackage{color}
\usepackage{paralist}
\usepackage{graphicx}
\usepackage[all]{xy}
\xyoption{rotate}
\usepackage{verbatim}
\usepackage{tikz,pgfplots}
\usepackage{xcolor}

\usepackage[tbtags]{amsmath} 
\usepackage{amssymb}
\usepackage{amsthm}
\usepackage{mathtools}

\allowdisplaybreaks[2]

\usepackage{marginnote}

\usepackage[text={16.5cm, 24.2cm}, centering]{geometry}
\newcommand{\N}{\mathbb N}
\newcommand{\F}{\mathbb F}
\newcommand{\Z}{\mathbb{Z}}
\newcommand{\C}{\mathbb{C}}

\newcommand{\Set}[0]{\mathrm{Set}}
\newcommand{\Top}[0]{\mathrm{Top}}
\newcommand{\Cat}[0]{\mathrm{Cat}}
\newcommand{\SSet}[0]{\mathrm{SSet}}

\newcommand{\Hom}[0]{\mathrm{Hom}}
\newcommand{\Ob}[0]{\mathrm{Ob}}

\newcommand{\mac}[0]{\mathcal {C}}
\newcommand{\mad}[0]{\mathcal{D}}

\newcommand{\inv}[0]{{-1}}
\newcommand{\oo}[0]{\otimes}
\newcommand{\id}[0]{\mathrm{id}}

\newcommand{\low}[2]{{#1}_{({#2})}}

\newcommand{\brhd}[0]{\blacktriangleright}

\newtheorem{theorem}{Theorem}[section]
\newtheorem*{theorem*}{Theorem}
\newtheorem{example}[theorem]{Example}
\newtheorem{lemma}[theorem]{Lemma}
\newtheorem{proposition}[theorem]{Proposition}
\newtheorem{corollary}[theorem]{Corollary}
\newtheorem{definition}[theorem]{Definition}
\newtheorem{remark}[theorem]{Remark}

\newcommand*\stdthebibliography{}
\let\stdthebibliography\thebibliography
\renewcommand{\thebibliography}[1]{%
\stdthebibliography{#1}\setlength{\itemsep}{-2pt}} 

\newlength\oversetwidth
\newlength\underwidth

\def\mytitle{Categorical generalisations of quantum double models}
\def\myauthors{A.-K.~Hirmer, C.~Meusburger}
\hypersetup{pdftitle=\mytitle, pdfauthor=\myauthors}

\begin{document}

\begin{center}
  {\huge\mytitle}

  \vspace{1em}

  {\large
    Anna-Katharina Hirmer\footnote{{\tt anna-katharina.hirmer@fau.de}}\\
   Catherine Meusburger\footnote{{\tt catherine.meusburger@math.uni-erlangen.de}}
 }

 Department Mathematik \\
  Friedrich-Alexander-Universit\"at Erlangen-N\"urnberg \\
  Cauerstra\ss e 11, 91058 Erlangen, Germany\\[+2ex]

{June 9, 2023}

\begin{abstract}
We show that every involutive Hopf monoid in a complete and finitely cocomplete symmetric monoidal category gives rise to  invariants of oriented surfaces
 defined in terms of ribbon graphs. For every ribbon graph this yields an object in the category, defined up to isomorphism, that depends only on the homeomorphism class of the associated surface. This object is 
constructed via (co)equalisers and images and equipped with a mapping class group action. It  can be viewed as a categorical generalisation of the ground state of  Kitaev's quantum double model or of a representation variety for a surface. We apply the construction to group objects in cartesian monoidal categories, in particular to simplicial groups as group objects in $\SSet$ and to crossed modules as group objects in $\Cat$. The former yields a simplicial set consisting of representation varieties, the latter a groupoid whose sets of objects and morphisms are obtained from representation varieties. 
  \end{abstract}
\end{center}

\section{Introduction}
\label{sec:intro}

Constructions that assign algebraic or geometric objects  with mapping class group actions to oriented surfaces  are of interest in many contexts. 
They arise in 3d topological quantum field theories (TQFTs) of  Turaev-Viro-Barrett-Westbury   or Reshetikhin-Turaev type \cite{TV,BW,RT}  and are encoded in the weaker notion of a modular functor, see  \cite[Def.~5.1.1]{BK}. For a recent construction of modular functors from finite tensor categories, see Fuchs, Schweigert and Schaumann \cite{FSS}, for a classification  via  factorisation homology,  Brochier and Woike \cite{BrW}. 

Modular functors also arise  in Hamiltonian quantisation formalisms for  representation varieties. The associated mapping class group actions were first discovered in the combinatorial quantisation formalism by Alekseev, Grosse, Schomerus \cite{AGS1,AGS2,AS} and Buffenoir and Roche \cite{BR1,BR2}, subsequently related to factorisation homology by Ben-Zvi, Brochier and Jordan \cite{BBJa,BBJb} and studied by Faitg \cite{Fa18,Fa19}. 

Objects with mapping class group actions also arise from  correlators in conformal field theories, see  the work by Fuchs, Schweigert and Stigner  \cite{FS,FSS14}. They are also present 
in models from condensed matter physics and topological quantum computing such as Levin-Wen models \cite{LW1} and Kitaev's quantum double model \cite{Ki}.  
Due to the work of Lyubashenko \cite{Ly95a,Ly95b,Ly96}  it is well-understood how to construct projective mapping class group actions  from Hopf algebras in abelian ribbon categories. 

Many of these constructions are based on assignments of algebraic data to certain  graphs on surfaces, and they require linear categories   with duals,  often  also  abelian,  finite or semisimple. These restrictions are of course well-motivated from the context of TQFTs, in the quantisation of gauge theories or  in condensed matter physics. However, it  is also desirable to go beyond them. 

In this article we show that any involutive Hopf monoid $H$ in  a complete and finitely cocomplete symmetric monoidal category $\mac$ yields an invariant of  oriented surfaces. We compute this invariant for  examples, such as simplicial groups as Hopf monoids in $\SSet$ and crossed modules as Hopf monoids in $\Cat$. 

The construction is based on the choice of a ribbon graph. It  determines an  object in $\mac$, defined up to isomorphisms, that depends only  on the homeomorphism class of the    surface obtained by attaching discs to the faces of the graph. As Hopf monoids are categorical generalisations of Hopf algebras, it can be viewed as a categorical generalisation of Kitaev's quantum double model or  of representation varieties or moduli spaces of flat bundles on surfaces.

More precisely, we consider for  a ribbon graph the $\vert E \vert$-fold tensor product $H^{\oo E}$, where $E$ is the edge set of the graph. We  use the structure morphisms of the Hopf monoid to associate  $H$-module structures to its vertices and  $H$-comodule structures to its faces. This requires a choice of a marking for each vertex or face, and each  marking  defines a Yetter-Drinfeld module structure over $H$. 
The  object assigned to the graph is  obtained  by equalising the $H$-comodule structures, by coequalising the $H$-module structures and by combining them  via a categorical image. It generalises  the protected space or ground state of Kitaev's quantum double model, and we therefore call it the protected object. We then show

\begin{theorem*} (Theorem \ref{th:topinv}) The isomorphism class of the protected object for  a ribbon graph depends only on the homeomorphism class of the associated surface.
\end{theorem*}

Essentially the same construction was  used by Meusburger and Vo\ss\, in  \cite{MV} to construct mapping class group actions from pivotal Hopf monoids in symmetric monoidal categories. These mapping class group actions are obtained from graphs with a single vertex and face  and act on the associated protected object. However, it was not established in \cite{MV} that this object is independent of the graph. By combining our results with the ones from \cite{MV} we obtain

\begin{theorem*} (Theorem \ref{th:mcgact})  The protected object for a Hopf monoid $H$ and a surface $\Sigma$ of genus $g\geq 1$ is equipped with an action of the mapping class group $\mathrm{Map}(\Sigma)$ by automorphisms. 
\end{theorem*}

The remainder of this article is dedicated to the study of examples. For a finite-dimensional semisimple Hopf algebra $H$ as a Hopf monoid in  $\mathrm{Vect}_\C$ the protected object assigned to a surface coincides with the protected space of the associated quantum double model, as defined by Kitaev and by Buerschaper et a.~in \cite{Ki,BMCA}. However, our model is also defined in the non-semisimple case. For instance, for a group algebra $k[G]$  over a commutative ring $k$ as a Hopf monoid in $k$-Mod and a surface $\Sigma$ of genus $g\geq 1$, the protected object is the free $k$-module generated by the representation variety $\Hom(\pi_1(\Sigma), G)/G$.

A large class of examples of involutive  Hopf monoids  are group objects in cartesian monoidal categories, where the tensor product is a categorical product. In all our examples of this type, the protected object is given in terms of representation varieties. For a group $H$ as a Hopf monoid in $\Set$, the object assigned to a connected surface $\Sigma$ is the representation variety $\Hom(\pi_1(\Sigma), H)/H$.  For a simplicial group $H=(H_n)_{n\in\Z}$ as a Hopf monoid in $\SSet$, it is a simplicial set given by the representation varieties for the  groups $H_n$ and post-composition with the face maps and degeneracies. In this sense the construction can be viewed as a  generalisation of representation varieties from groups to group objects. The module structures at vertices correspond to the  group action and the comodule structures at faces of the graph to moment maps.

Our main example  is the case where the underlying symmetric monoidal category is the category $\Cat$ of small categories and functors between them. Group objects in $\Cat$ are precisely crossed modules. They are given by a group homomorphism $\partial: A\to B$ and an action $\brhd: B\times A \to A$ by group automorphisms, subject to some consistency conditions. As Cat is complete and cocomplete, the associated protected objects exist, but are difficult to determine concretely. 
We relate them to simplicial groups  via the nerve functor and its left adjoint, which yields an explicit description of the protected objects and their mapping class group actions.

\begin{theorem*} (Theorem \ref{th:kitaevexplicit}, Corollary \ref{cor:mapclasscross}) The protected object for a crossed module $(B, A,\brhd,\partial)$  as a group object in $\Cat$ and a  surface $\Sigma$ of genus $g\geq 1$ is a groupoid $\mathcal G$  with $\mathrm{Ob}\mathcal G=\Hom(\pi_1(\Sigma), B)/B$  and  with equivalence classes of group homomorphisms $\tau: \pi_1(\Sigma)\to A\rtimes B$ as morphisms. The action of the mapping class group $\mathrm{Map}(\Sigma)$ is induced by its action on $\Hom(\pi_1(\Sigma), A\rtimes B)/A\rtimes B$.
\end{theorem*}

The equivalence classes of morphisms are given by the equivalence relation $\tau_1\circ \tau_2\sim \tau'_1\circ \tau'_2$ on the set of group homomorphisms $\tau: F_{2g}\to A\rtimes B$, whenever the composites exist and $\tau_1$, $\tau'_1$ and $\tau'_2$,$\tau_2$ are conjugate. We compute the protected object and the associated mapping class group action explicitly for some simple examples of crossed modules. 

To our knowledge this construction is new and differs substantially from the  constructions with crossed modules in  higher gauge theory settings such as the work of Martins and Picken \cite{MP2,MP} on higher holonomies, the work \cite{BCKM1,BCKM2} by Bullivant et.~al.~on higher lattices and topological phases, the work  \cite{KMM} by Koppen, Martins and Martin on topological phases from crossed modules of Hopf algebras and  the recent work \cite{SV} by Sozer and Virelizier on 3d homotopy quantum field theory.   
In those settings the structure maps of the crossed module often encode higher categorical structures in the topological data, such as homotopies between paths, or relate data  in triangulations or cell decompositions.  
In our approach they enter the formalism as data - a specific example of a group object - but the crossed module structure is not required 
to encode the topology or geometry. 

The article is structured as follows. Section \ref{sec:Hopf} introduces the algebraic background for the article. In Section \ref{sec:Hopfmonoids} we summarise the background on  Hopf monoids in symmetric monoidal categories. In Section \ref{subsec:modinv} we discuss their (co)modules and the construction of their (co)invariants via (co)equalisers and images in complete and finitely cocomplete symmetric monoidal categories. Section \ref{sec:ribbon} contains the required background on ribbon graphs and surfaces.

In Section \ref{sec:protected space} we formulate the categorical counterpart of Kitaev's quantum double model for an involutive Hopf monoid $H$ in a complete and finitely cocomplete symmetric monoidal category. This is a simple generalisation of the formulation in \cite{BMCA} for  finite-dimensional semisimple complex Hopf algebras, and an almost identical construction was used in \cite{MV}.  

For each ribbon graph we consider the tensor product $H^{\oo E}$, where $E$ is the edge set of the graph. We assign to each marked vertex a $H$-module structure and to each marked face an $H$-comodule structure on $H^{\oo E}$.  The protected object  is  constructed by (co)equalising these (co)module structures and taking an image.

In Section \ref{sec:graphindependence} we show that the protected object defined by a ribbon graph depends only on the homeomorphism class of the associated oriented surface $\Sigma$. 
We first demonstrate that moving the markings for the (co)module structures and edge reversals yield isomorphic protected objects. We then consider a number of graph transformations that are sufficient to reduce every connected ribbon graph to a standard graph and prove that these induce isomorphisms of the protected object. 

These sections are necessarily rather technical.
The reader primarily interested in the results may skip to the main theorem in Section \ref{subsec:topinv}, where we also treat some examples. In particular, we show that the protected object for a group $H$ as a Hopf monoid in $\Set$ and a connected surface $\Sigma$ is the representation variety $\Hom(\pi_1(\Sigma), H)/H$. We then consider  group algebras $H=k[G]$ and their duals $k[G]^*$ for a commutative ring $k$ as Hopf monoids in $k\text{-Mod}$. The associated protected objects are  the free $k$-module $\langle \Hom(\pi_1(\Sigma), G)/G\rangle_k$ and the set of maps $\Hom(\pi_1(\Sigma), G)/G\to k$. 

Section \ref{sec:sset} treats the example of a simplicial group $H=(H_n)_{n\in\N_0}$ as a Hopf monoid in $\SSet$. In this case, the protected object is a simplicial set given by the representation varieties $\Hom(\pi_1(\Sigma), H_n)/H_n$ and by post-composition with the face maps and degeneracies of $H$. This result is required for the construction of the protected object of a crossed module as a group object in $\Cat$.

The construction of the protected object for group objects in $\Cat$ is more involved and treated in Section \ref{sec:cat}.  We start by summarising the required background on crossed modules in Section \ref{subsec:crossed modules}. In Section \ref{subsec:coequcat} we discuss equalisers and coequalisers in $\Cat$ and summarise how the latter can be constructed via the nerve functor $N:\Cat\to \SSet$ and its left adjoint. In Section \ref{subsec:modcomodgrp} we apply these results to (co)equalise the (co)module structures over Hopf monoids in $\Cat$.  
In Section \ref{subsec:protected} we apply this to the (co)module structures associated with a ribbon graph and determine the protected object for the associated surface. We describe it explicitly and treat a simple example. 

In Section \ref{sec:mapclass} we describe the mapping class group action on the protected object. By combining our results with the ones from \cite{MV}, 
we obtain that the mapping class group of an oriented surface $\Sigma$ acts on the associated protected objects. We show that  in the case of a simplicial group $H=(H_n)_{n\in\N_0}$ this action is the one induced by its  action on the representation varieties $\Hom(\pi_1(\Sigma), H_n)/H_n$.
For the case of a crossed module  $(B, A,\brhd,\partial)$ as a group object in $\Cat$  we obtain a mapping class group action by invertible endofunctors on the associated groupoid, which is 
induced by the action on the representation variety $\Hom(\pi_1(\Sigma), A\rtimes B)/A\rtimes B$.

\section{Algebraic background}
\label{sec:Hopf}

\subsection{Involutive Hopf monoids}
\label{sec:Hopfmonoids}

Throughout the article $\mac$ is a symmetric monoidal category with unit object $e$ and  braidings $\tau_{X,Y}: X\oo Y\to Y\oo X$.  
We also suppose that $\mac$ is complete and finitely cocomplete. 
In  formulas, we suppress associators and unit constraints and  coherence data of monoidal functors.

\begin{definition}\label{def:hopfmonoid} $\quad$
\begin{compactenum}
\item A {\bf Hopf monoid} in $\mac$ is 
an object $H$ in $\mac$ together with morphisms $m: H\oo H\to H$, $\eta:e\to H$, $\Delta: H\to H\oo H$, $\epsilon: H\to e$ and $S: H\to H$, the multiplication, unit, comultiplication, counit and antipode,  such that
\begin{compactitem}
\item  the (co)multiplication satisfies the (co)associativity and (co)unitality conditions
\begin{align}\label{eq:counit}
&m\circ (m\oo 1_H)=m\circ (1_H\oo m), & &m\circ (\eta\oo 1_H)=m\circ (1_H\oo \eta)=1_H,\\
&(\Delta\oo 1_H)\circ\Delta=(1_H\oo \Delta)\circ \Delta, & &(\epsilon\oo 1_H)\circ\Delta
=(1_H\oo \epsilon)\circ \Delta=1_H,\nonumber
\end{align}
\item comultiplication and counit are monoid morphisms
\begin{align}
&\Delta\circ \eta=\eta\oo \eta, & &\Delta\circ m=(m\oo m)\circ (1_H\oo \tau_{H,H}\oo 1_H)\circ (\Delta\oo \Delta)\\
&\epsilon\circ\eta=1_e, & & \epsilon\circ m=\epsilon\oo \epsilon,\nonumber
\end{align}
\item $S$ satisfies the antipode condition
\begin{align}
m\circ (S\oo 1_H)\circ\Delta=m\circ (1_H\oo S)\circ \Delta=\eta\circ\epsilon.
\end{align}
\end{compactitem}
It is called {\bf involutive} if $S\circ S=1_H$.\\

\item A {\bf morphism of Hopf monoids} in $\mac$ is a morphism $f: H\to H'$ in $\mac$ with 
\begin{align}
&f\circ m=m'\circ (f\oo f), & &f\circ \eta=\eta', & &(f\oo f)\circ\Delta=\Delta'\circ f, & &\epsilon'\circ f=\epsilon.
\end{align}

\end{compactenum}

We denote by $\mathrm{Hopf}(\mac)$ the category of Hopf monoids and morphisms of Hopf monoids in $\mac$.
\end{definition}

The antipode of a Hopf monoid is unique, and it is an anti-monoid and anti-comonoid morphism
\begin{align}\label{eq:antialg}
&S\circ m=m^{op}\circ (S\oo S), & S\circ \eta=\eta, &  & (S\oo S)\circ \Delta=\Delta^{op}\circ S, & &\epsilon\circ S=\epsilon,
\end{align}
see for instance Porst \cite[Prop.~36]{Po}. If $H$ is involutive, the antipode satisfies the additional identities 
\begin{align}\label{eq:EightWithAntipode}
m^{op}\circ (S\oo 1_H)\circ\Delta=m^{op}\circ (1_H\oo S)\circ\Delta=\eta\circ\epsilon.
\end{align}

Every morphism of Hopf monoids $f: H\to H'$ satisfies $f\circ S=S'\circ f$. This follows  as for Hopf algebras by considering the convolution monoid $\Hom_\mac(H,H)$ with the  product $f\star g=m\circ (f\oo g)\circ \Delta$.  

In the following, we use generalised Sweedler notation for the coproduct in a Hopf monoid and write $\Delta(h)=\low h 1\oo \low h 2$, $(\Delta\oo 1_H)\circ\Delta(h)=(1_H\oo\Delta)\circ\Delta(h)=\low h 1\oo \low h 2\oo \low h 3$ etc. This is analogous to Sweedler notation for a Hopf algebra. It can  be viewed as a shorthand notation for a diagram that describes a morphism in a symmetric monoidal category, see \cite{MV} for examples. We also write $m^{(n)}: H^{\oo (n+1)}\to H$ and $\Delta^{(n)}: H\to H^{\oo (n+1)} $ for $n$-fold products and coproducts.

\begin{example} \label{ex:Hopf-monoids} $\quad$
\begin{compactenum}
\item For any commutative ring $k$ a  Hopf monoid in   $k$-Mod  is a Hopf algebra over $k$. In particular, for any field $\F$  a Hopf monoid in $\mathrm{Vect}_\F$ is a Hopf algebra over $\F$. 

\item For any finite group $G$ and commutative ring $k$, the group algebra $k[G]$ and its dual $k[G]^*$  are Hopf monoids in $k\mathrm{-Mod}$.
\item The tensor product  of two Hopf monoids  in $\mathcal C$ has a Hopf monoid structure  given by  the tensor product of (co)units, (co)multiplications  and antipodes and the braiding morphisms. Any tensor product of Hopf monoid morphisms is a morphism of Hopf monoids.

\item Every Hopf monoid $H=(H,m,\eta,\Delta,\epsilon, S)$ in a symmetric monoidal category $\mac$  defines a Hopf monoid $ H^*=(H,\Delta,\epsilon, m,\eta, S)$ in the symmetric monoidal category $\mac^{op}$. This generalises the dual Hopf algebra  in $\mathrm{Vect}_\F$.

\end{compactenum}
\end{example}

The following example yields many  subexamples, which are a focus in this article.

\begin{example}\label{ex:groupobjex} Let $(\mac, \times)$ be a cartesian monoidal category with terminal object $\bullet$. Let  $\epsilon_X: X\to \bullet$ be the terminal morphism  and    $\Delta_X: X\to X\times X$ the diagonal morphism for an object $X$. 

A Hopf monoid in $\mac$ is a {\bf group object} in $\mac$:   an object $H$ together with morphisms $m: H\times H\to H$, $\eta :\bullet\to H$ and $I:H\to H$
such that the following diagrams commute
\begin{align}\label{eq:groupobjdiags}
&\xymatrix{
H\times H\times H \ar[r]^{\quad 1_H\times m} \ar[d]_{m\times 1_H}& H\times H\ar[d]^m\\
H\times H\ar[r]_m & H
} &
&\xymatrix{H\cong \bullet\times H \ar[r]^{\;\eta\times 1_H}\ar[rd]_{1_H} & H\times H\ar[d]^m & \ar[l]_{1_H\times \eta} H\times \bullet\cong H \ar[ld]^{1_H}\\
& H
}\\
&\xymatrix{ H\times H \ar[rr]^{I\times 1_H} & & H\times H \ar[d]^m\\
H \ar[u]^{\Delta_H} \ar[r]^{\epsilon_H} & \bullet \ar[r]^{\eta}  & H
}
&
&\xymatrix{ 
H\ar[d]_{\Delta_H} \ar[r]^{\epsilon_H} & \bullet \ar[r]^{\eta}  & H\\
H\times H \ar[rr]^{1_H\times I} & & H\times H. \ar[u]_m
}\nonumber
\end{align}

A morphism of Hopf monoids  is a {\bf morphism of group objects}: a morphism $F: H\to H'$ with
\begin{align}\label{eq:groupmorphism}
 F\circ m=m'\circ (F\times F).
\end{align}
Note that this implies  $F\circ \eta=\eta'$ and $I'\circ F=F\circ I$.
\end{example}

\smallskip
\begin{example}\label{ex:groupobjects} $\quad$
\begin{compactenum}
\item A group object in the cartesian monoidal category 
$(\Set,\times)$ is a group.
\item A group object in the cartesian monoidal category 
$(\Top,\times)$ is a topological group.

\item A group object in the cartesian monoidal category $(\Cat,\times)$ of small categories and functors between them  is a crossed module (cf.~Definition \ref{Def:CrossedModule}). 
\item Let $G$ be a group  and $G\mathrm{-Set}=\Set^{\mathrm{B} G}$ the cartesian monoidal category of $G$-sets and $G$-equivariant maps. A 
group object in  $G\mathrm{-Set}$   is a group   with a $G$-action by  automorphisms. 
\item A group object in the cartesian monoidal category $\SSet=\Set^{\Delta^{op}}$ of simplicial sets and simplicial maps  is a simplicial group (cf.~Definition \ref{lem:simplicialgroup}). 

\end{compactenum}
\end{example}

The last two examples in Example \ref{ex:groupobjects} have counterparts for any  functor category $\mac^\mad$, where $\mad$ is small and $\mac$ symmetric monoidal. In this case the functor category $\mac^\mad$ inherits a symmetric monoidal structure from $\mac$, and we have 

\begin{lemma}\label{lemma:HopfMonoidsInFunctorCategory}
For any symmetric monoidal category $\mathcal C$ and a small category $\mathcal D$ the monoidal categories $\mathrm{Hopf}(\mac^\mad)$ and $\mathrm{Hopf}(\mac)^\mad$
are  symmetric monoidally equivalent.
\end{lemma}
\begin{proof} The equivalence is given by
the functor $R: \mathrm{Hopf}(\mac^\mad)\to \mathrm{Hopf}(\mathcal C)^\mad$ that sends a Hopf monoid $(H, m,  \eta, \Delta, \epsilon, S)$ to the functor $K: \mathcal D\to \mathrm{Hopf}(\mathcal C)$ with $K(D)=H(D)$ and the component morphisms $m_D$, $\eta_D$, $\Delta_D$, $\epsilon_D$, $S_D$ for $D\in \text{Ob}(\mathcal D)$ and with $K(f)=H(f)$ for a morphism $f$ in $\mathcal D$. Hopf monoid morphisms  in $\mac^\mad$ are sent to themselves. The functor $R$ has an obvious inverse, and both functors  are symmetric monoidal.\end{proof}

Further examples are obtained by taking the images of Hopf monoids under symmetric monoidal functors. If both of the categories  are cartesian monoidal,  it is sufficient that 
the functor preserves finite products, which holds in particular for any right adjoint functor.

\begin{example}\label{ex:transport0} $\quad$
\begin{compactenum}
\item  Let $F: \mac\to\mac'$ be a symmetric monoidal functor. Then for every Hopf monoid $H$ in $\mac$ the image  $F(H)$ has a canonical Hopf monoid structure. 
\item If $\mac,\mac'$ are  cartesian monoidal categories and $F:\mac\to\mac'$ a functor that preserves finite products, then  $F$ is symmetric monoidal, and for every group object $H$ in $\mac$ the image $F(H)$ is a group object in $\mac'$.
\end{compactenum}
\end{example}

\subsection{(Co)modules and their (co)invariants}
\label{subsec:modinv}

As their definitions involve only structure maps, (co)modules  over Hopf monoids in symmetric monoidal categories are defined analogously to (co)modules  over Hopf algebras. The only difference is that linear maps are replaced by morphisms. 

\begin{definition}\label{def:module} Let $H$ be a Hopf monoid in  $\mac$.
\begin{compactenum}
\item An $H$-{\bf module} in $\mac$ is an object $M$ in $\mac$ with a morphism $\rhd: H\oo M\to M$ satisfying
\begin{align}
\rhd\circ (m\oo 1_M)=\rhd\circ (1_H\oo\rhd), \qquad \rhd \circ (\eta\oo 1_M)=1_M.
\end{align}
A {\bf morphism of $H$-modules} is a morphism $f: M\to M'$ in $\mac$ with 
$\rhd'\circ (1_H \oo f)=f\circ \rhd$.\\

\item An $H$-{\bf comodule} in $\mac$ is an object $M$ in $\mac$ with a morphism $\delta: M\to  H\oo M$ satisfying
\begin{align}\label{eq:comod}
(\Delta\oo 1_M)\circ \delta=(1_H\oo \delta)\circ \delta, \qquad (\epsilon\oo 1_M)\circ\delta=1_M.
\end{align}
A {\bf morphism of $H$-comodules} is a morphism $f: M\to M'$ in $\mac$ with 
$(1_H \oo f)\circ \delta=\delta'\circ f$.
\end{compactenum}
\end{definition}

There are analogous notions of right (co)modules and  bi(co)modules  and morphisms between them. 
Just as in the case of a Hopf algebra, there are also various compatibility conditions that can be imposed between modules and comodule structures. 
The most important one in the following is the one for Yetter-Drinfeld modules.

\begin{definition}\label{def:yetterdrinfeld}  Let $H$ be a Hopf monoid in  $\mac$. 
\begin{compactenum}
\item
A
 {\bf Yetter-Drinfeld module} over $H$ is a triple $(M,\rhd,\delta)$ such that $(M,\rhd)$ is an $H$-module, $(M,\delta)$ is an $H$-comodule and 
 \begin{align*}
 \delta\circ \rhd=(m^{(2)}\oo \rhd)\circ (1_{H^{\oo 2}}\circ \tau_{H,H}\oo 1_M)\circ (1_{H^{\oo 3}}\oo S\oo 1_M)\circ (1_H\oo \tau_{H^{\oo 2},H}\oo 1_M)\circ (\Delta^{(2)}\oo \delta).
 \end{align*}
\item A {\bf morphism of Yetter-Drinfeld modules} is a morphism $f: M\to M'$ that is a module and a comodule morphism.
\end{compactenum}
\end{definition}

In Sweedler notation with  the conventions $\delta(m)=\low m 0\oo \low m 1$ and $\Delta(h)=\low h 1\oo \low h 2$ the Yetter-Drinfeld module condition in Definition \ref{def:yetterdrinfeld}  reads
\begin{align}
(h\rhd m)_{(0)}\oo (h\rhd m)_{(1)}=\low h 1 \low m 0 S(\low h 3)\oo (\low h 2\rhd \low m 1).
\end{align}
Yetter-Drinfeld modules over group objects in cartesian monoidal categories are especially simple to describe. In this case, composing the coaction morphism $\delta: M\to H\times  M$ with the projection morphism $\pi_1: H\times M\to H$ yields a morphism $F=\pi_1\circ \delta: M\to H$ reminiscent of a moment map. The Yetter-Drinfeld module condition  states that this morphism intertwines the $H$-module structure on $M$ and the conjugation action of $H$ on itself.

\begin{example}
\label{Ex:YDAndHopfModulesInCat} Let $H$ be a group object in a cartesian monoidal category, $(M, \rhd)$  a module and  $(M, \delta)$  a comodule over $H$. 
Then $(M, \rhd, \delta)$ is a Yetter-Drinfeld module  over $H$
 iff the morphism  
 $F:=\pi_1\circ \delta: M \to H$ satisfies
\begin{align}\label{eq:yddefgrp}
F\circ \rhd=m^{(2)}\circ (1_H\times\tau_{H,F(M)})\circ (1_H\times I\times 1_{F(M)})\circ(\Delta_H\times F).
\end{align}
\end{example}

If the objects of $\mac$ are sets,  condition \eqref{eq:yddefgrp} reads $F(h\rhd m)=hF(m) h^\inv$ for all $h\in H$, $m\in M$. By an abuse of notation, we  sometimes write such formulas  for the general case to keep notation simple.

By Example \ref{ex:transport0} the images of Hopf monoids  under symmetric monoidal functors are Hopf monoids. Analogous statements hold for their (co)modules.

\begin{example}\label{Ex:groupobjectfunctor}  $\quad$
\begin{compactenum}
\item If $F: \mac\to\mac'$  is a symmetric monoidal functor and $M$ a (co)module over a Hopf monoid $H$ in $\mac$, then $F(M)$ is a (co)module over the Hopf monoid $F(H)$. 

\item Let $\mac,\mac'$ be cartesian monoidal categories and $F: \mac\to\mac'$ a functor that preserves finite products. Then for every (co)module $M$ over a group object $H$ in $\mac$ the image $F(M)$ is a (co)module over the group object $F(H)$.
\end{compactenum}
\end{example}

(Co)invariants of (co)modules  cannot be generalised directly from Hopf algebras over fields to Hopf monoids in symmetric monoidal categories. To obtain  generalised notions of (co)invariants, we require that the symmetric monoidal category $\mac$ has all equalisers and coequalisers.

\begin{definition}\label{def:invariants} \cite[Def.~2.6]{MV} Let $\mac$ be a symmetric monoidal category that has all equalisers and coequalisers, $H$ a Hopf monoid in $\mac$. 

\begin{compactenum}
\item The {\bf invariants} of an $H$-module $(M,\rhd)$ are the coequaliser $(M^{H},\pi)$ of $\rhd$ and $\epsilon\oo 1_{M}$: 
\begin{align}
	\nonumber
\xymatrix{  H\oo M  \ar@<-.5ex>[r]_{\quad\epsilon\oo 1_M} \ar@<.5ex>[r]^{\quad\rhd} & M \ar[r]^{\pi} & M^H.
}
\end{align} 

	\item The {\bf coinvariants} of an $H$-comodule $(M,\delta)$ are the equaliser $(M^{coH},\iota)$ of $\delta$ and $\eta \oo 1_{M}$:
\begin{align}
	\nonumber
\xymatrix{ M^{coH} \ar[r]^{\iota} & M  \ar@<-.5ex>[r]_{\eta\oo 1_M\quad} \ar@<.5ex>[r]^{\delta\quad} & H\oo M.
}
\end{align}

\end{compactenum}
\end{definition}

As expected, $H$-(co)module  morphisms induce morphisms between the (co)invariants. This follows directly from the universal properties of the (co)equalisers.

\begin{lemma}\label{Lemma:IsosAufInvarianten}\cite[Lemma 2.7]{MV}
Suppose that $\mathcal C$ has all equalisers and coequalisers and $H$ is a Hopf monoid in $\mac$. Then for 
every $H$-module morphism $f: (M, \triangleright) \to (M', \triangleright')$ there is a unique morphism $f^H: M^H \to M'^H $ with $f^H \circ \pi = \pi' \circ f$. Likewise, for every $H$-comodule morphism $f: (M, \delta) \to (M', \delta')$ there is a unique morphism $f^{coH}: M^{coH} \to M'^{coH}$ with $\iota' \circ f^{coH}= f\circ \iota$.
\end{lemma}

Note that all definitions in this section are symmetric with respect to a Hopf monoid $H$ in $\mac$ and the dual Hopf monoid $ H^*$ in $\mac^{op}$ from Example \ref{ex:Hopf-monoids}. Modules and comodules over $H$ in $\mac$ correspond to comodules and modules over $ H^*$ in $\mac^{op}$, respectively, and the same holds for their (co)invariants. 
It is also directly apparent from the formula in Definition \ref{def:yetterdrinfeld} that Yetter-Drinfeld modules over $H$ correspond to Yetter-Drinfeld modules over $ H^*$.

For objects in a symmetric monoidal category $\mac$ that are both, modules and comodules over certain Hopf monoids in $\mac$, we combine the notion of invariants and coinvariants and impose both conditions. This requires that the category $\mac$ is equipped with {\em images}.  We work with a   general non-abelian notion of image, see Mitchell \cite[Sec.~I.10]{Mi} and Pareigis \cite[Sec.~1.13]{Pa}. There is an analogous notion of a {\em coimage}, which is  the image of the corresponding morphism   in  $\mac^{op}$, see \cite[Sec.~I.10]{Mi}.

An  {\em image} of a morphism $f: C\to C'$ in $\mac$ is an object $\mathrm{im}(f)$ together with a pair $(P,I)$ of a monomorphism $I: \mathrm{im}(f)\to C'$ and a morphism $P: C\to \mathrm{im}(f)$ with $I\circ P=f$ and the following universal property:   for any pair $(Q,J)$ of a monomorphism $J: X\to C'$ and a morphism $Q: C\to X$ with $J\circ Q=f$ there is a unique morphism $v:\mathrm{im}(f)\to X$ with $I=J\circ v$. 

Images are unique up to unique isomorphism. If $\mathcal C$ has all equalisers, then $P: C\to \mathrm{im}(f)$ is an epimorphism \cite[Prop.~10.1, Sec.~I.10]{Mi}. In an abelian category $\mathcal C$ this notion of  image coincides with the usual definition of an image as the kernel of the cokernel \cite[Lemma 3, Sec.~4.2]{Pa}. If $\mathcal C$ is complete, then all images exist, as any complete category has intersections \cite[Prop.~2.3, Sec.~II.2]{Mi}. This implies the existence of all images \cite[Sec.~I.10]{Mi}.

\begin{definition} \label{def:biinvariants} \cite[Def.~2.8]{MV}\footnote{Def.~2.8 in \cite{MV}  considers only the case $H=K$, as that is the only one required there.} Let $\mac$ be a complete and finitely cocomplete symmetric monoidal category and $H, K$  Hopf monoids in $\mac$. 
 The {\bf biinvariants} of an   $H$-module and  $K$-comodule $M$ are the image  of the morphism $\pi\circ\iota: M^{coK}\to M^H$  
\begin{align}\label{eq:imagedef}
\xymatrix{ M^{coK} \ar[rd]_P\ar[r]^\iota & M \ar[r]^\pi & M^H\\
& M_{inv}:=\mathrm{im}(\pi\circ \iota). \ar[ru]_{I}
}
\end{align}
\end{definition}

Requiring $\mathcal C$ to be complete and finitely cocomplete ensures the existence of invariants, coinvariants and biinvariants. Examples of such categories are $\mathrm{Set}$, $\mathrm{Top}$, $\mathrm{Grp}$, $\mathrm{Vect}_{\F}$, $\mathrm{Cat}$, $k\mathrm{-Mod}$ and the category $\mathrm{Ch}_{k\mathrm{-Mod}}$ of chain complexes of $k$-modules. 
For a small category $\mad$  and a complete and finitely cocomplete category $\mathcal C$ the category $\mac^\mad$ is also complete and finitely cocomplete, see for instance Pareigis \cite[Th.~1, Sec.~2.7]{Pa}. Hence,  $G\mathrm{-Set}$ and $\mathrm{SSet}$ also satisfy the requirement.

As discussed in \cite[Rem.~2.9]{MV} one could also consider the {\em coimage} of the morphism $\pi\circ \iota$ instead of its {\em image}.
This amounts to passing from  modules and comodules over the Hopf monoids $H,K$ in $\mac$ to comodules and modules over the Hopf monoids $H^*$, $K^*$  in $\mac^{op}$  from Example \ref{ex:Hopf-monoids}.

We illustrate (co)invariants and biinvariants with a few simple examples. (Co)invariants of (co)modules over Hopf monoids in $\SSet$ and $\Cat$ and the associated biinvariants for Yetter-Drinfeld modules are treated in Sections \ref{sec:sset} and \ref{subsec:modcomodgrp}, respectively.

\begin{example}\label{Ex:ImageInSet,Top}$\quad$
\begin{compactenum}
\item  A Hopf monoid $H$ in $\mac=\Set$ (in $\mac=\Top$) is a (topological) group $H$ and
\begin{compactitem}
\item an $H$-module is a (continuous) $H$-Set $\rhd: H\times M\to M$, 
\item  $M^H=\{H\rhd m\mid m\in M\}$  with  $\pi: M\to M^H$, $m\mapsto H\rhd m$ (and the  quotient topology),
\item  an $H$-comodule is given by a (continuous) map $F:M\to H$,
\item  $M^{coH}=F^\inv(1)$  with the inclusion $\iota: F^\inv(1)\to M$ (and  the subspace topology),

\item 
$
M_{inv}=\pi(F^\inv(1))=\{H\rhd m\mid F(m)=1\}
$
(with the final topology induced by $\pi$). 
\end{compactitem}

An $H$-module and $H$-comodule $(M,\rhd, F)$ is a  Yetter-Drinfeld module iff $F(h\rhd m)=hF(m)h^\inv$ for all $m\in M$, $h\in H$. \\[-1ex]

\item Let $G$ be a group and $H$ a group with a $G$-action by automorphisms, viewed  as a Hopf monoid in $G-\Set= \mathrm{Set}^{\mathrm{B}G}$. Then $H$-modules are $H\rtimes G$-sets, $H$-comodules are $G$-sets $M$ with $G$-equivariant maps $F: M\to H$ and 
\begin{compactitem}
\item $M^H=\{H\rhd m\mid m\in M\}$ is the orbit space for $H$ with the induced $G$-action and  $G$-equivariant canonical surjection $\pi: M\to M^H$,
\item  $M^{coH}=F^\inv(1)$ with the induced $G$-action and  $G$-equivariant inclusion $\iota: F^\inv(1)\to M$,
\item $M_{inv}=\pi(F^\inv(1))$ with the induced $G$-action. \\[-1ex]
\end{compactitem}

\item For a Hopf algebra $H$ over a commutative ring $k$ as a Hopf monoid in $k$-Mod, $H$-(co)modules and Yetter-Drinfeld modules are  (co)modules and Yetter-Drinfeld modules over $H$ in the usual sense. Their (co)invariants and biinvariants are 
\begin{compactitem}
\item $M^H=M/\langle \{h\rhd m-\epsilon(h) m\mid h\in H, m\in M\}\rangle$,
\item $M^{coH}=\{m\in M\mid \delta(m)=1\oo m\}$,
\item $M_{inv}=\pi(M^{coH})$.
\end{compactitem}
\end{compactenum}
\end{example}

While the coinvariants in Example \ref{Ex:ImageInSet,Top}, 3.~coincide with the usual  coinvariants for comodules over a Hopf algebra, 
the invariants form a quotient rather than a subset.
This distinction is irrelevant in the case of semisimple Hopf algebras, but not in general. As our definition is symmetric with respect to Hopf monoids in a symmetric monoidal category $\mac$ and the dual Hopf monoids in $\mac^{op}$, it is more natural in our setting.
The following example illustrates this. 

\begin{example} \label{Ex:ImageInKMod} For a finite group $G$ and a commutative ring $k$  the group algebra $k[G]$ and  its dual $k[G]^*$ are Hopf monoids in $k$-Mod.

For the group algebra $H=k[G]$ 
\begin{compactitem}
\item the invariants of a $H$-module $(M,\rhd)$ are $M^{H}=M/\langle \{g\rhd m-m\mid m\in M, g\in G\}\rangle$,
\item comodules are $G$-graded $k$-modules $M=\oplus_{g\in G} M_g$ with $\delta(m)=g\oo m$ for all $m\in M_g$,
\item their coinvariants are $M^{coH}=M_1$.
\end{compactitem}
A $k[G]$-module and comodule  $(M,\rhd, \delta)$ is a Yetter-Drinfeld module iff
$g\rhd M_h= M_{ghg^\inv}$ for all $g,h\in G$, and in this case
$M_{inv}\cong H_0(G, M_1)$.

For the dual Hopf monoid $H=k[G]^*$
\begin{compactitem}
\item modules are $G$-graded $k$-modules $M=\oplus_{g\in G} M_g$ with $\delta_g\rhd m=\delta_g(h) m$ for $m\in M_h$,
\item  their invariants are $M^H=M/(\oplus_{g\in G,g\neq 1} M_g)\cong M_1$,
\item comodules are $k[G]$-right modules $(M,\lhd)$ with $\delta(m)=\sum_{g\in G} \delta_g\oo (m\lhd g)$,
\item their coinvariants are $M^{coH}=\{m\in M\mid m\lhd g=m\,\forall g\in G\}$.
\end{compactitem}
A $k[G]^*$-module and comodule $(M,\rhd,\delta)$ is a Yetter-Drinfeld module iff $M_h\lhd g= M_{ghg^\inv}$ for all $g,h\in G$, and in this case $M_{inv}\cong H^0(G, M_1)$.

\end{example}

By Lemma \ref{Lemma:IsosAufInvarianten} morphisms of (co)modules over a Hopf monoid $H$ induce morphisms between their (co)invariants. The question if morphisms of both, modules and comodules, induce morphisms between the associated biinvariants is more subtle in general. It  is shown in \cite[Lemma 2.10]{MV} that this always holds for {\em isomorphisms}. As a direct generalisation  we have  in the notation of  \eqref{eq:imagedef}

\begin{lemma}\label{Lemma:IsomorphismOfBiinvariants}
Let $\mathcal C$ be complete and finitely cocomplete, $H$, $K$ Hopf monoids in $\mathcal C$ and $\Phi: M \to M'$  an isomorphism of $H$-modules and $K$-comodules. There is a unique morphism $\Phi_{inv}: M_{inv}\to M'_{inv}$ with $\pi'\circ \Phi\circ \iota= I'\circ \Phi_{inv}\circ P$, and $\Phi_{inv}$ is an isomorphism. 
\end{lemma}

\section{Ribbon graphs and surfaces}
\label{sec:ribbon}

In this section we summarise the background on {\em ribbon graphs}, also called {\em fat graphs} or {\em embedded graphs}, for more details we refer to the textbooks of Lando et. al. \cite{LZ} and Ellis-Monaghan and Moffatt \cite{EM}.
Throughout this article, all graphs are  {\em directed}  graphs with a finite number of vertices and edges.  In contrast to \cite{MV} we do not require that the graphs are   connected and allow  {\bf isolated vertices} with no incident edges.

\begin{definition} 
\label{def:ribbongraph} 
 A {\bf ribbon graph} is a graph with a cyclic ordering of the  edge ends at each vertex. 
\end{definition}

The  cyclic ordering of edge ends  at the vertices of a ribbon graph allows one to thicken its edges  to strips or ribbons and defines the faces of  
 the ribbon graph. 
 One says that a path in a ribbon graph {\bf turns maximally left at a vertex} if it  
 enters the vertex along an edge end and leaves it along an edge end that 
  comes directly before it with respect to the cyclic ordering. A {\bf face} of a ribbon graph  is  defined as a cyclic equivalence class of closed paths that turn maximally left at each vertex and traverse each edge at most once in each direction. Each isolated vertex is also viewed as a face, and such a face is called an {\bf isolated face}.
  
  In the following we denote by $V,E,F$ the sets of vertices, edges and faces of a ribbon graph and by $s(\alpha), t(\alpha)$ the starting and target vertex of an edge $\alpha$. We  say that two edge ends incident at a  vertex $v\in V$ are {\bf neighbours} or {\bf neighbouring} if one of them comes directly before or after the other with respect to the cyclic ordering at $v$.
An edge $\alpha$ with $s(\alpha)=t(\alpha)$ is called a {\bf loop}. A loop at $v$ whose starting and target end are neighbours  is called an {\bf isolated loop}. 
When drawing a  ribbon graph we take the cyclic  ordering of  edge ends at  vertices as the one in  the  drawing.

Ribbon graphs are directly related to embedded graphs on oriented surfaces. Every graph $\Gamma$ embedded into an oriented surface $\Sigma$  inherits a cyclic ordering of the edge ends at each vertex and hence a ribbon graph structure.  
 Attaching discs to the faces of the ribbon graph $\Gamma$ yields an oriented surface $\Sigma_\Gamma$ such that the connected components of $\Sigma_\Gamma\setminus \Gamma$ are discs and in bijection with faces of $\Gamma$, see Figure \ref{Ex:GraphTorus}. If $\Gamma$ is embedded into an oriented surface $\Sigma$, the surface $\Sigma_\Gamma$  is homeomorphic to $\Sigma$ iff each connected component of $\Sigma\setminus\Gamma$ is a disc.  In this case, we call $\Gamma$ {\bf properly embedded} in $\Sigma$.  Note that this implies a bijection between connected components of $\Gamma$ and of $\Sigma$, and    connected components of $\Sigma$ containing an isolated vertex are spheres. The genus $g$ of a connected component of $\Sigma$ is then determined by the Euler characteristic  $2-2g=|V|-|E|+|F|$, where $|V|,|E|,|F|$ are the number of vertices, edges and faces of the associated connected component of $\Gamma$.
 
  \begin{figure}
\begin{center}
\tikzset{every picture/.style={line width=0.75pt}}    

\begin{tikzpicture}[x=0.75pt,y=0.75pt,yscale=-1.35,xscale=1.35]
\draw [color={rgb, 255:red, 208; green, 2; blue, 27 }  ,draw opacity=1 ]   (46.5,50.25) .. controls (36.91,14.69) and (119.94,15.69) .. (102.07,53.53) ;
\draw [shift={(101.18,55.29)}, rotate = 298.3] [color={rgb, 255:red, 208; green, 2; blue, 27 }  ,draw opacity=1 ][line width=0.75]    (6.56,-1.97) .. controls (4.17,-0.84) and (1.99,-0.18) .. (0,0) .. controls (1.99,0.18) and (4.17,0.84) .. (6.56,1.97)   ;
\draw [color={rgb, 255:red, 65; green, 117; blue, 5 }  ,draw opacity=1 ]   (43.75,53) .. controls (6.58,16.66) and (87.13,10.98) .. (50.49,50.23) ;
\draw [shift={(49.34,51.44)}, rotate = 314.18] [color={rgb, 255:red, 65; green, 117; blue, 5 }  ,draw opacity=1 ][line width=0.75]    (6.56,-1.97) .. controls (4.17,-0.84) and (1.99,-0.18) .. (0,0) .. controls (1.99,0.18) and (4.17,0.84) .. (6.56,1.97)   ;
\draw [color={rgb, 255:red, 23; green, 19; blue, 254 }  ,draw opacity=1 ][line width=0.75]    (49.26,53) .. controls (61.68,57.26) and (81.93,59.54) .. (96.62,58.22) ;
\draw [shift={(98.43,58.04)}, rotate = 173.5] [color={rgb, 255:red, 23; green, 19; blue, 254 }  ,draw opacity=1 ][line width=0.75]    (6.56,-1.97) .. controls (4.17,-0.84) and (1.99,-0.18) .. (0,0) .. controls (1.99,0.18) and (4.17,0.84) .. (6.56,1.97)   ;
\draw   (171.71,58.15) .. controls (171.71,40.4) and (204.95,26.01) .. (245.95,26.01) .. controls (286.95,26.01) and (320.19,40.4) .. (320.19,58.15) .. controls (320.19,75.9) and (286.95,90.3) .. (245.95,90.3) .. controls (204.95,90.3) and (171.71,75.9) .. (171.71,58.15) -- cycle ;
\draw    (204,55.57) .. controls (228.76,66.01) and (266.76,67.72) .. (286.47,57.44) ;
\draw    (215.04,58.87) .. controls (234.94,48.45) and (260.37,47.02) .. (277.04,60.3) ;
\draw [color={rgb, 255:red, 2; green, 28; blue, 208 }  ,draw opacity=1 ]   (236.22,75.23) .. controls (253.2,76.47) and (260.39,77.36) .. (273.16,73.27) ;
\draw [shift={(275,72.66)}, rotate = 161.11] [color={rgb, 255:red, 2; green, 28; blue, 208 }  ,draw opacity=1 ][line width=0.75]    (6.56,-1.97) .. controls (4.17,-0.84) and (1.99,-0.18) .. (0,0) .. controls (1.99,0.18) and (4.17,0.84) .. (6.56,1.97)   ;
\draw [color={rgb, 255:red, 208; green, 2; blue, 27 }  ,draw opacity=1 ]   (230.71,75.23) .. controls (189.04,69.15) and (172.47,55.44) .. (225.9,38.01) ;
\draw [color={rgb, 255:red, 208; green, 2; blue, 27 }  ,draw opacity=1 ]   (225.9,38.01) .. controls (275.04,25.44) and (310.47,48.87) .. (298.19,65.44) ;
\draw [color={rgb, 255:red, 208; green, 2; blue, 27 }  ,draw opacity=1 ]   (298.19,65.44) .. controls (297.16,67.97) and (291.1,70.63) .. (282.36,72.32) ;
\draw [shift={(280.5,72.66)}, rotate = 350.33] [color={rgb, 255:red, 208; green, 2; blue, 27 }  ,draw opacity=1 ][line width=0.75]    (6.56,-1.97) .. controls (4.17,-0.84) and (1.99,-0.18) .. (0,0) .. controls (1.99,0.18) and (4.17,0.84) .. (6.56,1.97)   ;
\draw [color={rgb, 255:red, 65; green, 117; blue, 5 }  ,draw opacity=1 ]   (224.29,61.86) .. controls (234.04,61.12) and (234.13,66.3) .. (233.69,70.53) ;
\draw [shift={(233.47,72.48)}, rotate = 275.84] [color={rgb, 255:red, 65; green, 117; blue, 5 }  ,draw opacity=1 ][line width=0.75]    (6.56,-1.97) .. controls (4.17,-0.84) and (1.99,-0.18) .. (0,0) .. controls (1.99,0.18) and (4.17,0.84) .. (6.56,1.97)   ;
\draw [color={rgb, 255:red, 65; green, 117; blue, 5 }  ,draw opacity=1 ]   (216.77,87.58) .. controls (228.2,86.72) and (232.2,83.58) .. (233.47,77.99) ;
\draw [color={rgb, 255:red, 65; green, 117; blue, 5 }  ,draw opacity=1 ] [dash pattern={on 0.84pt off 2.51pt}]  (216.77,87.58) .. controls (210.48,85.86) and (207.62,69) .. (224.29,61.86) ;
\draw [color={rgb, 255:red, 245; green, 166; blue, 35 }  ,draw opacity=0.6 ]   (36,53) .. controls (7.91,20.15) and (74.77,-3.85) .. (64.48,42.15) ;
\draw [color={rgb, 255:red, 245; green, 166; blue, 35 }  ,draw opacity=0.6 ]   (64.48,42.15) .. controls (58.77,59.86) and (105.62,56.15) .. (95.62,39.58) ;
\draw [color={rgb, 255:red, 245; green, 166; blue, 35 }  ,draw opacity=0.6 ]   (49.05,41) .. controls (56.2,19.86) and (91.05,31.29) .. (95.62,39.58) ;
\draw [color={rgb, 255:red, 245; green, 166; blue, 35 }  ,draw opacity=0.6 ]   (49.05,41) .. controls (43.05,57.86) and (66.77,28.43) .. (49.34,27.29) ;
\draw [color={rgb, 255:red, 245; green, 166; blue, 35 }  ,draw opacity=0.6 ]   (49.34,27.29) .. controls (33.91,27.86) and (41.62,56.72) .. (44.2,38.15) ;
\draw [color={rgb, 255:red, 245; green, 166; blue, 35 }  ,draw opacity=0.6 ]   (44.2,38.15) .. controls (46.48,9.58) and (104.48,14.72) .. (109.62,35) ;
\draw [color={rgb, 255:red, 245; green, 166; blue, 35 }  ,draw opacity=0.6 ]   (109.62,35) .. controls (113.02,47.73) and (121.17,79.5) .. (64.78,62.12) ;
\draw [shift={(63.05,61.58)}, rotate = 17.67] [color={rgb, 255:red, 245; green, 166; blue, 35 }  ,draw opacity=0.6 ][line width=0.75]    (6.56,-1.97) .. controls (4.17,-0.84) and (1.99,-0.18) .. (0,0) .. controls (1.99,0.18) and (4.17,0.84) .. (6.56,1.97)   ;
\draw  [fill={rgb, 255:red, 0; green, 0; blue, 0 }  ,fill opacity=1 ] (43.75,53) .. controls (43.75,51.48) and (44.98,50.25) .. (46.5,50.25) .. controls (48.02,50.25) and (49.26,51.48) .. (49.26,53) .. controls (49.26,54.52) and (48.02,55.75) .. (46.5,55.75) .. controls (44.98,55.75) and (43.75,54.52) .. (43.75,53) -- cycle ;
\draw  [fill={rgb, 255:red, 0; green, 0; blue, 0 }  ,fill opacity=1 ] (98.43,58.04) .. controls (98.43,56.52) and (99.66,55.29) .. (101.18,55.29) .. controls (102.7,55.29) and (103.93,56.52) .. (103.93,58.04) .. controls (103.93,59.56) and (102.7,60.79) .. (101.18,60.79) .. controls (99.66,60.79) and (98.43,59.56) .. (98.43,58.04) -- cycle ;
\draw  [fill={rgb, 255:red, 0; green, 0; blue, 0 }  ,fill opacity=1 ] (230.71,75.23) .. controls (230.71,73.71) and (231.95,72.48) .. (233.47,72.48) .. controls (234.99,72.48) and (236.22,73.71) .. (236.22,75.23) .. controls (236.22,76.75) and (234.99,77.99) .. (233.47,77.99) .. controls (231.95,77.99) and (230.71,76.75) .. (230.71,75.23) -- cycle ;
\draw  [fill={rgb, 255:red, 0; green, 0; blue, 0 }  ,fill opacity=1 ] (275,72.66) .. controls (275,71.14) and (276.23,69.91) .. (277.75,69.91) .. controls (279.27,69.91) and (280.5,71.14) .. (280.5,72.66) .. controls (280.5,74.18) and (279.27,75.41) .. (277.75,75.41) .. controls (276.23,75.41) and (275,74.18) .. (275,72.66) -- cycle ;
\draw (110.57,65.7) node [anchor=north west][inner sep=0.75pt]  [font=\scriptsize]  {$\textcolor[rgb]{0.96,0.65,0.14}{f}$};
\end{tikzpicture}
\end{center}
\caption{
Attaching a disc to the face $f$ yields a torus.}
\label{Ex:GraphTorus}
\end{figure}
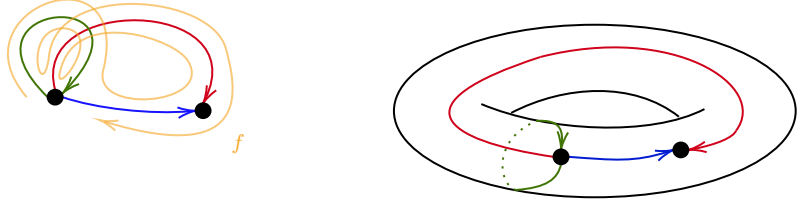

Note that each ribbon graph or embedded graph has a Poincar\'e dual obtained by replacing each vertex  (face) with a face (vertex) and each edge with a dual edge. 
This transforms the paths that characterise faces into paths that go counterclockwise around a vertex and vice versa. Edge ends  correspond to edge sides of the dual graph  and their
 cyclic ordering  at a vertex  to the cyclic ordering of the edge sides  in the dual face.

In the following we sometimes require a {\em linear} ordering of the edge ends at a vertex or of the edge sides in a face. This is achieved by inserting a marking, the {\em cilium},   that separates the edge ends or edge sides of minimal and maximal order, see for instance Figure \ref{fig:graphtrafos}, Definition \ref{def:standardgraph} or Example \ref{example:VertexAndFaceOperators}. For faces this corresponds to the choice of a starting vertex for the associated cyclic equivalence class of paths.

\begin{definition}\label{def:ciliated} $\quad$
\begin{compactenum}
\item A {\bf ciliated vertex} in a ribbon graph is a vertex with a choice of linear ordering of the incident edge ends that is compatible with their cyclic ordering. 
\item A  {\bf ciliated face} in a ribbon graph is a closed path that turns maximally left at each vertex,  including the starting vertex, and traverses  each edge at most once in each direction.
\end{compactenum}
A {\bf ciliated ribbon graph} is a ribbon graph in which each face and vertex is assigned a cilium. Isolated vertices and faces are trivially ciliated.
\end{definition}

For a closed surface $\Sigma$ of genus $g\geq 0$ we often work with  a ciliated  ribbon graph with a single vertex and  a single face that is given by a set of generators of the fundamental group
\begin{align}\label{eq:pi1gensurface}
\pi_1(\Sigma)=\langle \alpha_1,\beta_1,\ldots,\alpha_g,\beta_g\mid [\beta_g^\inv,\alpha_g]\cdots[\beta_1^\inv,\alpha_1]=1\rangle.
\end{align}

\begin{definition}\label{def:standardgraph} The {\bf standard graph} of an oriented surface $\Sigma$ of genus $g\geq 1$ is the  graph 
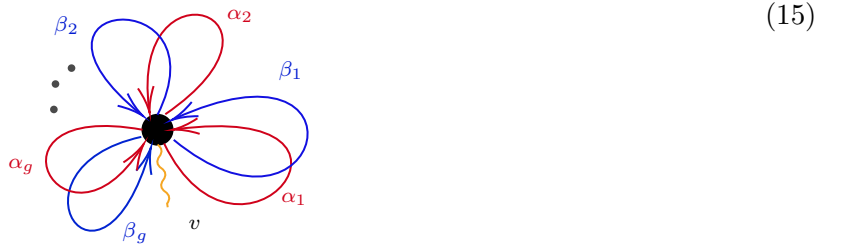
\begin{align}\label{eq:standardgraph}
\\[-7ex]
\tikzset{every picture/.style={line width=0.75pt}}        
\begin{tikzpicture}[x=0.75pt,y=0.75pt,yscale=-1.2,xscale=1.2]
\draw  [fill={rgb, 255:red, 0; green, 0; blue, 0 }  ,fill opacity=1 ] (74,64.75) .. controls (74,61.3) and (76.8,58.5) .. (80.25,58.5) .. controls (83.7,58.5) and (86.5,61.3) .. (86.5,64.75) .. controls (86.5,68.2) and (83.7,71) .. (80.25,71) .. controls (76.8,71) and (74,68.2) .. (74,64.75) -- cycle ;
\draw [color={rgb, 255:red, 208; green, 2; blue, 27 }  ,draw opacity=1 ]   (83,70.5) .. controls (114.84,136.42) and (181.33,52.1) .. (87.92,64.55) ;
\draw [shift={(86.5,64.75)}, rotate = 351.95] [color={rgb, 255:red, 208; green, 2; blue, 27 }  ,draw opacity=1 ][line width=0.75]    (10.93,-3.29) .. controls (6.95,-1.4) and (3.31,-0.3) .. (0,0) .. controls (3.31,0.3) and (6.95,1.4) .. (10.93,3.29)   ;
\draw [color={rgb, 255:red, 18; green, 16; blue, 224 }  ,draw opacity=1 ]   (87,69) .. controls (149.69,120.99) and (171.78,21.26) .. (87.28,60.15) ;
\draw [shift={(86,60.75)}, rotate = 334.78] [color={rgb, 255:red, 18; green, 16; blue, 224 }  ,draw opacity=1 ][line width=0.75]    (10.93,-3.29) .. controls (6.95,-1.4) and (3.31,-0.3) .. (0,0) .. controls (3.31,0.3) and (6.95,1.4) .. (10.93,3.29)   ;
\draw [color={rgb, 255:red, 208; green, 2; blue, 27 }  ,draw opacity=1 ]   (83.5,58.25) .. controls (144.39,15.93) and (68.55,-8.26) .. (76.73,55.78) ;
\draw [shift={(77,57.75)}, rotate = 261.45] [color={rgb, 255:red, 208; green, 2; blue, 27 }  ,draw opacity=1 ][line width=0.75]    (10.93,-3.29) .. controls (6.95,-1.4) and (3.31,-0.3) .. (0,0) .. controls (3.31,0.3) and (6.95,1.4) .. (10.93,3.29)   ;
\draw [color={rgb, 255:red, 18; green, 16; blue, 224 }  ,draw opacity=1 ]   (80.25,58.5) .. controls (108.36,2.03) and (12.96,8.68) .. (73.33,58.25) ;
\draw [shift={(74.25,59)}, rotate = 218.91] [color={rgb, 255:red, 18; green, 16; blue, 224 }  ,draw opacity=1 ][line width=0.75]    (10.93,-3.29) .. controls (6.95,-1.4) and (3.31,-0.3) .. (0,0) .. controls (3.31,0.3) and (6.95,1.4) .. (10.93,3.29)   ;
\draw [color={rgb, 255:red, 245; green, 166; blue, 35 }  ,draw opacity=1 ]   (80.25,71) .. controls (82.16,72.38) and (82.43,74.03) .. (81.05,75.94) .. controls (79.67,77.85) and (79.94,79.5) .. (81.85,80.87) .. controls (83.76,82.25) and (84.03,83.9) .. (82.65,85.81) .. controls (81.27,87.72) and (81.54,89.37) .. (83.45,90.74) .. controls (85.36,92.12) and (85.63,93.77) .. (84.25,95.68) -- (84.5,97.25) -- (84.5,97.25) ;
\draw [color={rgb, 255:red, 2; green, 36; blue, 208 }  ,draw opacity=1 ]   (73.5,67.75) .. controls (14.05,80.93) and (56.32,145.85) .. (77.19,73.35) ;
\draw [shift={(77.5,72.25)}, rotate = 105.61] [color={rgb, 255:red, 2; green, 36; blue, 208 }  ,draw opacity=1 ][line width=0.75]    (10.93,-3.29) .. controls (6.95,-1.4) and (3.31,-0.3) .. (0,0) .. controls (3.31,0.3) and (6.95,1.4) .. (10.93,3.29)   ;
\draw [color={rgb, 255:red, 208; green, 2; blue, 27 }  ,draw opacity=1 ]   (74,64.75) .. controls (3.21,51.88) and (41.22,123.79) .. (73.52,72.35) ;
\draw [shift={(74.5,70.75)}, rotate = 120.81] [color={rgb, 255:red, 208; green, 2; blue, 27 }  ,draw opacity=1 ][line width=0.75]    (10.93,-3.29) .. controls (6.95,-1.4) and (3.31,-0.3) .. (0,0) .. controls (3.31,0.3) and (6.95,1.4) .. (10.93,3.29)   ;
\draw  [color={rgb, 255:red, 74; green, 74; blue, 74 }  ,draw opacity=1 ][fill={rgb, 255:red, 74; green, 74; blue, 74 }  ,fill opacity=1 ] (45.67,38.92) .. controls (45.67,38.27) and (45.14,37.75) .. (44.5,37.75) .. controls (43.86,37.75) and (43.33,38.27) .. (43.33,38.92) .. controls (43.33,39.56) and (43.86,40.08) .. (44.5,40.08) .. controls (45.14,40.08) and (45.67,39.56) .. (45.67,38.92) -- cycle ;
\draw  [color={rgb, 255:red, 74; green, 74; blue, 74 }  ,draw opacity=1 ][fill={rgb, 255:red, 74; green, 74; blue, 74 }  ,fill opacity=1 ] (39,45.58) .. controls (39,44.94) and (38.48,44.42) .. (37.83,44.42) .. controls (37.19,44.42) and (36.67,44.94) .. (36.67,45.58) .. controls (36.67,46.23) and (37.19,46.75) .. (37.83,46.75) .. controls (38.48,46.75) and (39,46.23) .. (39,45.58) -- cycle ;
\draw  [color={rgb, 255:red, 74; green, 74; blue, 74 }  ,draw opacity=1 ][fill={rgb, 255:red, 74; green, 74; blue, 74 }  ,fill opacity=1 ] (38.33,56.25) .. controls (38.33,55.61) and (37.81,55.08) .. (37.17,55.08) .. controls (36.52,55.08) and (36,55.61) .. (36,56.25) .. controls (36,56.89) and (36.52,57.42) .. (37.17,57.42) .. controls (37.81,57.42) and (38.33,56.89) .. (38.33,56.25) -- cycle ;
\draw (130.5,89.9) node [anchor=north west][inner sep=0.75pt]  [font=\scriptsize]  {$\textcolor[rgb]{0.82,0.01,0.11}{\alpha _{1}}$};
\draw (108,13.4) node [anchor=north west][inner sep=0.75pt]  [font=\scriptsize]  {$\textcolor[rgb]{0.82,0.01,0.11}{\alpha }\textcolor[rgb]{0.82,0.01,0.11}{_{2}}$};
\draw (17,75.9) node [anchor=north west][inner sep=0.75pt]  [font=\scriptsize]  {$\textcolor[rgb]{0.82,0.01,0.11}{\alpha }\textcolor[rgb]{0.82,0.01,0.11}{_{g}}$};
\draw (129.5,34.9) node [anchor=north west][inner sep=0.75pt]  [font=\scriptsize,color={rgb, 255:red, 0; green, 0; blue, 0 }  ,opacity=1 ]  {$\textcolor[rgb]{0.01,0.11,0.82}{\beta _{1}}$};
\draw (36,15.9) node [anchor=north west][inner sep=0.75pt]  [font=\scriptsize,color={rgb, 255:red, 0; green, 0; blue, 0 }  ,opacity=1 ]  {$\textcolor[rgb]{0.01,0.11,0.82}{\beta }\textcolor[rgb]{0.01,0.11,0.82}{_{2}}$};
\draw (64.5,102.4) node [anchor=north west][inner sep=0.75pt]  [font=\scriptsize,color={rgb, 255:red, 0; green, 0; blue, 0 }  ,opacity=1 ]  {$\textcolor[rgb]{0.01,0.11,0.82}{\beta }\textcolor[rgb]{0.01,0.11,0.82}{_{g}}$};
\draw (92,101.23) node [anchor=north west][inner sep=0.75pt]  [font=\scriptsize]  {$v$};
\end{tikzpicture}\nonumber\\[-10ex]
\nonumber
\end{align}
with the face $f=[\beta_g^\inv, \alpha_g]\cdots[\beta_1^\inv,\alpha_1]$ and the ordering of edge ends at $v$ given by\\ $s(\alpha_1)<s(\beta_1)<t(\alpha_1)<t(\beta_1)<\ldots< s(\alpha_g)<s(\beta_g)<t(\alpha_g)<t(\beta_g)$. In particular, the standard graph for $S^2$ consists of a single isolated vertex and the associated isolated face.
\end{definition}

In the following we use certain graph transformations  to relate properly embedded ribbon graphs in a connected surface $\Sigma$ to its standard graph.

\begin{definition}\label{def:graphtrafos} Let $\Gamma$ be a ribbon graph with edge set $E$ and vertex set $V$.
\begin{compactenum}
\item The {\bf edge reversal}  reverses the orientation of an edge $\beta\in E$.

\item The  {\bf  contraction} of an edge $\alpha\in E$ that is not a loop  removes $\alpha\in E$ and fuses the vertices $s(\alpha)$ and $t(\alpha)$. 
\item 
The  {\bf edge slide}  slides an end of $\beta\in E$ that is a neighbour of an end of $\alpha\in E$ along $\alpha$. 
\item The {\bf loop deletion} removes an isolated loop $\beta\in E$ from $\Gamma$.
\end{compactenum}
In all cases except 2.~the resulting ribbon graph inherits all cilia from $\Gamma$. In 2.~one erases either the cilium of $t(\alpha)$ or  of $s(\alpha)$ and speaks of contracting $\alpha$ towards $t(\alpha)$ and $s(\alpha)$, respectively.
\end{definition}

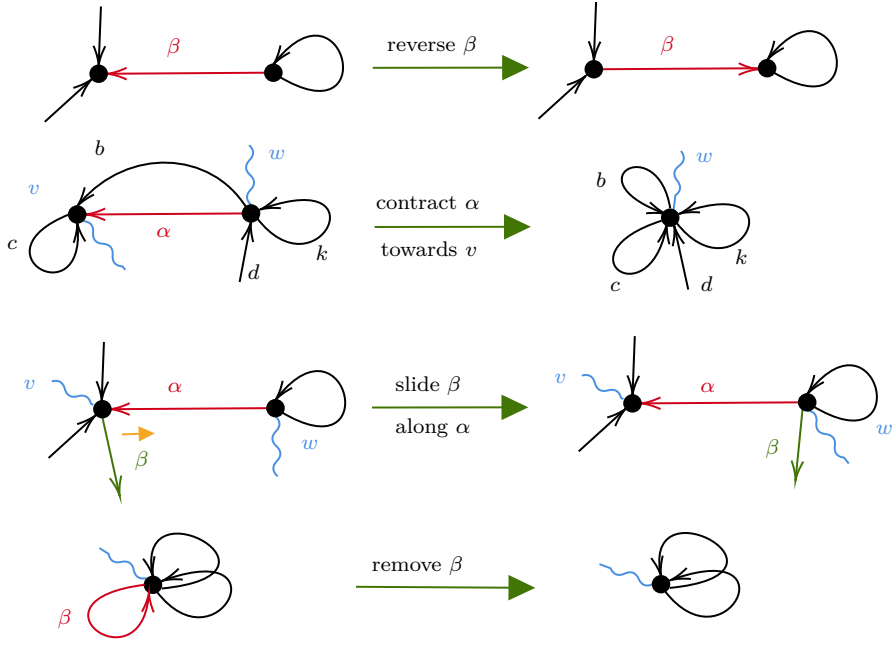
\begin{figure}
\begin{center}
\tikzset{every picture/.style={line width=0.75pt}}     
\begin{tikzpicture}[x=0.75pt,y=0.75pt,yscale=-1.5,xscale=1.5]
\draw [color={rgb, 255:red, 65; green, 117; blue, 5 }  ,draw opacity=1 ]   (128.35,35.76) -- (177.32,35.69) ;
\draw [shift={(180.32,35.69)}, rotate = 179.92] [fill={rgb, 255:red, 65; green, 117; blue, 5 }  ,fill opacity=1 ][line width=0.08]  [draw opacity=0] (8.93,-4.29) -- (0,0) -- (8.93,4.29) -- cycle    ;
\draw [color={rgb, 255:red, 208; green, 2; blue, 27 }  ,draw opacity=1 ]   (92.53,37.4) -- (41.93,37.74) ;
\draw [shift={(39.93,37.75)}, rotate = 359.62] [color={rgb, 255:red, 208; green, 2; blue, 27 }  ,draw opacity=1 ][line width=0.75]    (6.56,-1.97) .. controls (4.17,-0.84) and (1.99,-0.18) .. (0,0) .. controls (1.99,0.18) and (4.17,0.84) .. (6.56,1.97)   ;
\draw    (19.28,53.86) -- (33.47,40.87) ;
\draw [shift={(34.94,39.52)}, rotate = 137.54] [color={rgb, 255:red, 0; green, 0; blue, 0 }  ][line width=0.75]    (6.56,-1.97) .. controls (4.17,-0.84) and (1.99,-0.18) .. (0,0) .. controls (1.99,0.18) and (4.17,0.84) .. (6.56,1.97)   ;
\draw    (37.94,15.19) -- (37.26,33) ;
\draw [shift={(37.18,35)}, rotate = 272.2] [color={rgb, 255:red, 0; green, 0; blue, 0 }  ][line width=0.75]    (6.56,-1.97) .. controls (4.17,-0.84) and (1.99,-0.18) .. (0,0) .. controls (1.99,0.18) and (4.17,0.84) .. (6.56,1.97)   ;
\draw    (97.08,38.61) .. controls (131.42,59.54) and (119.33,5) .. (96.34,33.3) ;
\draw [shift={(95.29,34.65)}, rotate = 307.06] [color={rgb, 255:red, 0; green, 0; blue, 0 }  ][line width=0.75]    (6.56,-1.97) .. controls (4.17,-0.84) and (1.99,-0.18) .. (0,0) .. controls (1.99,0.18) and (4.17,0.84) .. (6.56,1.97)   ;
\draw  [fill={rgb, 255:red, 0; green, 0; blue, 0 }  ,fill opacity=1 ] (34.43,37.75) .. controls (34.43,36.23) and (35.66,35) .. (37.18,35) .. controls (38.7,35) and (39.93,36.23) .. (39.93,37.75) .. controls (39.93,39.27) and (38.7,40.5) .. (37.18,40.5) .. controls (35.66,40.5) and (34.43,39.27) .. (34.43,37.75) -- cycle ;
\draw  [fill={rgb, 255:red, 0; green, 0; blue, 0 }  ,fill opacity=1 ] (92.53,37.4) .. controls (92.53,35.88) and (93.77,34.65) .. (95.29,34.65) .. controls (96.81,34.65) and (98.04,35.88) .. (98.04,37.4) .. controls (98.04,38.92) and (96.81,40.16) .. (95.29,40.16) .. controls (93.77,40.16) and (92.53,38.92) .. (92.53,37.4) -- cycle ;
\draw [color={rgb, 255:red, 208; green, 2; blue, 27 }  ,draw opacity=1 ]   (257.23,35.9) -- (204.63,36.23) ;
\begin{scope}[shift={(52, 0)}]
\draw [shift={(204.63,36.24)}, rotate = 180.62] [color={rgb, 255:red, 208; green, 2; blue, 27 }  ,draw opacity=1 ][line width=0.75]    (6.56,-1.97) .. controls (4.17,-0.84) and (1.99,-0.18) .. (0,0) .. controls (1.99,0.18) and (4.17,0.84) .. (6.56,1.97)   ;
\end{scope}
\draw    (183.48,52.85) -- (197.67,39.86) ;
\draw [shift={(199.14,38.51)}, rotate = 137.54] [color={rgb, 255:red, 0; green, 0; blue, 0 }  ][line width=0.75]    (6.56,-1.97) .. controls (4.17,-0.84) and (1.99,-0.18) .. (0,0) .. controls (1.99,0.18) and (4.17,0.84) .. (6.56,1.97)   ;
\draw    (202.64,13.68) -- (201.96,31.49) ;
\draw [shift={(201.88,33.49)}, rotate = 272.2] [color={rgb, 255:red, 0; green, 0; blue, 0 }  ][line width=0.75]    (6.56,-1.97) .. controls (4.17,-0.84) and (1.99,-0.18) .. (0,0) .. controls (1.99,0.18) and (4.17,0.84) .. (6.56,1.97)   ;
\draw    (261.33,37.1) .. controls (295.67,58.03) and (283.59,3.49) .. (260.6,31.79) ;
\draw [shift={(259.54,33.14)}, rotate = 307.06] [color={rgb, 255:red, 0; green, 0; blue, 0 }  ][line width=0.75]    (6.56,-1.97) .. controls (4.17,-0.84) and (1.99,-0.18) .. (0,0) .. controls (1.99,0.18) and (4.17,0.84) .. (6.56,1.97)   ;
\draw  [fill={rgb, 255:red, 0; green, 0; blue, 0 }  ,fill opacity=1 ] (199.13,36.24) .. controls (199.13,34.72) and (200.36,33.49) .. (201.88,33.49) .. controls (203.4,33.49) and (204.63,34.72) .. (204.63,36.24) .. controls (204.63,37.76) and (203.4,39) .. (201.88,39) .. controls (200.36,39) and (199.13,37.76) .. (199.13,36.24) -- cycle ;
\draw  [fill={rgb, 255:red, 0; green, 0; blue, 0 }  ,fill opacity=1 ] (256.79,35.9) .. controls (256.79,34.38) and (258.02,33.14) .. (259.54,33.14) .. controls (261.06,33.14) and (262.29,34.38) .. (262.29,35.9) .. controls (262.29,37.42) and (261.06,38.65) .. (259.54,38.65) .. controls (258.02,38.65) and (256.79,37.42) .. (256.79,35.9) -- cycle ;
\draw [color={rgb, 255:red, 65; green, 117; blue, 5 }  ,draw opacity=1 ]   (128.13,149.69) -- (177.1,149.62) ;
\draw [shift={(180.1,149.62)}, rotate = 179.92] [fill={rgb, 255:red, 65; green, 117; blue, 5 }  ,fill opacity=1 ][line width=0.08]  [draw opacity=0] (8.93,-4.29) -- (0,0) -- (8.93,4.29) -- cycle    ;
\draw [color={rgb, 255:red, 74; green, 144; blue, 226 }  ,draw opacity=1 ]   (96,152.66) .. controls (97.71,154.27) and (97.76,155.94) .. (96.15,157.65) .. controls (94.54,159.37) and (94.59,161.04) .. (96.31,162.65) .. controls (98.02,164.27) and (98.07,165.94) .. (96.46,167.65) .. controls (94.85,169.36) and (94.9,171.03) .. (96.61,172.65) -- (96.62,172.82) -- (96.62,172.82) ;
\draw [color={rgb, 255:red, 208; green, 2; blue, 27 }  ,draw opacity=1 ]   (93.53,149.9) -- (42.93,150.24) ;
\draw [shift={(40.93,150.25)}, rotate = 359.62] [color={rgb, 255:red, 208; green, 2; blue, 27 }  ,draw opacity=1 ][line width=0.75]    (6.56,-1.97) .. controls (4.17,-0.84) and (1.99,-0.18) .. (0,0) .. controls (1.99,0.18) and (4.17,0.84) .. (6.56,1.97)   ;
\draw [color={rgb, 255:red, 74; green, 144; blue, 226 }  ,draw opacity=1 ]   (21.57,140.93) .. controls (23.85,140.31) and (25.3,141.14) .. (25.92,143.41) .. controls (26.54,145.68) and (27.99,146.51) .. (30.26,145.89) .. controls (32.53,145.27) and (33.98,146.1) .. (34.6,148.37) -- (35.39,148.82) -- (35.39,148.82) ;
\draw [color={rgb, 255:red, 65; green, 117; blue, 5 }  ,draw opacity=1 ]   (38.43,153) -- (43.77,177.4) ;
\draw [shift={(44.19,179.36)}, rotate = 257.66] [color={rgb, 255:red, 65; green, 117; blue, 5 }  ,draw opacity=1 ][line width=0.75]    (6.56,-1.97) .. controls (4.17,-0.84) and (1.99,-0.18) .. (0,0) .. controls (1.99,0.18) and (4.17,0.84) .. (6.56,1.97)   ;
\draw    (20.53,166.61) -- (34.72,153.62) ;
\draw [shift={(36.19,152.27)}, rotate = 137.54] [color={rgb, 255:red, 0; green, 0; blue, 0 }  ][line width=0.75]    (6.56,-1.97) .. controls (4.17,-0.84) and (1.99,-0.18) .. (0,0) .. controls (1.99,0.18) and (4.17,0.84) .. (6.56,1.97)   ;
\draw    (38.94,127.69) -- (38.26,145.5) ;
\draw [shift={(38.18,147.5)}, rotate = 272.2] [color={rgb, 255:red, 0; green, 0; blue, 0 }  ][line width=0.75]    (6.56,-1.97) .. controls (4.17,-0.84) and (1.99,-0.18) .. (0,0) .. controls (1.99,0.18) and (4.17,0.84) .. (6.56,1.97)   ;
\draw [color={rgb, 255:red, 245; green, 166; blue, 35 }  ,draw opacity=1 ]   (44.93,158.7) -- (52.86,158.45) ;
\draw [shift={(55.86,158.36)}, rotate = 178.2] [fill={rgb, 255:red, 245; green, 166; blue, 35 }  ,fill opacity=1 ][line width=0.08]  [draw opacity=0] (5.36,-2.57) -- (0,0) -- (5.36,2.57) -- cycle    ;
\draw    (97.79,151.11) .. controls (132.13,172.04) and (120.05,117.5) .. (97.06,145.8) ;
\draw [shift={(96,147.15)}, rotate = 307.06] [color={rgb, 255:red, 0; green, 0; blue, 0 }  ][line width=0.75]    (6.56,-1.97) .. controls (4.17,-0.84) and (1.99,-0.18) .. (0,0) .. controls (1.99,0.18) and (4.17,0.84) .. (6.56,1.97)   ;
\draw  [fill={rgb, 255:red, 0; green, 0; blue, 0 }  ,fill opacity=1 ] (35.68,150.25) .. controls (35.68,148.73) and (36.91,147.5) .. (38.43,147.5) .. controls (39.95,147.5) and (41.18,148.73) .. (41.18,150.25) .. controls (41.18,151.77) and (39.95,153) .. (38.43,153) .. controls (36.91,153) and (35.68,151.77) .. (35.68,150.25) -- cycle ;
\draw  [fill={rgb, 255:red, 0; green, 0; blue, 0 }  ,fill opacity=1 ] (93.25,149.9) .. controls (93.25,148.38) and (94.48,147.15) .. (96,147.15) .. controls (97.52,147.15) and (98.75,148.38) .. (98.75,149.9) .. controls (98.75,151.42) and (97.52,152.66) .. (96,152.66) .. controls (94.48,152.66) and (93.25,151.42) .. (93.25,149.9) -- cycle ;
\draw [color={rgb, 255:red, 74; green, 144; blue, 226 }  ,draw opacity=1 ]   (272.86,150.7) .. controls (275.2,150.99) and (276.22,152.31) .. (275.93,154.65) .. controls (275.64,156.99) and (276.66,158.31) .. (279,158.6) .. controls (281.34,158.89) and (282.36,160.21) .. (282.06,162.55) .. controls (281.77,164.89) and (282.79,166.2) .. (285.13,166.49) -- (286.68,168.48) -- (286.68,168.48) ;
\draw [color={rgb, 255:red, 208; green, 2; blue, 27 }  ,draw opacity=1 ]   (270.1,147.95) -- (219.5,148.29) ;
\draw [shift={(217.5,148.3)}, rotate = 359.62] [color={rgb, 255:red, 208; green, 2; blue, 27 }  ,draw opacity=1 ][line width=0.75]    (6.56,-1.97) .. controls (4.17,-0.84) and (1.99,-0.18) .. (0,0) .. controls (1.99,0.18) and (4.17,0.84) .. (6.56,1.97)   ;
\draw [color={rgb, 255:red, 74; green, 144; blue, 226 }  ,draw opacity=1 ]   (197.9,138.97) .. controls (200.17,138.36) and (201.62,139.19) .. (202.24,141.46) .. controls (202.86,143.73) and (204.31,144.56) .. (206.58,143.94) .. controls (208.85,143.32) and (210.3,144.15) .. (210.92,146.42) -- (211.71,146.87) -- (211.71,146.87) ;
\draw [color={rgb, 255:red, 65; green, 117; blue, 5 }  ,draw opacity=1 ]   (271.32,150.2) -- (269.16,172.2) ;
\draw [shift={(268.96,174.19)}, rotate = 275.62] [color={rgb, 255:red, 65; green, 117; blue, 5 }  ,draw opacity=1 ][line width=0.75]    (6.56,-1.97) .. controls (4.17,-0.84) and (1.99,-0.18) .. (0,0) .. controls (1.99,0.18) and (4.17,0.84) .. (6.56,1.97)   ;
\draw    (196.85,164.4) -- (211.04,151.42) ;
\draw [shift={(212.52,150.07)}, rotate = 137.54] [color={rgb, 255:red, 0; green, 0; blue, 0 }  ][line width=0.75]    (6.56,-1.97) .. controls (4.17,-0.84) and (1.99,-0.18) .. (0,0) .. controls (1.99,0.18) and (4.17,0.84) .. (6.56,1.97)   ;
\draw    (215.52,125.74) -- (214.83,143.55) ;
\draw [shift={(214.75,145.55)}, rotate = 272.2] [color={rgb, 255:red, 0; green, 0; blue, 0 }  ][line width=0.75]    (6.56,-1.97) .. controls (4.17,-0.84) and (1.99,-0.18) .. (0,0) .. controls (1.99,0.18) and (4.17,0.84) .. (6.56,1.97)   ;
\draw    (274.65,149.16) .. controls (308.99,170.08) and (296.9,115.55) .. (273.91,143.85) ;
\draw [shift={(272.86,145.2)}, rotate = 307.06] [color={rgb, 255:red, 0; green, 0; blue, 0 }  ][line width=0.75]    (6.56,-1.97) .. controls (4.17,-0.84) and (1.99,-0.18) .. (0,0) .. controls (1.99,0.18) and (4.17,0.84) .. (6.56,1.97)   ;
\draw  [fill={rgb, 255:red, 0; green, 0; blue, 0 }  ,fill opacity=1 ] (212,148.3) .. controls (212,146.78) and (213.23,145.55) .. (214.75,145.55) .. controls (216.27,145.55) and (217.5,146.78) .. (217.5,148.3) .. controls (217.5,149.82) and (216.27,151.05) .. (214.75,151.05) .. controls (213.23,151.05) and (212,149.82) .. (212,148.3) -- cycle ;
\draw  [fill={rgb, 255:red, 0; green, 0; blue, 0 }  ,fill opacity=1 ] (270.1,147.95) .. controls (270.1,146.43) and (271.34,145.2) .. (272.86,145.2) .. controls (274.38,145.2) and (275.61,146.43) .. (275.61,147.95) .. controls (275.61,149.47) and (274.38,150.7) .. (272.86,150.7) .. controls (271.34,150.7) and (270.1,149.47) .. (270.1,147.95) -- cycle ;
\draw [color={rgb, 255:red, 74; green, 144; blue, 226 }  ,draw opacity=1 ]   (87.24,61.69) .. controls (88.95,63.31) and (89,64.98) .. (87.39,66.69) .. controls (85.78,68.41) and (85.83,70.08) .. (87.55,71.69) .. controls (89.26,73.3) and (89.31,74.97) .. (87.7,76.68) .. controls (86.09,78.4) and (86.14,80.07) .. (87.86,81.68) -- (87.86,81.85) -- (87.86,81.85) ;
\draw [color={rgb, 255:red, 208; green, 2; blue, 27 }  ,draw opacity=1 ]   (85.39,84.6) -- (34.79,84.94) ;
\draw [shift={(32.79,84.95)}, rotate = 359.62] [color={rgb, 255:red, 208; green, 2; blue, 27 }  ,draw opacity=1 ][line width=0.75]    (6.56,-1.97) .. controls (4.17,-0.84) and (1.99,-0.18) .. (0,0) .. controls (1.99,0.18) and (4.17,0.84) .. (6.56,1.97)   ;
\draw [color={rgb, 255:red, 74; green, 144; blue, 226 }  ,draw opacity=1 ]   (32.07,87.18) .. controls (34.42,87.36) and (35.51,88.63) .. (35.33,90.98) .. controls (35.15,93.33) and (36.23,94.59) .. (38.58,94.77) .. controls (40.93,94.95) and (42.02,96.22) .. (41.84,98.57) .. controls (41.66,100.92) and (42.75,102.19) .. (45.1,102.36) -- (46.07,103.5) -- (46.07,103.5) ;
\draw    (89.84,86.16) .. controls (109.95,112.34) and (131.18,69.57) .. (92.75,82.64) ;
\draw [shift={(90.96,83.27)}, rotate = 339.82] [color={rgb, 255:red, 0; green, 0; blue, 0 }  ][line width=0.75]    (6.56,-1.97) .. controls (4.17,-0.84) and (1.99,-0.18) .. (0,0) .. controls (1.99,0.18) and (4.17,0.84) .. (6.56,1.97)   ;
\draw  [fill={rgb, 255:red, 0; green, 0; blue, 0 }  ,fill opacity=1 ] (27.29,84.95) .. controls (27.29,83.43) and (28.52,82.2) .. (30.04,82.2) .. controls (31.56,82.2) and (32.79,83.43) .. (32.79,84.95) .. controls (32.79,86.47) and (31.56,87.7) .. (30.04,87.7) .. controls (28.52,87.7) and (27.29,86.47) .. (27.29,84.95) -- cycle ;
\draw  [fill={rgb, 255:red, 0; green, 0; blue, 0 }  ,fill opacity=1 ] (85.11,84.6) .. controls (85.11,83.08) and (86.34,81.85) .. (87.86,81.85) .. controls (89.38,81.85) and (90.61,83.08) .. (90.61,84.6) .. controls (90.61,86.12) and (89.38,87.36) .. (87.86,87.36) .. controls (86.34,87.36) and (85.11,86.12) .. (85.11,84.6) -- cycle ;
\draw    (27.29,84.95) .. controls (-4.62,97.9) and (32.31,118.86) .. (30.22,89.56) ;
\draw [shift={(30.04,87.7)}, rotate = 83.23] [color={rgb, 255:red, 0; green, 0; blue, 0 }  ][line width=0.75]    (6.56,-1.97) .. controls (4.17,-0.84) and (1.99,-0.18) .. (0,0) .. controls (1.99,0.18) and (4.17,0.84) .. (6.56,1.97)   ;
\draw    (86.29,82.61) .. controls (74.37,63.32) and (47.03,62.42) .. (31.23,80.75) ;
\draw [shift={(30.04,82.2)}, rotate = 308.18] [color={rgb, 255:red, 0; green, 0; blue, 0 }  ][line width=0.75]    (6.56,-1.97) .. controls (4.17,-0.84) and (1.99,-0.18) .. (0,0) .. controls (1.99,0.18) and (4.17,0.84) .. (6.56,1.97)   ;
\draw    (84.07,107.59) -- (87.49,89.32) ;
\draw [shift={(87.86,87.36)}, rotate = 100.62] [color={rgb, 255:red, 0; green, 0; blue, 0 }  ][line width=0.75]    (6.56,-1.97) .. controls (4.17,-0.84) and (1.99,-0.18) .. (0,0) .. controls (1.99,0.18) and (4.17,0.84) .. (6.56,1.97)   ;
\draw [color={rgb, 255:red, 65; green, 117; blue, 5 }  ,draw opacity=1 ]   (128.97,89.14) -- (177.94,89.07) ;
\draw [shift={(180.94,89.07)}, rotate = 179.92] [fill={rgb, 255:red, 65; green, 117; blue, 5 }  ,fill opacity=1 ][line width=0.08]  [draw opacity=0] (8.93,-4.29) -- (0,0) -- (8.93,4.29) -- cycle    ;
\draw [color={rgb, 255:red, 74; green, 144; blue, 226 }  ,draw opacity=1 ]   (230.86,63.67) .. controls (232.31,65.53) and (232.1,67.19) .. (230.24,68.64) .. controls (228.38,70.09) and (228.18,71.74) .. (229.63,73.6) .. controls (231.08,75.46) and (230.88,77.11) .. (229.02,78.56) -- (228.41,83.45) -- (228.41,83.45) ;
\draw    (229.27,87.58) .. controls (249.38,113.76) and (270.61,70.99) .. (232.18,84.05) ;
\draw [shift={(230.38,84.69)}, rotate = 339.82] [color={rgb, 255:red, 0; green, 0; blue, 0 }  ][line width=0.75]    (6.56,-1.97) .. controls (4.17,-0.84) and (1.99,-0.18) .. (0,0) .. controls (1.99,0.18) and (4.17,0.84) .. (6.56,1.97)   ;
\draw  [fill={rgb, 255:red, 0; green, 0; blue, 0 }  ,fill opacity=1 ] (224.54,86.02) .. controls (224.54,84.5) and (225.77,83.27) .. (227.29,83.27) .. controls (228.81,83.27) and (230.04,84.5) .. (230.04,86.02) .. controls (230.04,87.54) and (228.81,88.77) .. (227.29,88.77) .. controls (225.77,88.77) and (224.54,87.54) .. (224.54,86.02) -- cycle ;
\draw    (224.41,87.05) .. controls (190.53,98.04) and (219.64,119.29) .. (225.84,90.02) ;
\draw [shift={(226.19,88.16)}, rotate = 99.46] [color={rgb, 255:red, 0; green, 0; blue, 0 }  ][line width=0.75]    (6.56,-1.97) .. controls (4.17,-0.84) and (1.99,-0.18) .. (0,0) .. controls (1.99,0.18) and (4.17,0.84) .. (6.56,1.97)   ;
\draw    (227.29,82.87) .. controls (223.2,57.41) and (195.46,72.79) .. (223.76,83.99) .. controls (223.92,84.05) and (224.07,84.12) .. (224.23,84.18) ;
\draw [shift={(226.11,84.87)}, rotate = 199.28] [color={rgb, 255:red, 0; green, 0; blue, 0 }  ][line width=0.75]    (6.56,-1.97) .. controls (4.17,-0.84) and (1.99,-0.18) .. (0,0) .. controls (1.99,0.18) and (4.17,0.84) .. (6.56,1.97)   ;
\draw    (233.3,110.12) -- (229.05,90.73) ;
\draw [shift={(228.62,88.77)}, rotate = 77.64] [color={rgb, 255:red, 0; green, 0; blue, 0 }  ][line width=0.75]    (6.56,-1.97) .. controls (4.17,-0.84) and (1.99,-0.18) .. (0,0) .. controls (1.99,0.18) and (4.17,0.84) .. (6.56,1.97)   ;
\draw [color={rgb, 255:red, 74; green, 144; blue, 226 }  ,draw opacity=1 ]   (53.06,206.89) .. controls (50.76,207.39) and (49.36,206.49) .. (48.86,204.18) .. controls (48.37,201.87) and (46.97,200.97) .. (44.66,201.47) .. controls (42.35,201.97) and (40.95,201.07) .. (40.46,198.76) -- (37.46,196.83) -- (37.46,196.83) ;
\draw    (56.57,211.5) .. controls (76.68,237.67) and (98.54,194.13) .. (60.13,207.17) ;
\draw [shift={(58.33,207.81)}, rotate = 339.82] [color={rgb, 255:red, 0; green, 0; blue, 0 }  ][line width=0.75]    (6.56,-1.97) .. controls (4.17,-0.84) and (1.99,-0.18) .. (0,0) .. controls (1.99,0.18) and (4.17,0.84) .. (6.56,1.97)   ;
\draw  [fill={rgb, 255:red, 0; green, 0; blue, 0 }  ,fill opacity=1 ] (52.49,209.14) .. controls (52.49,207.62) and (53.72,206.39) .. (55.24,206.39) .. controls (56.76,206.39) and (57.99,207.62) .. (57.99,209.14) .. controls (57.99,210.66) and (56.76,211.89) .. (55.24,211.89) .. controls (53.72,211.89) and (52.49,210.66) .. (52.49,209.14) -- cycle ;
\draw    (58.13,210.39) .. controls (104.73,205.81) and (50.6,175.2) .. (54.9,204.52) ;
\draw [shift={(55.24,206.39)}, rotate = 258.02] [color={rgb, 255:red, 0; green, 0; blue, 0 }  ][line width=0.75]    (6.56,-1.97) .. controls (4.17,-0.84) and (1.99,-0.18) .. (0,0) .. controls (1.99,0.18) and (4.17,0.84) .. (6.56,1.97)   ;
\draw [color={rgb, 255:red, 208; green, 2; blue, 27 }  ,draw opacity=1 ]   (52.49,209.14) .. controls (11.13,214.28) and (50.49,243.19) .. (53.75,213.6) ;
\draw [shift={(53.91,211.72)}, rotate = 93.12] [color={rgb, 255:red, 208; green, 2; blue, 27 }  ,draw opacity=1 ][line width=0.75]    (6.56,-1.97) .. controls (4.17,-0.84) and (1.99,-0.18) .. (0,0) .. controls (1.99,0.18) and (4.17,0.84) .. (6.56,1.97)   ;
\draw [color={rgb, 255:red, 65; green, 117; blue, 5 }  ,draw opacity=1 ]   (122.77,209.81) -- (179.43,210.15) ;
\draw [shift={(182.43,210.17)}, rotate = 180.34] [fill={rgb, 255:red, 65; green, 117; blue, 5 }  ,fill opacity=1 ][line width=0.08]  [draw opacity=0] (8.93,-4.29) -- (0,0) -- (8.93,4.29) -- cycle    ;
\draw [color={rgb, 255:red, 74; green, 144; blue, 226 }  ,draw opacity=1 ]   (221.49,208.94) .. controls (219.34,209.91) and (217.78,209.31) .. (216.81,207.16) .. controls (215.84,205.01) and (214.29,204.42) .. (212.14,205.39) .. controls (209.99,206.36) and (208.43,205.76) .. (207.46,203.61) -- (203.77,202.21) -- (203.77,202.21) ;
\draw    (225.57,211.3) .. controls (245.68,237.47) and (267.54,193.93) .. (229.13,206.97) ;
\draw [shift={(227.33,207.61)}, rotate = 339.82] [color={rgb, 255:red, 0; green, 0; blue, 0 }  ][line width=0.75]    (6.56,-1.97) .. controls (4.17,-0.84) and (1.99,-0.18) .. (0,0) .. controls (1.99,0.18) and (4.17,0.84) .. (6.56,1.97)   ;
\draw  [fill={rgb, 255:red, 0; green, 0; blue, 0 }  ,fill opacity=1 ] (221.49,208.94) .. controls (221.49,207.42) and (222.72,206.19) .. (224.24,206.19) .. controls (225.76,206.19) and (226.99,207.42) .. (226.99,208.94) .. controls (226.99,210.46) and (225.76,211.69) .. (224.24,211.69) .. controls (222.72,211.69) and (221.49,210.46) .. (221.49,208.94) -- cycle ;
\draw    (227.13,210.19) .. controls (273.73,205.61) and (219.6,175) .. (223.9,204.32) ;
\draw [shift={(224.24,206.19)}, rotate = 258.02] [color={rgb, 255:red, 0; green, 0; blue, 0 }  ][line width=0.75]    (6.56,-1.97) .. controls (4.17,-0.84) and (1.99,-0.18) .. (0,0) .. controls (1.99,0.18) and (4.17,0.84) .. (6.56,1.97)   ;
\draw (132.19,24.24) node [anchor=north west][inner sep=0.75pt]  [font=\scriptsize] [align=left] {reverse $\beta$};
\draw (58.81,24.78) node [anchor=north west][inner sep=0.75pt]  [font=\scriptsize]  {$\textcolor[rgb]{0.82,0.01,0.11}{\beta }$};
\draw (223.06,24.27) node [anchor=north west][inner sep=0.75pt]  [font=\scriptsize]  {$\textcolor[rgb]{0.82,0.01,0.11}{\beta }$};
\draw (134.9,138.31) node [anchor=north west][inner sep=0.75pt]  [font=\scriptsize] [align=left] { slide $\beta$};
\draw (59.24,140.28) node [anchor=north west][inner sep=0.75pt]  [font=\scriptsize]  {$\textcolor[rgb]{0.82,0.01,0.11}{\alpha }$};
\draw (48.25,163.48) node [anchor=north west][inner sep=0.75pt]  [font=\scriptsize]  {$\textcolor[rgb]{0.25,0.46,0.02}{\beta }$};
\draw (11.29,139.33) node [anchor=north west][inner sep=0.75pt]  [font=\scriptsize]  {$\textcolor[rgb]{0.29,0.56,0.89}{v}$};
\draw (103.67,159.52) node [anchor=north west][inner sep=0.75pt]  [font=\scriptsize]  {$\textcolor[rgb]{0.29,0.56,0.89}{w}$};
\draw (134.9,152.36) node [anchor=north west][inner sep=0.75pt]  [font=\scriptsize] [align=left] {along $\alpha$};
\draw (236.38,140.33) node [anchor=north west][inner sep=0.75pt]  [font=\scriptsize]  {$\textcolor[rgb]{0.82,0.01,0.11}{\alpha }$};
\draw (258,159.28) node [anchor=north west][inner sep=0.75pt]  [font=\scriptsize]  {$\textcolor[rgb]{0.25,0.46,0.02}{\beta }$};
\draw (187.86,137.38) node [anchor=north west][inner sep=0.75pt]  [font=\scriptsize]  {$\textcolor[rgb]{0.29,0.56,0.89}{v}$};
\draw (294.81,154.42) node [anchor=north west][inner sep=0.75pt]  [font=\scriptsize]  {$\textcolor[rgb]{0.29,0.56,0.89}{w}$};
\draw (55.42,88.38) node [anchor=north west][inner sep=0.75pt]  [font=\scriptsize]  {$\textcolor[rgb]{0.82,0.01,0.11}{\alpha }$};
\draw (12.7,74.03) node [anchor=north west][inner sep=0.75pt]  [font=\scriptsize]  {$\textcolor[rgb]{0.29,0.56,0.89}{v}$};
\draw (92.64,62.44) node [anchor=north west][inner sep=0.75pt]  [font=\scriptsize]  {$\textcolor[rgb]{0.29,0.56,0.89}{w}$};
\draw (34.86,58.96) node [anchor=north west][inner sep=0.75pt]  [font=\scriptsize]  {$b$};
\draw (5.75,92.51) node [anchor=north west][inner sep=0.75pt]  [font=\scriptsize]  {$c$};
\draw (85.75,100.51) node [anchor=north west][inner sep=0.75pt]  [font=\scriptsize]  {$d$};
\draw (108.42,94.29) node [anchor=north west][inner sep=0.75pt]  [font=\scriptsize]  {$k$};
\draw (128.54,77.84) node [anchor=north west][inner sep=0.75pt]  [font=\scriptsize] [align=left] {contract $\alpha$ };
\draw (130,93.03) node [anchor=north west][inner sep=0.75pt]  [font=\scriptsize] [align=left] {towards $v$};
\draw (234.96,63.41) node [anchor=north west][inner sep=0.75pt]  [font=\scriptsize]  {$\textcolor[rgb]{0.29,0.56,0.89}{w}$};
\draw (201.62,69.75) node [anchor=north west][inner sep=0.75pt]  [font=\scriptsize]  {$b$};
\draw (206.29,106.59) node [anchor=north west][inner sep=0.75pt]  [font=\scriptsize]  {$c$};
\draw (236.45,104.72) node [anchor=north west][inner sep=0.75pt]  [font=\scriptsize]  {$d$};
\draw (247.85,95.7) node [anchor=north west][inner sep=0.75pt]  [font=\scriptsize]  {$k$};
\draw (22.57,216.47) node [anchor=north west][inner sep=0.75pt]  [font=\scriptsize]  {$\textcolor[rgb]{0.82,0.01,0.11}{\beta }$};
\draw (127.02,198.53) node [anchor=north west][inner sep=0.75pt]  [font=\scriptsize] [align=left] {remove $\beta$};
\end{tikzpicture}
\end{center}
\caption{Examples of graph transformations}
\label{fig:graphtrafos}
\end{figure}

These graph transformations are illustrated in Figure \ref{fig:graphtrafos}. Note that they are not independent. Contracting an edge $\alpha$ towards $t(\alpha)$ is the same as first sliding some edge ends  along $\alpha$ and then contracting $\alpha$ towards $t(\alpha)$. 
Contracting an edge $\alpha$ towards $t(\alpha)$ is also the same as first reversing $\alpha$, then contracting $\alpha$ towards $s(\alpha)$ and then reversing $\alpha$. 
By reversing $\alpha$ and $\beta$ before and after a slide, one can reduce all edge slides to the ones that slide the target end of $\beta$ along the left of $\alpha$. 

There are of course other possible graph transformations such as deleting edges, which is dual to edge contractions. 
However, the graph transformations in Definition \ref{def:graphtrafos} are sufficient to transform any connected ribbon graph into a standard graph. 
This  is well-known and appears implicitly in many publications. We summarise the argument for the convenience of the reader.

\begin{proposition}\label{prop:standardgraph}   Every connected ribbon graph
can be transformed into the standard graph  \eqref{eq:standardgraph} by edge reversals,  edge slides, edge contractions and loop deletions.
\end{proposition}

\begin{proof}
Selecting a maximal tree in $\Gamma$ and contracting all edges in the tree  transforms $\Gamma$ into a graph $\Gamma'$ with a single vertex. By applying edge slides one can transform $\Gamma'$ into a graph $\Gamma''$ that coincides with \eqref{eq:standardgraph} up to edge orientation and up to the presence  a number of isolated loops between the cilium and the starting end of $\alpha_1$. This follows from an analogous statement for chord diagrams, which correspond to ribbon graphs with a single vertex, see for instance Chmutov, Duzhin and Mostovoy \cite[Sec.~4.8.6]{CDM}. Deleting the isolated loops and reversing edges in $\Gamma''$ then yields the standard graph \eqref{eq:standardgraph}.
\end{proof}

\section{(Co)modules from Hopf monoids and ribbon graphs}
\label{sec:protected space}

In this section we use involutive Hopf monoids in symmetric monoidal categories to assign (co)modules over  Hopf monoids to  ciliated  ribbon graphs. In Section \ref{sec:graphindependence}
we then show that their biinvariants are topological invariants: their isomorphism classes depend only on the genus of the  surface obtained by attaching discs to the faces of the graph. In Sections \ref{sec:sset} and \ref{sec:cat} we determine these biinvariants for simplicial groups as Hopf monoids in $\SSet$ and for crossed modules as group objects in $\Cat$. 

The construction generalises Kitaev's quantum double model and the toric code from \cite{Ki},  which was first formulated for the  group algebra of a finite group over $\C$ and  then generalised by Buerschaper et al.~in \cite{BMCA} to finite-dimensional semisimple $C^*$-Hopf algebras. A very similar construction to the one in this article is used in  \cite{MV} to obtain mapping class group actions from pivotal  Hopf  monoids in symmetric monoidal categories. The work \cite{MV} considers the biinvariants of a Yetter-Drinfeld module structure assigned to the standard graph
\eqref{eq:standardgraph}, but it does not establish  that the biinvariants are graph-independent. 

The construction of the (co)module structures from an involutive Hopf monoid  and  a ciliated ribbon graph in this section is directly analogous to the one in \cite{MV}, which in turn is a straightforward generalisation of  \cite{Ki,BMCA}. The only difference is that $H^*$-modules in \cite{BMCA} are replaced by $H$-comodules and $D(H)$-modules by Yetter-Drinfeld modules over $H$.

What differs substantially from \cite{Ki,BMCA} are the notions of (co)invariants, biinvariants and the construction of the topological invariant. The works in \cite{Ki,BMCA}  rely on  the normalised Haar integral  of a finite-dimensional semisimple complex Hopf algebra,  which is not available in our setting. Our construction is more general, as the only
assumptions are that the underlying symmetric monoidal category is complete and finitely cocomplete and the Hopf monoid involutive.
The article \cite{MV} also allows pivotal Hopf monoids. The involutive Hopf monoids in this article are examples of pivotal Hopf monoids,  with their unit as  pivotal structure.

Let $H$ be an involutive Hopf monoid in a complete and finitely cocomplete symmetric monoidal category $\mac$ and $\Gamma$ a ciliated ribbon graph  with vertex set $V$, edge set $E$ and face set $F$. 

We consider the $|E|$-fold tensor product of $H$ with itself, together with an assignment of the copies of $H$ in this tensor product to the  edges of $\Gamma$, which we emphasise by writing $H^{\oo E}$. If $E=\emptyset$, we set $H^{\oo E}=e$.  The object $H^{\oo E}$ can be viewed as the counterpart of the Hilbert space of Kitaev's quantum double model in \cite{Ki,BMCA}. 

We assign to each edge $\alpha\in E$ two $H$-module structures $\rhd_{\alpha\pm}: H\oo H^{\oo E}\to H^{\oo E}$ and  $H$-comodule structures $\delta_{\alpha\pm}: H^{\oo E}\to H\oo H^{\oo E}$. The $H$-module structures $\rhd_{\alpha+}$ and $\rhd_{\alpha-}$ are assigned to the target and starting end of $\alpha$ and the $H$-comodule structures to its left and right side, respectively.  They are induced by the standard  $H$-(co)module  structures on $H$ via left (co)multiplication.

This requires some notation. Given a morphism $f: H\to K$ in $\mac$ and an edge $\alpha\in E$ we write  $f_\alpha$  for the morphism that applies $f$ to the copy of $H$ in $H^{\oo E}$ that belongs to $\alpha$  and the identity morphism to the other copies. We write $\tau_{\alpha}: H^{\oo E}\to H^{\oo E}$ or $\tau_{\alpha}: H\oo H^{\oo E}\to H\oo H^{\oo E}$ for the composite of braidings that moves the copy of $H$ for $\alpha$ to the left. 
We denote by  $m_{\alpha}: H \oo H^{\oo E} \to H^{\oo E}$   the morphism that moves the first copy of $H$   to the left of the one for $\alpha$ and then applies $m$ to them.

\begin{definition}\label{def:TraingleCoActions} 
The $H$-module structures $\rhd_{\alpha\pm}: H \oo H^{\oo E}\to H^{\oo E}$ and $H$-comodule structures $\delta_{\alpha\pm}:H^{\oo E}\to H \oo H^{\oo E}$ for an edge $\alpha\in E$ are 
\begin{align*}
\rhd_{\alpha +} := m_{\alpha}, \quad 
\rhd_{\alpha -}:= S_{\alpha}\circ \rhd_{\alpha + }\circ (1_H \oo S_{\alpha}), \quad
 \delta_{\alpha +}:= \tau_{\alpha}\circ \Delta_{\alpha},  \quad
 \delta_{\alpha -}:= (1_H \oo S_{\alpha}) \circ \delta_{\alpha+}\circ S_{\alpha}.
\end{align*}
\end{definition}

By definition, the (co)module structures assigned to different edges of a graph commute, since they (co)act on different copies of $H$ in the tensor product $H^{\oo E}$. A direct computation using \eqref{eq:counit} and \eqref{eq:antialg}  shows that the two $H$-(co)module  structures assigned to a given edge commute  as well. The proof is directly analogous to the ones for Hopf algebras in \cite{BMCA}.

\begin{lemma}\label{Lemma:TraingelOperatorsAreHopfMonoid}\cite[Lemma 5.2, 2.]{MV}
For any edge $\alpha\in E$ the $H$-module structures $\rhd_{\alpha \pm}$ and the $H$-comodule structures $\delta_{\alpha \pm}$ commute:
\begin{align*}
\rhd_{\alpha -}\circ (1_H \oo \rhd_{\alpha +}) &=\rhd_{\alpha +}\circ (1_H \oo \rhd_{\alpha -}) \circ (\tau_{H, H}\oo 1_{H^{\oo E}})\text{,}\\
(1_H \oo \delta_{\alpha -})\circ \delta_{\alpha +} &= (\tau_{H, H}\oo 1_{H^{\oo E}}) \circ (1_H \oo \delta_{\alpha +})\circ \delta_{\alpha -}\text{.}
\end{align*}
\end{lemma}

The (co)module structures from Definition \ref{def:TraingleCoActions} define an $H$-module structure on $H^{\oo E}$ for each ciliated vertex $v$  and an $H$-comodule structure  on $H^{\oo E}$ for each  ciliated face $f$  of $\Gamma$. The former applies the comultiplication to $H$,  distributes the resulting copies of $H$ to the  edge ends at $v$ according to their ordering  and  acts on them with $\rhd_{\alpha\pm}$ according to their orientation. Dually, the coaction applies the $H$-coaction $\delta_{\alpha\pm}$ to each edge $\alpha$ in $f$, depending on its orientation relative to $f$, and multiplies the resulting copies of $H$ according to  the order of the edge sides in $f$.

\begin{definition} \label{def:vertexfaceops} \cite[Def.~5.3]{MV}$ $
\begin{enumerate}
\item The {$H$-module structure} $\rhd_v: H\oo H^{\oo E}\to H^{\oo E}$ for a ciliated vertex $v$ with incident edge ends $\alpha_1<\alpha_2<\ldots<\alpha_n$ is
\begin{align}
\rhd_v=\rhd_{\alpha_1}\circ (1_H\oo \rhd_{\alpha_2})\circ \ldots \circ  (1_{H^{\oo (n-1)}}\oo  \rhd_{\alpha_n}) \circ (\Delta^{(n-1)} \oo 1_{H^{\oo E}}),
\end{align}
where $\rhd_\alpha=\rhd_{e(\alpha)+}$ if $\alpha$ is incoming, $\rhd_{\alpha}=\rhd_{e(\alpha)-}$ if $\alpha$ is outgoing  and $e(\alpha)$ is the edge of  $\alpha$.

\item The {$H$-comodule structure} $\delta_f: H^{\oo E}\to H\oo H^{\oo E}$ for a ciliated face  $f$ that traverses the edges $\alpha_n,\alpha_{n-1},\ldots,\alpha_1$  in this order
is
\begin{align}
\delta_f= (m^{(n-1)}\oo 1_{H^{\oo E}}) \circ (1_{H^{\oo(n-1)}}\oo \delta_{\alpha_r})\circ \ldots\circ (1_H\oo\delta_{\alpha_2})\circ \delta_{\alpha_1},
\end{align}
where $\delta_\alpha=\delta_{e(\alpha)+}$ if $\alpha$ is traversed with, $\delta_{\alpha}=\delta_{e(\alpha)-}$ if $\alpha$ is traversed against its orientation  and $e(\alpha)$ is the edge of  $\alpha$.
\end{enumerate}
To an isolated vertex  and face we assign the (co)module  structures $\rhd_v=\epsilon\oo 1_{H^{\oo E}}$ and $\delta_f=\eta\oo 1_{H^{\oo E}}$.
\end{definition}

To avoid heavy notation we use Sweedler notation and   describe these (co)module structures by labelling edges of  a graph  with letters representing the associated copies of $H$.

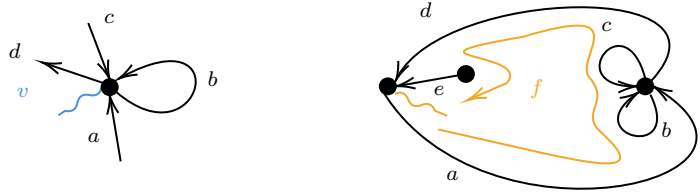
\begin{example}\label{example:VertexAndFaceOperators} $ $
\begin{center}
\tikzset{every picture/.style={line width=0.75pt}} 
\begin{tikzpicture}[x=0.75pt,y=0.75pt,yscale=-1.45,xscale=1.45]
\draw [color={rgb, 255:red, 74; green, 144; blue, 226 }  ,draw opacity=1 ]   (44.71,38.61) .. controls (44.23,40.92) and (42.84,41.83) .. (40.53,41.35) .. controls (38.22,40.86) and (36.83,41.77) .. (36.35,44.08) .. controls (35.86,46.39) and (34.47,47.3) .. (32.16,46.82) -- (29.82,48.35) -- (29.82,48.35) ;
\draw    (40.1,16.35) -- (46.76,33.99) ;
\draw [shift={(47.47,35.86)}, rotate = 249.32] [color={rgb, 255:red, 0; green, 0; blue, 0 }  ][line width=0.75]    (6.56,-1.97) .. controls (4.17,-0.84) and (1.99,-0.18) .. (0,0) .. controls (1.99,0.18) and (4.17,0.84) .. (6.56,1.97)   ;
\draw [color={rgb, 255:red, 0; green, 0; blue, 0 }  ,draw opacity=1 ]   (49.82,40.35) .. controls (79.1,65.57) and (91.85,13.13) .. (51.63,35.09) ;
\draw [shift={(50.39,35.78)}, rotate = 330.38] [color={rgb, 255:red, 0; green, 0; blue, 0 }  ,draw opacity=1 ][line width=0.75]    (6.56,-1.97) .. controls (4.17,-0.84) and (1.99,-0.18) .. (0,0) .. controls (1.99,0.18) and (4.17,0.84) .. (6.56,1.97)   ;
\draw    (45.29,37.18) -- (26.09,30.65) ;
\draw [shift={(24.2,30.01)}, rotate = 18.78] [color={rgb, 255:red, 0; green, 0; blue, 0 }  ][line width=0.75]    (6.56,-1.97) .. controls (4.17,-0.84) and (1.99,-0.18) .. (0,0) .. controls (1.99,0.18) and (4.17,0.84) .. (6.56,1.97)   ;
\draw  [fill={rgb, 255:red, 0; green, 0; blue, 0 }  ,fill opacity=1 ] (44.71,38.61) .. controls (44.71,37.09) and (45.95,35.86) .. (47.47,35.86) .. controls (48.99,35.86) and (50.22,37.09) .. (50.22,38.61) .. controls (50.22,40.13) and (48.99,41.36) .. (47.47,41.36) .. controls (45.95,41.36) and (44.71,40.13) .. (44.71,38.61) -- cycle ;
\draw    (51.25,64.35) -- (47.79,43.34) ;
\draw [shift={(47.47,41.36)}, rotate = 80.66] [color={rgb, 255:red, 0; green, 0; blue, 0 }  ][line width=0.75]    (6.56,-1.97) .. controls (4.17,-0.84) and (1.99,-0.18) .. (0,0) .. controls (1.99,0.18) and (4.17,0.84) .. (6.56,1.97)   ;
\draw  [fill={rgb, 255:red, 0; green, 0; blue, 0 }  ,fill opacity=1 ] (140.43,38.32) .. controls (140.43,36.8) and (141.66,35.57) .. (143.18,35.57) .. controls (144.7,35.57) and (145.93,36.8) .. (145.93,38.32) .. controls (145.93,39.84) and (144.7,41.08) .. (143.18,41.08) .. controls (141.66,41.08) and (140.43,39.84) .. (140.43,38.32) -- cycle ;
\draw  [fill={rgb, 255:red, 0; green, 0; blue, 0 }  ,fill opacity=1 ] (167.29,34.04) .. controls (167.29,32.52) and (168.52,31.29) .. (170.04,31.29) .. controls (171.56,31.29) and (172.79,32.52) .. (172.79,34.04) .. controls (172.79,35.56) and (171.56,36.79) .. (170.04,36.79) .. controls (168.52,36.79) and (167.29,35.56) .. (167.29,34.04) -- cycle ;
\draw  [fill={rgb, 255:red, 0; green, 0; blue, 0 }  ,fill opacity=1 ] (229,38.32) .. controls (229,36.8) and (230.23,35.57) .. (231.75,35.57) .. controls (233.27,35.57) and (234.5,36.8) .. (234.5,38.32) .. controls (234.5,39.84) and (233.27,41.08) .. (231.75,41.08) .. controls (230.23,41.08) and (229,39.84) .. (229,38.32) -- cycle ;
\draw [color={rgb, 255:red, 0; green, 0; blue, 0 }  ,draw opacity=1 ]   (233.53,41.21) .. controls (244.95,64.61) and (209.12,55.69) .. (227.97,41.46) ;
\draw [shift={(229.53,40.35)}, rotate = 146.31] [color={rgb, 255:red, 0; green, 0; blue, 0 }  ,draw opacity=1 ][line width=0.75]    (6.56,-1.97) .. controls (4.17,-0.84) and (1.99,-0.18) .. (0,0) .. controls (1.99,0.18) and (4.17,0.84) .. (6.56,1.97)   ;
\draw [color={rgb, 255:red, 0; green, 0; blue, 0 }  ,draw opacity=1 ]   (231.75,35.57) .. controls (224.85,9.03) and (202.99,37.68) .. (227.06,38.33) ;
\draw [shift={(229,38.32)}, rotate = 178.19] [color={rgb, 255:red, 0; green, 0; blue, 0 }  ,draw opacity=1 ][line width=0.75]    (6.56,-1.97) .. controls (4.17,-0.84) and (1.99,-0.18) .. (0,0) .. controls (1.99,0.18) and (4.17,0.84) .. (6.56,1.97)   ;
\draw    (233.93,36.3) .. controls (265.95,1.13) and (163.16,4.16) .. (146.63,34.87) ;
\draw [shift={(145.93,36.3)}, rotate = 293.7] [color={rgb, 255:red, 0; green, 0; blue, 0 }  ][line width=0.75]    (6.56,-1.97) .. controls (4.17,-0.84) and (1.99,-0.18) .. (0,0) .. controls (1.99,0.18) and (4.17,0.84) .. (6.56,1.97)   ;
\draw    (141.65,40.3) .. controls (176.16,96.87) and (287.39,71.67) .. (236.11,39.31) ;
\draw [shift={(234.5,38.32)}, rotate = 30.77] [color={rgb, 255:red, 0; green, 0; blue, 0 }  ][line width=0.75]    (6.56,-1.97) .. controls (4.17,-0.84) and (1.99,-0.18) .. (0,0) .. controls (1.99,0.18) and (4.17,0.84) .. (6.56,1.97)   ;
\draw    (170.04,34.04) -- (147.9,37.97) ;
\draw [shift={(145.93,38.32)}, rotate = 349.92] [color={rgb, 255:red, 0; green, 0; blue, 0 }  ][line width=0.75]    (6.56,-1.97) .. controls (4.17,-0.84) and (1.99,-0.18) .. (0,0) .. controls (1.99,0.18) and (4.17,0.84) .. (6.56,1.97)   ;
\draw [color={rgb, 255:red, 245; green, 166; blue, 35 }  ,draw opacity=1 ]   (145.57,40.9) .. controls (147.76,40.01) and (149.29,40.66) .. (150.17,42.85) .. controls (151.05,45.04) and (152.58,45.69) .. (154.77,44.81) .. controls (156.96,43.93) and (158.49,44.58) .. (159.37,46.77) -- (163.63,48.58) -- (163.63,48.58) ;
\draw [color={rgb, 255:red, 245; green, 166; blue, 35 }  ,draw opacity=1 ]   (160.57,53.29) .. controls (209.91,64.3) and (237.63,73.15) .. (215.34,53.15) ;
\draw [color={rgb, 255:red, 245; green, 166; blue, 35 }  ,draw opacity=1 ]   (215.34,53.15) .. controls (210.2,46.3) and (219.63,43.72) .. (212.49,35.15) ;
\draw [color={rgb, 255:red, 245; green, 166; blue, 35 }  ,draw opacity=1 ]   (189.63,21.15) .. controls (229.91,8.87) and (207.06,28.87) .. (212.49,35.15) ;
\draw [color={rgb, 255:red, 245; green, 166; blue, 35 }  ,draw opacity=1 ]   (189.63,21.15) .. controls (142.77,29.15) and (193.06,26.01) .. (184.2,37.15) ;
\draw [color={rgb, 255:red, 245; green, 166; blue, 35 }  ,draw opacity=1 ]   (184.2,37.15) -- (172.32,42.35) ;
\draw [shift={(170.49,43.15)}, rotate = 336.37] [color={rgb, 255:red, 245; green, 166; blue, 35 }  ,draw opacity=1 ][line width=0.75]    (6.56,-1.97) .. controls (4.17,-0.84) and (1.99,-0.18) .. (0,0) .. controls (1.99,0.18) and (4.17,0.84) .. (6.56,1.97)   ;
\draw (14.56,37.92) node [anchor=north west][inner sep=0.75pt]  [font=\scriptsize]  {$\textcolor[rgb]{0.29,0.56,0.89}{v}$};
\draw (80.19,31.89) node [anchor=north west][inner sep=0.75pt]  [font=\scriptsize]  {$b$};
\draw (44.57,12.22) node [anchor=north west][inner sep=0.75pt]  [font=\scriptsize]  {$c$};
\draw (11.71,22.22) node [anchor=north west][inner sep=0.75pt]  [font=\scriptsize]  {$d$};
\draw (38.86,53.37) node [anchor=north west][inner sep=0.75pt]  [font=\scriptsize]  {$a$};
\draw (162.29,66.51) node [anchor=north west][inner sep=0.75pt]  [font=\scriptsize]  {$a$};
\draw (236.33,49.6) node [anchor=north west][inner sep=0.75pt]  [font=\scriptsize]  {$b$};
\draw (215.71,14.99) node [anchor=north west][inner sep=0.75pt]  [font=\scriptsize]  {$c$};
\draw (153.14,7.97) node [anchor=north west][inner sep=0.75pt]  [font=\scriptsize]  {$d$};
\draw (158,37.26) node [anchor=north west][inner sep=0.75pt]  [font=\scriptsize]  {$e$};
\draw (191.14,34.83) node [anchor=north west][inner sep=0.75pt]  [font=\scriptsize]  {$\textcolor[rgb]{0.96,0.65,0.14}{f}$};
\end{tikzpicture}
\end{center}

\vspace{-1cm}
The $H$-module structure $\rhd_v$ for the ciliated vertex $v$ with incident edge ends $t(a)< s(b)< t(b)< t(c)< s(d)$ and the $H$-comodule structure $\delta_f$ for the ciliated face $f= e \circ e^\inv \circ d \circ c^\inv \circ  b \circ a$  are 
\begin{align*}
h \rhd_v \, ( a\oo b\oo c\oo d) &= \low h 1 a \oo \low h 3 b S( \low h 2 ) \oo \low h 4 c \oo d S( \low h 5 ),\\
\delta_f \, (a\oo b\oo c\oo d\oo e) &= \low e 1 S( \low e 3 ) \low d 1 S( \low c 2 ) \low b 1 \low a 1 \oo  \low a 2 \oo \low b 2 \oo \low c 1 \oo \low d 2 \oo \low e 2.
\end{align*}
\end{example}

The interaction of the $H$-module and $H$-comodule structures assigned to ciliated vertices and faces of the graph is investigated in \cite{Ki, BMCA,MV}. 
They are local in the sense that the $H$-(co)module structure for a vertex (face) affects only those copies of $H$ that belong to their incident edges. As the action $\rhd_{\alpha+}$ for an edge $\alpha\in E$ acts by left- and $\rhd_{\alpha-}$  by right-multiplication, the $H$-module structures for different vertices commute. The same holds for the $H$-comodule structures at different faces. Moreover, $H$-module structures commute with $H$-comodule structures unless  their cilia share a vertex or a face. The $H$-module and $H$-comodule structure for each cilium define a Yetter-Drinfeld module structure.

\begin{lemma}\label{lemma:VertexAndFaceOperatorsCommmute}\cite[Lemma 5.5]{MV}
\begin{compactenum}
\item The $H$-left module structures for distinct vertices $v\neq v'\in V$ and the $H$-left comodule structures for distinct  faces $f\neq f'\in F$ commute for all choices of cilia:
\begin{align}
\rhd_{v'}\circ (1_H \oo \rhd_v) &=\rhd_v \circ (1_H \oo \rhd_{v'})\circ (\tau_{H, H}\oo 1_{H^{\oo E}}), \label{eq:VertexActionsCommute} \\
(1_H \oo \delta_{f'}) \circ \delta_f &= (\tau_{H, H}\oo 1_{H^{\oo E}})\circ (1_H \oo \delta_f)\circ \delta_{f'} \text{.}\label{eq:FaceCoactionsCommute}
\end{align}
\item If two cilia are at distinct vertices and distinct faces, the $H$-module structure for one of them commutes with the $H$-comodule structure for the other:
\begin{align}
\delta_f \circ \rhd_v= (1_H \oo  \rhd_v)\circ (\tau_{H, H}\oo 1_{H^{\oo E}}) \circ (1_H \oo \delta_f)\text{.}\label{eq:VertexAndFaceActionsCommute}
\end{align}
\item If $v\in V$ and $f\in F$ share a cilium, then $(H^{\oo E}, \rhd_v, \delta_f)$ is a Yetter-Drinfeld module over $H$.
\end{compactenum}
\end{lemma}

\begin{example}\label{ex:modcomodpi_1gen}
Let $H$ be an involutive Hopf monoid in $\mac$ and $\Gamma$ the standard graph \eqref{eq:standardgraph}  on a surface $\Sigma$ of genus $g\geq 1$.
Then the associated Yetter-Drinfeld module structure on  $H^{\oo E}$ is 
\begin{align}\label{eq:ydgen}
&h\rhd(a^1\oo b^1\oo \ldots\oo a^g\oo b^g)\\
&\qquad=\low h 3 a^1 S(\low h 1)\oo \low h 4 b^1 S(\low h 2)\oo\ldots\oo \low h {4g-1} a^g S(\low h {4g-3}) \oo \low h {4g}b^gS(\low h {4g-2})\nonumber\\
&\delta(a^1\oo b^1\oo \ldots\oo a^g\oo b^g)\nonumber\\
&\qquad=S(b^g_{(3)}) a^g_{(1)} b^g_{(1)}  S(a^g_{(3)})\cdots S(b^1_{(3)}) a^1_{(1)} b^1_{(1)}  S(a^1_{(3)})\oo a^1_{(2)}\oo b^1_{(2)}\oo\ldots\oo a^g_{(2)}\oo b^g_{(2)}.\nonumber
\end{align}

If $H$ is a group object in a cartesian monoidal category, this reduces to
\begin{align}\label{eq:YDmodulegroup}
&h\rhd (a_1,b_1,\ldots, a_g,b_g)=(h a_1 h^\inv, hb_1 h^\inv,\ldots, ha_g h^\inv, hb_g h^\inv)\\
&\delta(a_1,b_1,\ldots, a_g, b_g)=([b_g^\inv, a_g]\cdots[b_1^\inv, a_1], a_1,b_1,\ldots, a_g, b_g).\nonumber
\end{align}
\end{example}

If each vertex and face of $\Gamma$ is  equipped with a  cilium,  then Definition \ref{def:vertexfaceops} assigns 
 an $H$-(co)module structure on $H^{\oo E}$ to each vertex (face) of $\Gamma$. By Lemma \ref{lemma:VertexAndFaceOperatorsCommmute} these (co)module structures  commute and hence  combine into  $H^{\oo E}$-module and $H^{\oo F}$-comodule structures on $H^{\oo E}$.

\begin{definition}\label{def:setact} The $H^{\oo n}$-module  structure for a subset $\emptyset \neq \mathcal V:=\{v_1, \dots, v_n\}\subset V$ and  $H^{\oo m}$-comodule structure for a subset $\emptyset \neq \mathcal F:=\{f_1, \dots, f_m\} \subset F$  are 
\begin{align}
\rhd_{\mathcal V} &:= \rhd_{v_1}\circ (1_H \oo \rhd_{v_2}) \circ \dots \circ (1_{H^{\oo (n-2)}} \oo \rhd_{v_{n-1}})\circ (1_{H^{\oo (n-1)}}\oo \rhd_{v_n}): H^{\oo n}\oo H^{\oo E} \to H^{\oo E}, \label{eq:ComposedVertexAction}\\
\delta_{\mathcal F} &:= (1_{H^{\oo (m-1)}} \oo \delta_{f_m})\circ (1_{H^{\oo (m-2)}} \oo \delta_{f_{m-1}})\circ \dots \circ (1_H \oo \delta_{f_2}) \circ \delta_{f_1}: H^{\oo E} \to H^{\oo m}\oo H^{\oo E}.\nonumber
\end{align}
\end{definition}

Equations \eqref{eq:VertexActionsCommute} and \eqref{eq:FaceCoactionsCommute} ensure that the (co)actions do not depend on the numbering of vertices or faces in Definition \ref{def:setact}. That \eqref{eq:ComposedVertexAction} defines an $H^{\oo n}$-module structure  follows from the
 identity 
\begin{align*}
\rhd_{\mathcal V'} \circ (1_{H^{\oo \vert \mathcal V'\vert }} \oo \rhd_v) \circ (\tau_{H,H^{\oo \vert \mathcal V'\vert}} \oo 1_{H^{\oo E}}) = \rhd_v \circ (1_H \oo \rhd_{\mathcal V'}),
\end{align*}
  valid for any subset $\emptyset \neq \mathcal V'\subset V$, $v\in V\setminus \mathcal V'$. The dual statement for $\delta_{\mathcal F}$  follows analogously.

 The module and comodule structure  from Definition \ref{def:setact}  define the categorical counterpart of the  {\em protected space} or {\em ground state} in Kitaev's quantum double model. In the models based on a  finite-dimensional semisimple complex Hopf algebras in \cite{Ki,BMCA} the ground state  is  an eigenspace  of a Hamiltonian that combines these $H$-(co)module structures. The normalised Haar integral defines a projector on the ground state.
  In our setting these structures are not available. Instead, we consider the biinvariants from Definition \ref{def:biinvariants} for  the action and coaction from \eqref{eq:ComposedVertexAction}.

\begin{definition}\label{def:KitaevModel}$ $ 
The {\bf protected object} for an involutive Hopf monoid $H$ and  a ciliated ribbon graph $\Gamma$ are the biinvariants 
 $M_{inv}=\mathrm{Im}(\pi \circ \iota)$ of  $H^{\oo E}$ with the module structure $\rhd_V$ and comodule structure $\delta_F$  from \eqref{eq:ComposedVertexAction}. 
\end{definition}

In the quantum double models for  a finite-dimensional semisimple complex Hopf algebra 
 it is directly apparent that imposing (co)invariance under all  individual (co)actions at the vertices (faces) of a graph is the same as imposing (co)invariance under the combined action in Definition \ref{def:setact}. In this setting the (co)invariants for the individual (co)actions are linear subspaces of $H^{\oo E}$ and the (co)invariants of the combined (co)actions their intersections.
In our setting an analogous statement  follows from the universal properties of the 
coequaliser $\pi_{\mathcal V}: H^{\oo E}\to M^H_{\mathcal V}$  for the action $\rhd_{\mathcal V}$  and the  equaliser  $\iota_{\mathcal F}: M^{coH}_{\mathcal F}\to H^{\oo E}$ for the coaction $\delta_{\mathcal F}$, as given in Definition \ref{def:invariants}.

\begin{lemma}\label{Lemma:RelatingTheInvariants}$ $
Let $\emptyset\neq \mathcal V \subset V$, $\emptyset\neq \mathcal F\subset F$ be subsets. 
\begin{compactenum}
\item For any subset $\emptyset\neq \mathcal V'\subset\mathcal V$ the morphism $\pi_{\mathcal V}: H^{\oo E}\to M^H_{\mathcal V}$ satisfies
\begin{align}
\pi_{\mathcal V}\circ \rhd_{\mathcal V'}= \pi_{\mathcal V}\circ (\epsilon^{\vert \mathcal V'\vert}\oo 1_{H^{\oo E}})\label{eq:BigInvariantsCoequalise}.
\end{align}
There is a unique morphism $\chi_{\mathcal V', \mathcal V}: M_{\mathcal V'}^H\to M_{\mathcal V}^H$ with $\chi_{\mathcal V', \mathcal V} \circ \pi_{\mathcal V'}=\pi_{\mathcal V}$. 
It
is an epimorphism.\\
\item For any subset $\emptyset \neq \mathcal F'\subset \mathcal F$ the morphism $\iota_{\mathcal F}: M^{coH}_{\mathcal F}\to H^{\oo E}$ satisfies 
\begin{align}
\delta_{\mathcal F'}\circ \iota_{\mathcal F}= (\eta^{\vert \mathcal F'\vert} \oo 1_{H^{\oo E}})\circ \iota_{\mathcal F}\label{eq:BigCoinvariantsEqualise}.
\end{align}
There is a unique morphism $\xi_{\mathcal F', \mathcal F}: M_{\mathcal F}^{coH}\to M_{\mathcal F'}^{coH}$ with $\iota_{\mathcal F'}\circ \xi_{\mathcal F', \mathcal F}= \iota_{\mathcal F}$. 
It is a monomorphism.
\end{compactenum}
\end{lemma}
\begin{proof}
We prove 1., as 2.~is the dual statement.  It suffices to verify \eqref{eq:BigInvariantsCoequalise} for $\mathcal V=\{v_1, \ldots, v_n\}$, $\mathcal V'=\{v_j\}$,  and the claim follows by induction over $\vert \mathcal V'\vert$.  For this note first that Definition \ref{def:setact} implies
\begin{align}
\rhd_{\mathcal V}\circ (\eta^{\oo (j-1)}\oo 1_H \oo \eta^{\oo (n-j)}\oo 1_{H^{\oo E}})&= \rhd_{v_j} & &\forall j \in \{1, \dots, n\} \label{eq:GrosseUndEinzelneDreiecke}.
\end{align}
As $\pi_{\mathcal V}$ is the coequaliser of $\rhd_{\mathcal V}$ and $\epsilon^{\oo n}\oo 1_{H^{\oo E}}$ one obtains
\begin{align*}
\pi_{\mathcal V}\circ \rhd_{v_j}&\overset{\eqref{eq:GrosseUndEinzelneDreiecke}}{=} \pi_{\mathcal V}\circ \rhd_{\mathcal V}\circ (\eta^{\oo (j-1)}\oo 1_H\oo \eta^{\oo (n-j)}\oo 1_{H^{\oo E}})\\
&= \pi_{\mathcal V}\circ (\epsilon^{\oo n}\oo 1_{H^{\oo E}})\circ (\eta^{\oo (j-1)}\oo 1_H\oo \eta^{\oo (n-j)}\oo 1_{H^{\oo E}})\, = \,  \pi_{\mathcal V} \circ (\epsilon\oo 1_{H^{\oo E}}).
\end{align*}
Equation \eqref{eq:BigInvariantsCoequalise} and the universal property of the coequaliser $ \pi_{\mathcal V'}$ imply the existence of a unique morphism $\chi_{\mathcal V', \mathcal V}: M_{\mathcal V'}^H\to M_{\mathcal V}^H$ with $\chi_{\mathcal V', \mathcal V}\circ \pi_{\mathcal V'}=\pi_{\mathcal V}$. For any two morphisms $q_1, q_2: M_{\mathcal V}^H\to X$  with $q_1 \circ \chi_{\mathcal V', \mathcal V}=q_2 \circ \chi_{\mathcal V', \mathcal V}$ one  has
$q_1 \circ \chi_{\mathcal V', \mathcal V}\circ \pi_{\mathcal V'}= q_1\circ \pi_{\mathcal V}= q_2 \circ \pi_{\mathcal V}= q_2 \circ \chi_{\mathcal V', \mathcal V} \circ \pi_{\mathcal V'}$. 
As $\pi_{\mathcal V}$ is a coequaliser and hence an epimorphism, this implies $q_1=q_2$, and $\chi_{\mathcal V', \mathcal V}$ is an epimorphism. 
\end{proof}

It is also directly apparent from Definition \ref{def:setact} that (co)module morphisms with respect to all  individual (co)module structures at vertices  and faces in $\mathcal V$ and $\mathcal F$ are also (co)module morphisms with respect to the (co)actions  $\rhd_{\mathcal V}$ and $\delta_{\mathcal F}$. More precisely, 
for ciliated ribbon graphs 
$\Gamma, \Gamma'$, subsets $\emptyset\neq \mathcal V \subset V$, $\emptyset\neq \mathcal V' \subset V'$ and a bijection $\varphi:\mathcal V \to \mathcal V'$, $v\mapsto v'$,    any  morphism $g: H^{\oo E}\to H^{\oo E'}$ that is a module morphism with respect 
to $\rhd_v$ and $\rhd_{v'}$ for all $v \in V$ is also a module morphism with respect to $\rhd_{\mathcal V}$ and $\rhd_{\mathcal V'}$. An analogous statement holds for $\delta_{\mathcal F}$ and  comodule morphisms.

\section{Graph independence}
\label{sec:graphindependence}

In this section we show that the protected object from Definition \ref{def:KitaevModel} is a topological invariant:  Although its definition requires a ciliated ribbon graph $\Gamma$, its isomorphism class depends only on the homeomorphism class of the surface obtained by attaching  discs to the faces of $\Gamma$. 

To prove this, we show first in Section \ref{subsec:edge slides}  that  the (co)invariants associated to the (co)module structures at the vertices (faces) of $\Gamma$ depend neither on the edge orientation nor on the choices of the cilia. Reversing the orientation of edges and different choices of cilia yield  isomorphisms between these (co)invariants and hence also between the biinvariants.
We then show in Section \ref{subsec:contract} and  \ref{subsec:loops} that the other graph transformations from Definition \ref{def:graphtrafos} induce isomorphisms between the protected objects, although not necessarily between the (co)invariants. In Section \ref{subsec:topinv} we combine these results to obtain topological invariance and treat some simple examples.

As in Section \ref{sec:protected space} we consider a complete and finitely cocomplete symmetric monoidal category $\mathcal C$,  an involutive Hopf monoid $H$  in $\mathcal C$ and a ciliated ribbon graph $\Gamma$.

\subsection{Edge orientation reversal and  moving the cilium}
\label{subsec:edge slides}
As  edge orientation reversal switches the start and target and the left and right side of an edge $\alpha\in E$, it  exchanges the associated actions $\rhd_{\alpha\pm}$ and coactions $\delta_{\alpha\pm}$  from Definition \ref{def:TraingleCoActions}. It is directly apparent from their definitions that this is achieved by applying the antipode.

\begin{definition}\label{def:edgereversal}  The automorphism of $H^{\oo E}$ associated to the {\bf  reversal} of an edge $\alpha\in E$ is $S_\alpha: H^{\oo E}\to H^{\oo E}$.
\end{definition}

\begin{lemma}\label{Lemma:EdgeReversalModuleComoduleIso} For any ciliated vertex $v\in V$,   ciliated face $f\in F$ and edge $\beta\in E$  the edge reversal $S_{\beta}$  is an isomorphism of $H$-modules and $H$-comodules with respect to $\rhd_v$ and  $\delta_f$.
\end{lemma}

\begin{proof}
We denote by $\rhd_v'$ and $\delta_f'$ the module and comodule structure in the graph where the  orientation of $\beta$ is reversed and verify that  $\rhd_v' \circ (1_H\oo S_{\beta})=S_{\beta} \circ \rhd_v$ and $\delta_f'\circ S_\beta=(1_H\oo S_\beta)\circ\delta_f$. If $\beta$ is not incident at $v$ and $f$, the copy of $H$ in $H^{\oo E}$ assigned to $\beta$ is not affected by $\rhd_v,\rhd_v'$ and $\delta_f,\delta'_f$, and
the identity follows directly. If $\beta$ is incident at $v$ or $f$, it follows from the expressions for the (co)actions in Definitions \ref{def:TraingleCoActions} and \ref{def:vertexfaceops}.
\end{proof}

As a direct consequence of  Lemma \ref{Lemma:EdgeReversalModuleComoduleIso},  Lemma \ref{Lemma:IsosAufInvarianten} and Lemma \ref{Lemma:IsomorphismOfBiinvariants} one has

\begin{corollary}\label{cor:edgereverssal} Reversing the orientation of an edge  in $\Gamma$ to obtain  $\Gamma'$  induces isomorphisms between the invariants, coinvariants and protected objects of $\Gamma$ and $\Gamma'$.
\end{corollary}

\begin{lemma}\label{lemma:InvariantsIndependentFromCilium}$ $
The (co)invariants for the $H$-(co)module structure at a given vertex (face) do not depend on the choice of cilia: moving the position of the cilium yields isomorphic (co)invariants. This induces isomorphisms of the protected objects.
\end{lemma}
\begin{proof}
We  focus on the $H$-module structure and its invariants. We consider a fixed position of the cilium at a vertex $v$ with associated vertex action $\rhd_v$ and coequaliser $\pi_v: H^{\oo E}\to M_v^H$ and compare it to the action $\rhd_v'$ and coequaliser $\pi_v': H^{\oo E}\to M'^H_v$ obtained by  rotating the cilium counterclockwise by one position. We first show that the coequaliser 
$\pi_v: H^{\oo E}\to M_v^H$  satisfies 
\begin{align}\label{eq:rotate}
\pi_v \circ \rhd_v'=\pi_v \circ (\epsilon\oo 1_{H^{\oo E}}).
\end{align}
By definition of the $H$-module structure $\rhd_v$  and by  Lemma \ref{Lemma:EdgeReversalModuleComoduleIso} it is sufficient to prove this for a vertex with $n$ incoming edges. The computations for vertices with incident loops are analogous. 
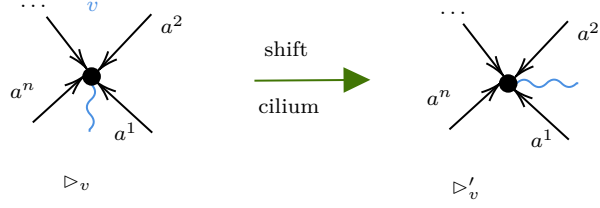
\begin{center}
\tikzset{every picture/.style={line width=0.75pt}}     

\begin{tikzpicture}[x=0.75pt,y=0.75pt,yscale=-1.5,xscale=1.5]
\draw [color={rgb, 255:red, 74; green, 144; blue, 226 }  ,draw opacity=1 ]   (47.47,41.36) .. controls (49.06,43.1) and (48.99,44.77) .. (47.25,46.36) .. controls (45.51,47.95) and (45.44,49.61) .. (47.03,51.35) .. controls (48.62,53.09) and (48.55,54.76) .. (46.81,56.35) -- (46.77,57.15) -- (46.77,57.15) ;
\draw    (32.49,18.58) -- (44.53,34.27) ;
\draw [shift={(45.75,35.86)}, rotate = 232.48] [color={rgb, 255:red, 0; green, 0; blue, 0 }  ][line width=0.75]    (6.56,-1.97) .. controls (4.17,-0.84) and (1.99,-0.18) .. (0,0) .. controls (1.99,0.18) and (4.17,0.84) .. (6.56,1.97)   ;
\draw    (27.91,54.01) -- (43.24,39.96) ;
\draw [shift={(44.71,38.61)}, rotate = 137.49] [color={rgb, 255:red, 0; green, 0; blue, 0 }  ][line width=0.75]    (6.56,-1.97) .. controls (4.17,-0.84) and (1.99,-0.18) .. (0,0) .. controls (1.99,0.18) and (4.17,0.84) .. (6.56,1.97)   ;
\draw  [fill={rgb, 255:red, 0; green, 0; blue, 0 }  ,fill opacity=1 ] (44.71,38.61) .. controls (44.71,37.09) and (45.95,35.86) .. (47.47,35.86) .. controls (48.99,35.86) and (50.22,37.09) .. (50.22,38.61) .. controls (50.22,40.13) and (48.99,41.36) .. (47.47,41.36) .. controls (45.95,41.36) and (44.71,40.13) .. (44.71,38.61) -- cycle ;
\draw    (67.63,57.44) -- (50.83,41.96) ;
\draw [shift={(49.36,40.61)}, rotate = 42.65] [color={rgb, 255:red, 0; green, 0; blue, 0 }  ][line width=0.75]    (6.56,-1.97) .. controls (4.17,-0.84) and (1.99,-0.18) .. (0,0) .. controls (1.99,0.18) and (4.17,0.84) .. (6.56,1.97)   ;
\draw    (67.34,17.72) -- (51.3,34.86) ;
\draw [shift={(49.93,36.32)}, rotate = 313.11] [color={rgb, 255:red, 0; green, 0; blue, 0 }  ][line width=0.75]    (6.56,-1.97) .. controls (4.17,-0.84) and (1.99,-0.18) .. (0,0) .. controls (1.99,0.18) and (4.17,0.84) .. (6.56,1.97)   ;
\draw [color={rgb, 255:red, 65; green, 117; blue, 5 }  ,draw opacity=1 ]   (101.63,40.01) -- (137.2,39.75) ;
\draw [shift={(140.2,39.72)}, rotate = 179.58] [fill={rgb, 255:red, 65; green, 117; blue, 5 }  ,fill opacity=1 ][line width=0.08]  [draw opacity=0] (8.93,-4.29) -- (0,0) -- (8.93,4.29) -- cycle    ;
\draw    (171.06,20.92) -- (183.11,36.61) ;
\draw [shift={(184.32,38.2)}, rotate = 232.48] [color={rgb, 255:red, 0; green, 0; blue, 0 }  ][line width=0.75]    (6.56,-1.97) .. controls (4.17,-0.84) and (1.99,-0.18) .. (0,0) .. controls (1.99,0.18) and (4.17,0.84) .. (6.56,1.97)   ;
\draw    (166.49,56.35) -- (181.81,42.3) ;
\draw [shift={(183.29,40.95)}, rotate = 137.49] [color={rgb, 255:red, 0; green, 0; blue, 0 }  ][line width=0.75]    (6.56,-1.97) .. controls (4.17,-0.84) and (1.99,-0.18) .. (0,0) .. controls (1.99,0.18) and (4.17,0.84) .. (6.56,1.97)   ;
\draw    (206.2,59.78) -- (189.4,44.3) ;
\draw [shift={(187.93,42.95)}, rotate = 42.65] [color={rgb, 255:red, 0; green, 0; blue, 0 }  ][line width=0.75]    (6.56,-1.97) .. controls (4.17,-0.84) and (1.99,-0.18) .. (0,0) .. controls (1.99,0.18) and (4.17,0.84) .. (6.56,1.97)   ;
\draw    (205.91,20.06) -- (189.87,37.2) ;
\draw [shift={(188.5,38.66)}, rotate = 313.11] [color={rgb, 255:red, 0; green, 0; blue, 0 }  ][line width=0.75]    (6.56,-1.97) .. controls (4.17,-0.84) and (1.99,-0.18) .. (0,0) .. controls (1.99,0.18) and (4.17,0.84) .. (6.56,1.97)   ;
\draw [color={rgb, 255:red, 74; green, 144; blue, 226 }  ,draw opacity=1 ]   (188.79,40.95) .. controls (190.45,39.28) and (192.12,39.27) .. (193.79,40.93) .. controls (195.46,42.59) and (197.13,42.58) .. (198.79,40.91) .. controls (200.45,39.24) and (202.12,39.23) .. (203.79,40.89) .. controls (205.46,42.55) and (207.13,42.54) .. (208.79,40.87) -- (209.34,40.87) -- (209.34,40.87) ;
\draw  [fill={rgb, 255:red, 0; green, 0; blue, 0 }  ,fill opacity=1 ] (183.29,40.95) .. controls (183.29,39.43) and (184.52,38.2) .. (186.04,38.2) .. controls (187.56,38.2) and (188.79,39.43) .. (188.79,40.95) .. controls (188.79,42.47) and (187.56,43.7) .. (186.04,43.7) .. controls (184.52,43.7) and (183.29,42.47) .. (183.29,40.95) -- cycle ;
\draw (44.84,12.49) node [anchor=north west][inner sep=0.75pt]  [font=\scriptsize]  {$\textcolor[rgb]{0.29,0.56,0.89}{v}$};
\draw (69.48,17.03) node [anchor=north west][inner sep=0.75pt]  [font=\scriptsize]  {$a^2$};
\draw (22.71,13.65) node [anchor=north west][inner sep=0.75pt]  [font=\scriptsize]  {$\ldots$};
\draw (19.43,39.94) node [anchor=north west][inner sep=0.75pt]  [font=\scriptsize]  {$a^n$};
\draw (54,52.94) node [anchor=north west][inner sep=0.75pt]  [font=\scriptsize]  {$a^1$};
\draw (104,25.86) node [anchor=north west][inner sep=0.75pt]   [align=left] {{\scriptsize shift}};
\draw (102,45) node [anchor=north west][inner sep=0.75pt]   [align=left] {{\scriptsize cilium}};
\draw (208.05,19.37) node [anchor=north west][inner sep=0.75pt]  [font=\scriptsize]  {$a^2$};
\draw (162.29,15.99) node [anchor=north west][inner sep=0.75pt]  [font=\scriptsize]  {$\ldots$};
\draw (158,42.28) node [anchor=north west][inner sep=0.75pt]  [font=\scriptsize]  {$a^n$};
\draw (192.57,55.28) node [anchor=north west][inner sep=0.75pt]  [font=\scriptsize]  {$a^1$};
\draw (36.86,70.8) node [anchor=north west][inner sep=0.75pt]  [font=\scriptsize]  {$\rhd_v$};
\draw (166.29,71.37) node [anchor=north west][inner sep=0.75pt]  [font=\scriptsize]  {$\rhd'_v$};
\end{tikzpicture}
\end{center}
For a vertex with $n$ incoming edges we have
\begin{align*}
&\pi_v \circ \rhd_v'\, (h \oo a^1\oo a^2 \oo \ldots \oo a^n)\, =\, \pi_v\, (\low h n a^1 \oo \low h 1 a^2 \oo \ldots \oo \low h {n-1} a^n)\\
&=\pi_v \, (\low {\low h 2} 1 S(\low h 1 ) \low h 3 a^1 \oo \low {\low h 2} 2 a^2 \oo \ldots \oo \low {\low h 2 } n a^n)\\
&= 
\pi_v \circ \rhd_v \, (\low h 2 \oo S(\low h 1 ) \low h 3 a^1 \oo a^2\oo \ldots \oo a^n)
=\pi_v \, (\epsilon(\low h 2 )\oo S(\low h 1) \low h 3 a^1 \oo  a^2 \oo \ldots \oo a^n)\\ 
&=\pi_v \circ (\epsilon\oo 1_{H^{\oo n}}) \, (h \oo a^1\oo a^2 \oo \ldots \oo a^n),
\end{align*}
where we used first the definition of $\rhd_v'$, then the defining property of the antipode and that $S\circ S=1_H$, then the definition of $\rhd_v$, the fact that $\pi_v$ coequalises $\rhd_v$ and $\epsilon\oo 1_{H^{\oo E}}$ and then again the defining properties of the antipode and the counitality of $H$. 

Inductively, we obtain \eqref{eq:rotate} for all positions of the cilium at $v$ and the same identity with $\pi_v,\rhd_v$ and $\pi'_v,\rhd'_v$ swapped. With the universal property of the coequalisers $\pi_v$, $\pi'_v$ this yields unique morphisms $\phi: M^H_v \to M'^H_v$, $\phi': M'^H_v\to M^H_v$ with $\phi\circ \pi_v=\pi_v'$ and $\phi'\circ \pi_v'=\pi_v$. As $\pi_v$, $\pi_v'$ are epimorphisms,  this implies $\phi'=\phi^\inv$. 
 
 The dual claim for the comodule structure and its coinvariants follows analogously. For all positions of the cilium at $f$ with associated coaction $\delta'_f$, there is a unique morphism $\psi: M^{coH}\to M'^{coH}$ with  $\iota'_f\circ \psi=\iota_f$, and $\psi$ is an isomorphism. Combining these statements for the (co)invariants of all vertices (faces) and using Lemmas
  \ref{Lemma:IsomorphismOfBiinvariants} and 
  \ref{Lemma:RelatingTheInvariants}  yields isomorphisms of the protected objects.
\end{proof}

\subsection{Edge slides and edge contractions}
\label{subsec:contract}
We now consider the edge slides and edge contractions from Definition \ref{def:graphtrafos}. Edge slides were already investigated in \cite{MV}, where it was shown that they define   mapping class group actions. They yield automorphisms of the object $H^{\oo E}$ that are  morphisms of $H$-modules and $H$-comodules as long as no edge ends slide over cilia.

\begin{definition}\label{def:EdgeSlide}\cite[Def.~6.1]{MV}\\
Let $\alpha\neq \beta$ be edges of $\Gamma$ with the starting end of $\alpha$  directly before the target end of $\beta$ in the ordering at $s(\alpha)=t(\beta)$. 
The {\bf edge slide} of the target end of $\beta$ along $\alpha$ corresponds to the isomorphism
$$S_{\alpha,\beta}:= \rhd_{\beta +}\circ \delta_{\alpha +}: H^{\oo E}\to H^{\oo E}\text{ with } S_{\alpha,\beta}^\inv=\rhd_{\beta +}\circ (S \oo 1_{H^{\oo E}})\circ \delta_{\alpha +}: H^{\oo E}\to H^{\oo E}.$$
Edge slides for other edge orientations are defined by reversing edge orientations with the antipode. 
\end{definition}

\begin{example} \label{ex:edgeslide}The isomorphisms induced by the edge slides
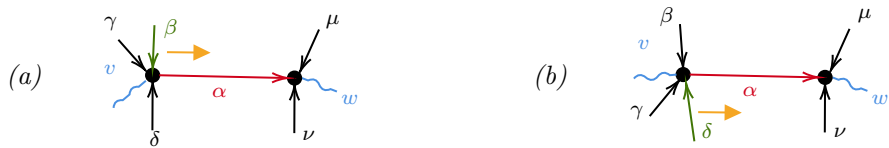
\begin{center}
\tikzset{every picture/.style={line width=0.75pt}}     
\begin{tikzpicture}[x=0.75pt,y=0.75pt,yscale=-1.2,xscale=1.2]
\draw  [fill={rgb, 255:red, 0; green, 0; blue, 0 }  ,fill opacity=1 ] (34.43,45.75) .. controls (34.43,44.23) and (35.66,43) .. (37.18,43) .. controls (38.7,43) and (39.93,44.23) .. (39.93,45.75) .. controls (39.93,47.27) and (38.7,48.5) .. (37.18,48.5) .. controls (35.66,48.5) and (34.43,47.27) .. (34.43,45.75) -- cycle ;
\draw [color={rgb, 255:red, 74; green, 144; blue, 226 }  ,draw opacity=1 ]   (113.79,52.58) -- (111.13,51.55) -- (111.13,51.55) .. controls (108.97,52.5) and (107.42,51.89) .. (106.47,49.74) .. controls (105.52,47.58) and (103.97,46.98) .. (101.81,47.93) -- (99.22,46.93) -- (99.22,46.93) ;
\draw  [fill={rgb, 255:red, 0; green, 0; blue, 0 }  ,fill opacity=1 ] (93.43,46.93) .. controls (93.43,45.41) and (94.66,44.18) .. (96.18,44.18) .. controls (97.7,44.18) and (98.93,45.41) .. (98.93,46.93) .. controls (98.93,48.45) and (97.7,49.68) .. (96.18,49.68) .. controls (94.66,49.68) and (93.43,48.45) .. (93.43,46.93) -- cycle ;
\draw [color={rgb, 255:red, 74; green, 144; blue, 226 }  ,draw opacity=1 ]   (20.93,59.44) .. controls (21.13,57.09) and (22.4,56.02) .. (24.75,56.21) -- (27.58,53.81) -- (27.58,53.81) .. controls (27.78,51.46) and (29.05,50.39) .. (31.4,50.58) -- (34.36,48.07) -- (34.36,48.07) ;
\draw [color={rgb, 255:red, 208; green, 2; blue, 27 }  ,draw opacity=1 ]   (39.93,45.75) -- (83.23,46.7) -- (91.43,46.88) ;
\draw [shift={(93.43,46.93)}, rotate = 181.26] [color={rgb, 255:red, 208; green, 2; blue, 27 }  ,draw opacity=1 ][line width=0.75]    (6.56,-1.97) .. controls (4.17,-0.84) and (1.99,-0.18) .. (0,0) .. controls (1.99,0.18) and (4.17,0.84) .. (6.56,1.97)   ;
\draw    (96.08,70.01) -- (96.17,51.68) ;
\draw [shift={(96.18,49.68)}, rotate = 90.3] [color={rgb, 255:red, 0; green, 0; blue, 0 }  ][line width=0.75]    (6.56,-1.97) .. controls (4.17,-0.84) and (1.99,-0.18) .. (0,0) .. controls (1.99,0.18) and (4.17,0.84) .. (6.56,1.97)   ;
\draw    (106.36,26.58) -- (99.28,42.43) ;
\draw [shift={(98.47,44.25)}, rotate = 294.07] [color={rgb, 255:red, 0; green, 0; blue, 0 }  ][line width=0.75]    (6.56,-1.97) .. controls (4.17,-0.84) and (1.99,-0.18) .. (0,0) .. controls (1.99,0.18) and (4.17,0.84) .. (6.56,1.97)   ;
\draw    (37.08,68.83) -- (37.17,50.5) ;
\draw [shift={(37.18,48.5)}, rotate = 90.3] [color={rgb, 255:red, 0; green, 0; blue, 0 }  ][line width=0.75]    (6.56,-1.97) .. controls (4.17,-0.84) and (1.99,-0.18) .. (0,0) .. controls (1.99,0.18) and (4.17,0.84) .. (6.56,1.97)   ;
\draw    (22.93,30.87) -- (32.86,42.18) ;
\draw [shift={(34.18,43.68)}, rotate = 228.73] [color={rgb, 255:red, 0; green, 0; blue, 0 }  ][line width=0.75]    (6.56,-1.97) .. controls (4.17,-0.84) and (1.99,-0.18) .. (0,0) .. controls (1.99,0.18) and (4.17,0.84) .. (6.56,1.97)   ;
\draw [color={rgb, 255:red, 65; green, 117; blue, 5 }  ,draw opacity=1 ]   (37.93,24.76) -- (37.26,41) ;
\draw [shift={(37.18,43)}, rotate = 272.36] [color={rgb, 255:red, 65; green, 117; blue, 5 }  ,draw opacity=1 ][line width=0.75]    (6.56,-1.97) .. controls (4.17,-0.84) and (1.99,-0.18) .. (0,0) .. controls (1.99,0.18) and (4.17,0.84) .. (6.56,1.97)   ;
\draw [color={rgb, 255:red, 245; green, 166; blue, 35 }  ,draw opacity=1 ]   (43.14,36.07) -- (58.28,36.26) ;
\draw [shift={(61.28,36.3)}, rotate = 180.72] [fill={rgb, 255:red, 245; green, 166; blue, 35 }  ,fill opacity=1 ][line width=0.08]  [draw opacity=0] (7.14,-3.43) -- (0,0) -- (7.14,3.43) -- cycle    ;
\begin{scope}[shift={(90, 0)}]
\draw  [fill={rgb, 255:red, 0; green, 0; blue, 0 }  ,fill opacity=1 ] (165,45.47) .. controls (165,43.95) and (166.23,42.71) .. (167.75,42.71) .. controls (169.27,42.71) and (170.5,43.95) .. (170.5,45.47) .. controls (170.5,46.99) and (169.27,48.22) .. (167.75,48.22) .. controls (166.23,48.22) and (165,46.99) .. (165,45.47) -- cycle ;
\draw [color={rgb, 255:red, 74; green, 144; blue, 226 }  ,draw opacity=1 ]   (244.36,52.3) -- (241.7,51.26) -- (241.7,51.26) .. controls (239.54,52.21) and (237.99,51.61) .. (237.04,49.45) .. controls (236.09,47.3) and (234.53,46.7) .. (232.38,47.65) -- (229.79,46.64) -- (229.79,46.64) ;
\draw  [fill={rgb, 255:red, 0; green, 0; blue, 0 }  ,fill opacity=1 ] (224,46.64) .. controls (224,45.12) and (225.23,43.89) .. (226.75,43.89) .. controls (228.27,43.89) and (229.5,45.12) .. (229.5,46.64) .. controls (229.5,48.16) and (228.27,49.4) .. (226.75,49.4) .. controls (225.23,49.4) and (224,48.16) .. (224,46.64) -- cycle ;
\draw [color={rgb, 255:red, 74; green, 144; blue, 226 }  ,draw opacity=1 ]   (147.28,47.15) .. controls (148.78,45.34) and (150.44,45.18) .. (152.25,46.68) .. controls (154.06,48.18) and (155.72,48.02) .. (157.23,46.21) .. controls (158.73,44.39) and (160.39,44.23) .. (162.21,45.73) -- (165,45.47) -- (165,45.47) ;
\draw [color={rgb, 255:red, 208; green, 2; blue, 27 }  ,draw opacity=1 ]   (170.5,45.47) -- (213.8,46.42) -- (222,46.6) ;
\draw [shift={(224,46.64)}, rotate = 181.26] [color={rgb, 255:red, 208; green, 2; blue, 27 }  ,draw opacity=1 ][line width=0.75]    (6.56,-1.97) .. controls (4.17,-0.84) and (1.99,-0.18) .. (0,0) .. controls (1.99,0.18) and (4.17,0.84) .. (6.56,1.97)   ;
\draw    (226.65,69.72) -- (226.74,51.4) ;
\draw [shift={(226.75,49.4)}, rotate = 90.3] [color={rgb, 255:red, 0; green, 0; blue, 0 }  ][line width=0.75]    (6.56,-1.97) .. controls (4.17,-0.84) and (1.99,-0.18) .. (0,0) .. controls (1.99,0.18) and (4.17,0.84) .. (6.56,1.97)   ;
\draw    (236.93,26.3) -- (229.85,42.14) ;
\draw [shift={(229.04,43.97)}, rotate = 294.07] [color={rgb, 255:red, 0; green, 0; blue, 0 }  ][line width=0.75]    (6.56,-1.97) .. controls (4.17,-0.84) and (1.99,-0.18) .. (0,0) .. controls (1.99,0.18) and (4.17,0.84) .. (6.56,1.97)   ;
\draw    (153.85,62.87) -- (165.06,49.19) ;
\draw [shift={(166.32,47.65)}, rotate = 129.34] [color={rgb, 255:red, 0; green, 0; blue, 0 }  ][line width=0.75]    (6.56,-1.97) .. controls (4.17,-0.84) and (1.99,-0.18) .. (0,0) .. controls (1.99,0.18) and (4.17,0.84) .. (6.56,1.97)   ;
\draw    (166.42,24.3) -- (167.61,40.72) ;
\draw [shift={(167.75,42.71)}, rotate = 265.86] [color={rgb, 255:red, 0; green, 0; blue, 0 }  ][line width=0.75]    (6.56,-1.97) .. controls (4.17,-0.84) and (1.99,-0.18) .. (0,0) .. controls (1.99,0.18) and (4.17,0.84) .. (6.56,1.97)   ;
\draw [color={rgb, 255:red, 65; green, 117; blue, 5 }  ,draw opacity=1 ]   (172.42,73.44) -- (169.27,51.13) ;
\draw [shift={(168.99,49.15)}, rotate = 81.96] [color={rgb, 255:red, 65; green, 117; blue, 5 }  ,draw opacity=1 ][line width=0.75]    (6.56,-1.97) .. controls (4.17,-0.84) and (1.99,-0.18) .. (0,0) .. controls (1.99,0.18) and (4.17,0.84) .. (6.56,1.97)   ;
\draw [color={rgb, 255:red, 245; green, 166; blue, 35 }  ,draw opacity=1 ]   (174,61.21) -- (189.13,61.4) ;
\draw [shift={(192.13,61.44)}, rotate = 180.72] [fill={rgb, 255:red, 245; green, 166; blue, 35 }  ,fill opacity=1 ][line width=0.08]  [draw opacity=0] (7.14,-3.43) -- (0,0) -- (7.14,3.43) -- cycle    ;
\end{scope}
\draw (60.71,49.72) node [anchor=north west][inner sep=0.75pt]  [font=\scriptsize]  {$\textcolor[rgb]{0.82,0.01,0.11}{\alpha }$};
\draw (15.71,40.26) node [anchor=north west][inner sep=0.75pt]  [font=\scriptsize]  {$\textcolor[rgb]{0.29,0.56,0.89}{v}$};
\draw (40.57,21.11) node [anchor=north west][inner sep=0.75pt]  [font=\scriptsize]  {$\textcolor[rgb]{0.25,0.46,0.02}{\beta}$};
\draw (15.71,19.97) node [anchor=north west][inner sep=0.75pt]  [font=\scriptsize]  {$\gamma$};
\draw (114.43,52.83) node [anchor=north west][inner sep=0.75pt]  [font=\scriptsize]  {$\textcolor[rgb]{0.29,0.56,0.89}{w}$};
\draw (34.57,67.97) node [anchor=north west][inner sep=0.75pt]  [font=\scriptsize]  {$\delta$};
\draw (108,19.11) node [anchor=north west][inner sep=0.75pt]  [font=\scriptsize]  {$\mu$};
\draw (98.29,67.69) node [anchor=north west][inner sep=0.75pt]  [font=\scriptsize]  {$\nu$};
\draw (-25, 40) node [anchor=north west][inner sep=0.75pt]  [font=\small]  {(a)};
\begin{scope}[shift={(90, 0)}]
\draw (191.29,49.44) node [anchor=north west][inner sep=0.75pt]  [font=\scriptsize]  {$\textcolor[rgb]{0.82,0.01,0.11}{\alpha }$};
\draw (147.71,29.69) node [anchor=north west][inner sep=0.75pt]  [font=\scriptsize]  {$\textcolor[rgb]{0.29,0.56,0.89}{v}$};
\draw (174.29,64.61) node [anchor=north west][inner sep=0.75pt]  [font=\scriptsize]  {$\textcolor[rgb]{0.25,0.46,0.02}{\delta}$};
\draw (156.86,15.97) node [anchor=north west][inner sep=0.75pt]  [font=\scriptsize]  {$\beta$};
\draw (245,52.54) node [anchor=north west][inner sep=0.75pt]  [font=\scriptsize]  {$\textcolor[rgb]{0.29,0.56,0.89}{w}$};
\draw (144.57,57.4) node [anchor=north west][inner sep=0.75pt]  [font=\scriptsize]  {$\gamma$};
\draw (239.14,16.54) node [anchor=north west][inner sep=0.75pt]  [font=\scriptsize]  {$\mu$};
\draw (229.14,66.54) node [anchor=north west][inner sep=0.75pt]  [font=\scriptsize]  {$\nu$};
\draw (104, 40) node [anchor=north west][inner sep=0.75pt]  [font=\small]  {(b)};
\end{scope}
\end{tikzpicture}
\end{center}
are obtained by (a) applying Definition \ref{def:EdgeSlide}  and (b) first reversing the orientation of $\alpha$,  applying the inverse edge slide from Definition \ref{def:EdgeSlide} and then reversing the orientation of $\alpha$. This yields
\begin{align*}
&(a) &&S_{\alpha,\beta}(\alpha\oo \beta\oo \gamma \oo \delta\oo \mu\oo \nu)=\low \alpha 2\oo \low \alpha 1 \beta\oo \gamma \oo \delta\oo \mu\oo \nu,\\
&(b) & &S_{\alpha,\beta}(\alpha\oo \beta\oo \gamma \oo \delta\oo \mu\oo \nu)=\low \alpha 1\oo  \beta\oo \gamma \oo \low \alpha 2\delta\oo \mu\oo \nu.\\
\end{align*}
\end{example}

By construction, edge slides affect only the two copies of $H$  in $H^{\oo E}$ of the edges involved in the slide and  commute with  edge orientation reversals. 
Moreover, they respect the module and comodule structures at vertices and faces and hence induce isomorphisms between the  protected objects.

\begin{proposition}\label{Prop:EdgeSlideIsIso}\cite[Prop.~6.2]{MV}\\
Let $v$ and $f$ be a ciliated vertex and face in a ribbon graph $\Gamma$ with associated $H$-module structure $\rhd_v$ and $H$-comodule structure $\delta_f$. Any edge slide that does not slide edge ends over their cilia is an isomorphism of $H$-left modules and $H$-left comodules with respect to $\rhd_v$ and  $\delta_f$.
\end{proposition}

\begin{corollary}\label{Coro:EdgeSlideIsoOnProtectedObject} $ $\\
Edge slides 
 from a ribbon graph $\Gamma$ to a ribbon graph $\Gamma'$ induce isomorphisms between the invariants, coinvariants and  protected objects of $\Gamma$ and $\Gamma'$. 
\end{corollary}

\begin{proof}For edge slides that do not slide edge ends over cilia, this follows directly from Lemmas \ref{Lemma:IsosAufInvarianten},  
 \ref{Lemma:IsomorphismOfBiinvariants} and Proposition \ref{Prop:EdgeSlideIsIso}. If an edge end slides over a cilium, we can apply Lemma \ref{lemma:InvariantsIndependentFromCilium} to move the cilium and obtain the same  result.
\end{proof}

We now consider edge contractions. Recall from  Definition \ref{def:graphtrafos} that an edge $\alpha\in E$ may only be contracted if its starting and target vertex differ and that contracting $\alpha$ towards  $v\in\{s(\alpha),t(\alpha)\}$ erases the cilium at $v$, while the cilium at the other vertex is preserved. 

\begin{definition}\label{def:EdgeContraction} The morphism $c_{\alpha,v}: H^{\oo E}\to H^{\oo (E-1)}$ induced by an {\bf edge contraction} of an edge $\alpha$ towards $v\in \{s(\alpha), t(\alpha)\}$ is
\begin{align*}
c_{\alpha,v}=\begin{cases}
\rhd_{v, \alpha}\circ \tau_{\alpha}\circ S_{\alpha} &\text{ if } v=t(\alpha)\\
\rhd_{v, \alpha} \circ \tau_{\alpha} &\text{ if } v=s(\alpha)
\end{cases}
\end{align*}
 where $\rhd_{v, \alpha}: H^{\oo E}\to H^{\oo (E-1)}$ denotes the $H$-module structure  from Definition \ref{def:vertexfaceops} at $v$, where $\alpha$ is replaced by a cilium and $\tau_\alpha$ is given before Definition \ref{def:TraingleCoActions}. If $v$ is univalent, then $c_{\alpha,v}=\epsilon_{\alpha}$.
\end{definition}

\begin{example}\label{ex:EdgeContraction}
Contracting the edge $\alpha$ towards $v$ in 
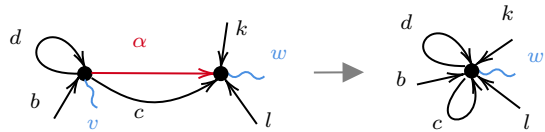
\begin{center}
\tikzset{every picture/.style={line width=0.75pt}}       
\begin{tikzpicture}[x=0.75pt,y=0.75pt,yscale=-1.2,xscale=1.2]
\draw  [fill={rgb, 255:red, 0; green, 0; blue, 0 }  ,fill opacity=1 ] (44.71,38.61) .. controls (44.71,37.09) and (45.95,35.86) .. (47.47,35.86) .. controls (48.99,35.86) and (50.22,37.09) .. (50.22,38.61) .. controls (50.22,40.13) and (48.99,41.36) .. (47.47,41.36) .. controls (45.95,41.36) and (44.71,40.13) .. (44.71,38.61) -- cycle ;
\draw [color={rgb, 255:red, 74; green, 144; blue, 226 }  ,draw opacity=1 ]   (122.13,39.72) .. controls (120.34,41.26) and (118.68,41.14) .. (117.15,39.35) .. controls (115.61,37.56) and (113.95,37.44) .. (112.16,38.98) .. controls (110.37,40.52) and (108.71,40.4) .. (107.17,38.61) -- (106.78,38.58) -- (106.78,38.58) ;
\draw [color={rgb, 255:red, 208; green, 2; blue, 27 }  ,draw opacity=1 ]   (50.22,38.32) -- (99.28,38.57) ;
\draw [shift={(101.28,38.58)}, rotate = 180.29] [color={rgb, 255:red, 208; green, 2; blue, 27 }  ,draw opacity=1 ][line width=0.75]    (6.56,-1.97) .. controls (4.17,-0.84) and (1.99,-0.18) .. (0,0) .. controls (1.99,0.18) and (4.17,0.84) .. (6.56,1.97)   ;
\draw    (118.99,59.72) -- (107.18,42.97) ;
\draw [shift={(106.03,41.33)}, rotate = 54.82] [color={rgb, 255:red, 0; green, 0; blue, 0 }  ][line width=0.75]    (6.56,-1.97) .. controls (4.17,-0.84) and (1.99,-0.18) .. (0,0) .. controls (1.99,0.18) and (4.17,0.84) .. (6.56,1.97)   ;
\draw    (34.7,57.44) -- (43.8,41.98) ;
\draw [shift={(44.81,40.25)}, rotate = 120.45] [color={rgb, 255:red, 0; green, 0; blue, 0 }  ][line width=0.75]    (6.56,-1.97) .. controls (4.17,-0.84) and (1.99,-0.18) .. (0,0) .. controls (1.99,0.18) and (4.17,0.84) .. (6.56,1.97)   ;
\draw  [fill={rgb, 255:red, 0; green, 0; blue, 0 }  ,fill opacity=1 ] (101.28,38.58) .. controls (101.28,37.06) and (102.51,35.83) .. (104.03,35.83) .. controls (105.55,35.83) and (106.78,37.06) .. (106.78,38.58) .. controls (106.78,40.1) and (105.55,41.33) .. (104.03,41.33) .. controls (102.51,41.33) and (101.28,40.1) .. (101.28,38.58) -- cycle ;
\draw    (106.7,17.15) -- (104.31,33.85) ;
\draw [shift={(104.03,35.83)}, rotate = 278.15] [color={rgb, 255:red, 0; green, 0; blue, 0 }  ][line width=0.75]    (6.56,-1.97) .. controls (4.17,-0.84) and (1.99,-0.18) .. (0,0) .. controls (1.99,0.18) and (4.17,0.84) .. (6.56,1.97)   ;
\draw [color={rgb, 255:red, 74; green, 144; blue, 226 }  ,draw opacity=1 ]   (52.7,52.87) .. controls (50.5,52.02) and (49.82,50.5) .. (50.67,48.3) .. controls (51.52,46.1) and (50.84,44.58) .. (48.64,43.73) -- (47.47,41.08) -- (47.47,41.08) ;
\draw    (44.71,38.61) .. controls (11.95,38.89) and (32.97,7.31) .. (46.64,34.14) ;
\draw [shift={(47.47,35.86)}, rotate = 245.47] [color={rgb, 255:red, 0; green, 0; blue, 0 }  ][line width=0.75]    (6.56,-1.97) .. controls (4.17,-0.84) and (1.99,-0.18) .. (0,0) .. controls (1.99,0.18) and (4.17,0.84) .. (6.56,1.97)   ;
\draw    (50.29,40.43) .. controls (64.98,49.82) and (76.79,56.87) .. (100.93,42.08) ;
\draw [shift={(102.42,41.15)}, rotate = 147.53] [color={rgb, 255:red, 0; green, 0; blue, 0 }  ][line width=0.75]    (6.56,-1.97) .. controls (4.17,-0.84) and (1.99,-0.18) .. (0,0) .. controls (1.99,0.18) and (4.17,0.84) .. (6.56,1.97)   ;
\draw [color={rgb, 255:red, 74; green, 144; blue, 226 }  ,draw opacity=1 ]   (226.42,38.61) .. controls (224.63,40.15) and (222.97,40.03) .. (221.43,38.24) .. controls (219.9,36.45) and (218.24,36.33) .. (216.45,37.87) .. controls (214.66,39.41) and (213,39.29) .. (211.46,37.5) -- (211.07,37.47) -- (211.07,37.47) ;
\draw    (228.42,53.15) -- (212.76,41.15) ;
\draw [shift={(211.17,39.94)}, rotate = 37.46] [color={rgb, 255:red, 0; green, 0; blue, 0 }  ][line width=0.75]    (6.56,-1.97) .. controls (4.17,-0.84) and (1.99,-0.18) .. (0,0) .. controls (1.99,0.18) and (4.17,0.84) .. (6.56,1.97)   ;
\draw    (185.56,43.15) -- (202.86,39.29) ;
\draw [shift={(204.81,38.86)}, rotate = 167.42] [color={rgb, 255:red, 0; green, 0; blue, 0 }  ][line width=0.75]    (6.56,-1.97) .. controls (4.17,-0.84) and (1.99,-0.18) .. (0,0) .. controls (1.99,0.18) and (4.17,0.84) .. (6.56,1.97)   ;
\draw  [fill={rgb, 255:red, 0; green, 0; blue, 0 }  ,fill opacity=1 ] (205.56,37.47) .. controls (205.56,35.95) and (206.79,34.72) .. (208.31,34.72) .. controls (209.83,34.72) and (211.07,35.95) .. (211.07,37.47) .. controls (211.07,38.99) and (209.83,40.22) .. (208.31,40.22) .. controls (206.79,40.22) and (205.56,38.99) .. (205.56,37.47) -- cycle ;
\draw    (225.28,24.3) -- (212.25,33.3) ;
\draw [shift={(210.6,34.43)}, rotate = 325.36] [color={rgb, 255:red, 0; green, 0; blue, 0 }  ][line width=0.75]    (6.56,-1.97) .. controls (4.17,-0.84) and (1.99,-0.18) .. (0,0) .. controls (1.99,0.18) and (4.17,0.84) .. (6.56,1.97)   ;
\draw    (205.56,35.44) .. controls (172.8,35.72) and (193.82,6.09) .. (207.49,33) ;
\draw [shift={(208.31,34.72)}, rotate = 245.47] [color={rgb, 255:red, 0; green, 0; blue, 0 }  ][line width=0.75]    (6.56,-1.97) .. controls (4.17,-0.84) and (1.99,-0.18) .. (0,0) .. controls (1.99,0.18) and (4.17,0.84) .. (6.56,1.97)   ;
\draw    (205.56,40.58) .. controls (185.05,64.05) and (215.67,62.67) .. (208.91,41.86) ;
\draw [shift={(208.31,40.22)}, rotate = 68.16] [color={rgb, 255:red, 0; green, 0; blue, 0 }  ][line width=0.75]    (6.56,-1.97) .. controls (4.17,-0.84) and (1.99,-0.18) .. (0,0) .. controls (1.99,0.18) and (4.17,0.84) .. (6.56,1.97)   ;
\draw [color={rgb, 255:red, 128; green, 128; blue, 128 }  ,draw opacity=1 ]   (142.7,38.13) -- (160.56,38.03) ;
\draw [shift={(163.56,38.01)}, rotate = 179.67] [fill={rgb, 255:red, 128; green, 128; blue, 128 }  ,fill opacity=1 ][line width=0.08]  [draw opacity=0] (8.93,-4.29) -- (0,0) -- (8.93,4.29) -- cycle    ;
\draw (65.99,22.49) node [anchor=north west][inner sep=0.75pt]  [font=\scriptsize]  {$\textcolor[rgb]{0.82,0.01,0.11}{\alpha }$};
\draw (47.14,56.83) node [anchor=north west][inner sep=0.75pt]  [font=\scriptsize]  {$\textcolor[rgb]{0.29,0.56,0.89}{v}$};
\draw (23.43,45.65) node [anchor=north west][inner sep=0.75pt]  [font=\scriptsize]  {$b$};
\draw (66.57,51.94) node [anchor=north west][inner sep=0.75pt]  [font=\scriptsize]  {$c$};
\draw (14.86,18.8) node [anchor=north west][inner sep=0.75pt]  [font=\scriptsize]  {$d$};
\draw (108.86,14.8) node [anchor=north west][inner sep=0.75pt]  [font=\scriptsize]  {$k$};
\draw (123.43,27.4) node [anchor=north west][inner sep=0.75pt]  [font=\scriptsize]  {$\textcolor[rgb]{0.29,0.56,0.89}{w}$};
\draw (121.14,53.69) node [anchor=north west][inner sep=0.75pt]  [font=\scriptsize]  {$l$};
\draw (175.71,36.54) node [anchor=north west][inner sep=0.75pt]  [font=\scriptsize]  {$b$};
\draw (190.57,56.54) node [anchor=north west][inner sep=0.75pt]  [font=\scriptsize]  {$c$};
\draw (177.43,12.26) node [anchor=north west][inner sep=0.75pt]  [font=\scriptsize]  {$d$};
\draw (219.43,8.83) node [anchor=north west][inner sep=0.75pt]  [font=\scriptsize]  {$k$};
\draw (228.57,54.26) node [anchor=north west][inner sep=0.75pt]  [font=\scriptsize]  {$l$};
\draw (230.29,28.83) node [anchor=north west][inner sep=0.75pt]  [font=\scriptsize]  {$\textcolor[rgb]{0.29,0.56,0.89}{w}$};
\end{tikzpicture}
\end{center}
gives the morphism $c_{\alpha,v}$ with
\begin{align*}
c_{\alpha,v}\, (\alpha \oo b\oo c\oo d\oo k \oo l)= \low \alpha 3 b \oo c S(\low \alpha 4 ) \oo \low \alpha 1 d S(\low \alpha 2 ) \oo k \oo l.
\end{align*}
\end{example}

It follows directly from Definition \ref{def:EdgeContraction} that first reversing the orientation of an edge $\beta$ and then contracting it is the same as just contracting $\beta$. It
also follows from  Definitions \ref{def:TraingleCoActions} and \ref{def:vertexfaceops}
that reversing the orientation of an edge $\beta$ commutes with contractions of all edges $\alpha\neq \beta$.
The contraction of an edge $\alpha$ also commutes with edge slides along $\alpha$, which allows one to express any edge contraction as a composite of  edge slides and an edge contraction towards a univalent vertex.

\begin{lemma}\label{lemma:EdgeContractionReversalCommutes}
Let $\Gamma'$ be obtained by reversing an edge  $\beta$  in $\Gamma$. Then  
\begin{align}
&c'_{\beta, v}\circ S_\beta=c_{\beta, v} & &c_{\alpha,v}'\circ S_{\beta}= S_{\beta}\circ c_{\alpha,v}\quad \text{for }\alpha\neq\beta. \label{eq:EdgeContractionEdgeReversalCommute}
\end{align}
\end{lemma}

\begin{lemma}\label{Lemma:EdgeContractionAndSlidesCompatible}
Contracting an edge $\alpha$  gives the same morphism as  first sliding  edge ends  along $\alpha$ and then contracting $\alpha$.
\end{lemma}
\begin{proof}
It suffices to slide a single edge end along $\alpha$, as the statement follows inductively. We denote by $c_{\alpha, v}$ the contraction of $\alpha$ in $\Gamma$ and by $c_{\alpha,v}'$ the contraction of $\alpha$ in the graph $\Gamma'$ obtained by sliding  an edge $b$ along $\alpha$. Suppose that there are no loops incident at $s(\alpha)$ and $t(\alpha)$ in $\Gamma$ and $\Gamma'$.  As edge slides and edge contractions commute with edge reversals by  Definition \ref{def:EdgeSlide}  and   Lemma \ref{lemma:EdgeContractionReversalCommutes}, respectively,  we can  assume   $v=s(\alpha)$ and all other edge ends at $v$ and $w=t(\alpha)$ are incoming. It is then sufficient to consider an edge slide of $b$ along the left and right of $\alpha$: 
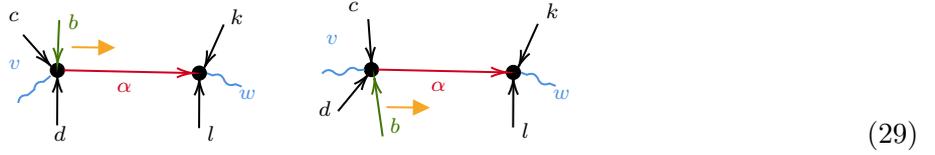
\begin{align}\label{eq:slidecontract}
\tikzset{every picture/.style={line width=0.75pt}} 
\begin{tikzpicture}[x=0.75pt,y=0.75pt,yscale=-1.2,xscale=1.2]
\draw  [fill={rgb, 255:red, 0; green, 0; blue, 0 }  ,fill opacity=1 ] (34.43,45.75) .. controls (34.43,44.23) and (35.66,43) .. (37.18,43) .. controls (38.7,43) and (39.93,44.23) .. (39.93,45.75) .. controls (39.93,47.27) and (38.7,48.5) .. (37.18,48.5) .. controls (35.66,48.5) and (34.43,47.27) .. (34.43,45.75) -- cycle ;
\draw [color={rgb, 255:red, 74; green, 144; blue, 226 }  ,draw opacity=1 ]   (113.79,52.58) -- (111.13,51.55) -- (111.13,51.55) .. controls (108.97,52.5) and (107.42,51.89) .. (106.47,49.74) .. controls (105.52,47.58) and (103.97,46.98) .. (101.81,47.93) -- (99.22,46.93) -- (99.22,46.93) ;
\draw  [fill={rgb, 255:red, 0; green, 0; blue, 0 }  ,fill opacity=1 ] (93.43,46.93) .. controls (93.43,45.41) and (94.66,44.18) .. (96.18,44.18) .. controls (97.7,44.18) and (98.93,45.41) .. (98.93,46.93) .. controls (98.93,48.45) and (97.7,49.68) .. (96.18,49.68) .. controls (94.66,49.68) and (93.43,48.45) .. (93.43,46.93) -- cycle ;
\draw [color={rgb, 255:red, 74; green, 144; blue, 226 }  ,draw opacity=1 ]   (20.93,59.44) .. controls (21.13,57.09) and (22.4,56.02) .. (24.75,56.21) -- (27.58,53.81) -- (27.58,53.81) .. controls (27.78,51.46) and (29.05,50.39) .. (31.4,50.58) -- (34.36,48.07) -- (34.36,48.07) ;
\draw [color={rgb, 255:red, 208; green, 2; blue, 27 }  ,draw opacity=1 ]   (39.93,45.75) -- (83.23,46.7) -- (91.43,46.88) ;
\draw [shift={(93.43,46.93)}, rotate = 181.26] [color={rgb, 255:red, 208; green, 2; blue, 27 }  ,draw opacity=1 ][line width=0.75]    (6.56,-1.97) .. controls (4.17,-0.84) and (1.99,-0.18) .. (0,0) .. controls (1.99,0.18) and (4.17,0.84) .. (6.56,1.97)   ;
\draw    (96.08,70.01) -- (96.17,51.68) ;
\draw [shift={(96.18,49.68)}, rotate = 90.3] [color={rgb, 255:red, 0; green, 0; blue, 0 }  ][line width=0.75]    (6.56,-1.97) .. controls (4.17,-0.84) and (1.99,-0.18) .. (0,0) .. controls (1.99,0.18) and (4.17,0.84) .. (6.56,1.97)   ;
\draw    (106.36,26.58) -- (99.28,42.43) ;
\draw [shift={(98.47,44.25)}, rotate = 294.07] [color={rgb, 255:red, 0; green, 0; blue, 0 }  ][line width=0.75]    (6.56,-1.97) .. controls (4.17,-0.84) and (1.99,-0.18) .. (0,0) .. controls (1.99,0.18) and (4.17,0.84) .. (6.56,1.97)   ;
\draw    (37.08,68.83) -- (37.17,50.5) ;
\draw [shift={(37.18,48.5)}, rotate = 90.3] [color={rgb, 255:red, 0; green, 0; blue, 0 }  ][line width=0.75]    (6.56,-1.97) .. controls (4.17,-0.84) and (1.99,-0.18) .. (0,0) .. controls (1.99,0.18) and (4.17,0.84) .. (6.56,1.97)   ;
\draw    (22.93,30.87) -- (32.86,42.18) ;
\draw [shift={(34.18,43.68)}, rotate = 228.73] [color={rgb, 255:red, 0; green, 0; blue, 0 }  ][line width=0.75]    (6.56,-1.97) .. controls (4.17,-0.84) and (1.99,-0.18) .. (0,0) .. controls (1.99,0.18) and (4.17,0.84) .. (6.56,1.97)   ;
\draw [color={rgb, 255:red, 65; green, 117; blue, 5 }  ,draw opacity=1 ]   (37.93,24.76) -- (37.26,41) ;
\draw [shift={(37.18,43)}, rotate = 272.36] [color={rgb, 255:red, 65; green, 117; blue, 5 }  ,draw opacity=1 ][line width=0.75]    (6.56,-1.97) .. controls (4.17,-0.84) and (1.99,-0.18) .. (0,0) .. controls (1.99,0.18) and (4.17,0.84) .. (6.56,1.97)   ;
\draw [color={rgb, 255:red, 245; green, 166; blue, 35 }  ,draw opacity=1 ]   (43.14,36.07) -- (58.28,36.26) ;
\draw [shift={(61.28,36.3)}, rotate = 180.72] [fill={rgb, 255:red, 245; green, 166; blue, 35 }  ,fill opacity=1 ][line width=0.08]  [draw opacity=0] (7.14,-3.43) -- (0,0) -- (7.14,3.43) -- cycle    ;
\draw  [fill={rgb, 255:red, 0; green, 0; blue, 0 }  ,fill opacity=1 ] (165,45.47) .. controls (165,43.95) and (166.23,42.71) .. (167.75,42.71) .. controls (169.27,42.71) and (170.5,43.95) .. (170.5,45.47) .. controls (170.5,46.99) and (169.27,48.22) .. (167.75,48.22) .. controls (166.23,48.22) and (165,46.99) .. (165,45.47) -- cycle ;
\draw [color={rgb, 255:red, 74; green, 144; blue, 226 }  ,draw opacity=1 ]   (244.36,52.3) -- (241.7,51.26) -- (241.7,51.26) .. controls (239.54,52.21) and (237.99,51.61) .. (237.04,49.45) .. controls (236.09,47.3) and (234.53,46.7) .. (232.38,47.65) -- (229.79,46.64) -- (229.79,46.64) ;
\draw  [fill={rgb, 255:red, 0; green, 0; blue, 0 }  ,fill opacity=1 ] (224,46.64) .. controls (224,45.12) and (225.23,43.89) .. (226.75,43.89) .. controls (228.27,43.89) and (229.5,45.12) .. (229.5,46.64) .. controls (229.5,48.16) and (228.27,49.4) .. (226.75,49.4) .. controls (225.23,49.4) and (224,48.16) .. (224,46.64) -- cycle ;
\draw [color={rgb, 255:red, 74; green, 144; blue, 226 }  ,draw opacity=1 ]   (147.28,47.15) .. controls (148.78,45.34) and (150.44,45.18) .. (152.25,46.68) .. controls (154.06,48.18) and (155.72,48.02) .. (157.23,46.21) .. controls (158.73,44.39) and (160.39,44.23) .. (162.21,45.73) -- (165,45.47) -- (165,45.47) ;
\draw [color={rgb, 255:red, 208; green, 2; blue, 27 }  ,draw opacity=1 ]   (170.5,45.47) -- (213.8,46.42) -- (222,46.6) ;
\draw [shift={(224,46.64)}, rotate = 181.26] [color={rgb, 255:red, 208; green, 2; blue, 27 }  ,draw opacity=1 ][line width=0.75]    (6.56,-1.97) .. controls (4.17,-0.84) and (1.99,-0.18) .. (0,0) .. controls (1.99,0.18) and (4.17,0.84) .. (6.56,1.97)   ;
\draw    (226.65,69.72) -- (226.74,51.4) ;
\draw [shift={(226.75,49.4)}, rotate = 90.3] [color={rgb, 255:red, 0; green, 0; blue, 0 }  ][line width=0.75]    (6.56,-1.97) .. controls (4.17,-0.84) and (1.99,-0.18) .. (0,0) .. controls (1.99,0.18) and (4.17,0.84) .. (6.56,1.97)   ;
\draw    (236.93,26.3) -- (229.85,42.14) ;
\draw [shift={(229.04,43.97)}, rotate = 294.07] [color={rgb, 255:red, 0; green, 0; blue, 0 }  ][line width=0.75]    (6.56,-1.97) .. controls (4.17,-0.84) and (1.99,-0.18) .. (0,0) .. controls (1.99,0.18) and (4.17,0.84) .. (6.56,1.97)   ;
\draw    (153.85,62.87) -- (165.06,49.19) ;
\draw [shift={(166.32,47.65)}, rotate = 129.34] [color={rgb, 255:red, 0; green, 0; blue, 0 }  ][line width=0.75]    (6.56,-1.97) .. controls (4.17,-0.84) and (1.99,-0.18) .. (0,0) .. controls (1.99,0.18) and (4.17,0.84) .. (6.56,1.97)   ;
\draw    (166.42,24.3) -- (167.61,40.72) ;
\draw [shift={(167.75,42.71)}, rotate = 265.86] [color={rgb, 255:red, 0; green, 0; blue, 0 }  ][line width=0.75]    (6.56,-1.97) .. controls (4.17,-0.84) and (1.99,-0.18) .. (0,0) .. controls (1.99,0.18) and (4.17,0.84) .. (6.56,1.97)   ;
\draw [color={rgb, 255:red, 65; green, 117; blue, 5 }  ,draw opacity=1 ]   (172.42,73.44) -- (169.27,51.13) ;
\draw [shift={(168.99,49.15)}, rotate = 81.96] [color={rgb, 255:red, 65; green, 117; blue, 5 }  ,draw opacity=1 ][line width=0.75]    (6.56,-1.97) .. controls (4.17,-0.84) and (1.99,-0.18) .. (0,0) .. controls (1.99,0.18) and (4.17,0.84) .. (6.56,1.97)   ;
\draw [color={rgb, 255:red, 245; green, 166; blue, 35 }  ,draw opacity=1 ]   (174,61.21) -- (189.13,61.4) ;
\draw [shift={(192.13,61.44)}, rotate = 180.72] [fill={rgb, 255:red, 245; green, 166; blue, 35 }  ,fill opacity=1 ][line width=0.08]  [draw opacity=0] (7.14,-3.43) -- (0,0) -- (7.14,3.43) -- cycle    ;
\draw (60.71,49.72) node [anchor=north west][inner sep=0.75pt]  [font=\scriptsize]  {$\textcolor[rgb]{0.82,0.01,0.11}{\alpha }$};
\draw (15.71,40.26) node [anchor=north west][inner sep=0.75pt]  [font=\scriptsize]  {$\textcolor[rgb]{0.29,0.56,0.89}{v}$};
\draw (40.57,21.11) node [anchor=north west][inner sep=0.75pt]  [font=\scriptsize]  {$\textcolor[rgb]{0.25,0.46,0.02}{b}$};
\draw (15.71,19.97) node [anchor=north west][inner sep=0.75pt]  [font=\scriptsize]  {$c$};
\draw (111.43,52.83) node [anchor=north west][inner sep=0.75pt]  [font=\scriptsize]  {$\textcolor[rgb]{0.29,0.56,0.89}{w}$};
\draw (34.57,67.97) node [anchor=north west][inner sep=0.75pt]  [font=\scriptsize]  {$d$};
\draw (108,19.11) node [anchor=north west][inner sep=0.75pt]  [font=\scriptsize]  {$k$};
\draw (98.29,67.69) node [anchor=north west][inner sep=0.75pt]  [font=\scriptsize]  {$l$};
\draw (191.29,49.44) node [anchor=north west][inner sep=0.75pt]  [font=\scriptsize]  {$\textcolor[rgb]{0.82,0.01,0.11}{\alpha }$};
\draw (147.71,29.69) node [anchor=north west][inner sep=0.75pt]  [font=\scriptsize]  {$\textcolor[rgb]{0.29,0.56,0.89}{v}$};
\draw (174.29,64.61) node [anchor=north west][inner sep=0.75pt]  [font=\scriptsize]  {$\textcolor[rgb]{0.25,0.46,0.02}{b}$};
\draw (156.86,15.97) node [anchor=north west][inner sep=0.75pt]  [font=\scriptsize]  {$c$};
\draw (242,52.54) node [anchor=north west][inner sep=0.75pt]  [font=\scriptsize]  {$\textcolor[rgb]{0.29,0.56,0.89}{w}$};
\draw (144.57,57.4) node [anchor=north west][inner sep=0.75pt]  [font=\scriptsize]  {$d$};
\draw (239.14,16.54) node [anchor=north west][inner sep=0.75pt]  [font=\scriptsize]  {$k$};
\draw (229.14,66.54) node [anchor=north west][inner sep=0.75pt]  [font=\scriptsize]  {$l$};
\end{tikzpicture}
\end{align}
Omitting the copies of $H$ for edges not incident at $v,w$ we compute for the edge slides in \eqref{eq:slidecontract}
\begin{align*}
c_{\alpha,v}'\circ S_{\alpha,b} (\alpha \oo b \oo c \oo d \oo k \oo l) &= c_{\alpha,v}'(\low \alpha 2 \oo \low \alpha 1 b \oo c \oo d \oo k \oo l) \\
&= \low \alpha 1 b \oo \low \alpha 2 c \oo \low \alpha 3 d \oo k \oo l \,= \, c_{\alpha,v}(\alpha \oo b \oo c\oo d\oo k \oo l),\\
c_{\alpha,v}'\circ S_{\alpha,b}(\alpha \oo b\oo c \oo d \oo k \oo l) &= c_{\alpha,v}'(\low \alpha 1 \oo \low \alpha 2 b \oo c\oo d  \oo k \oo l) \\
&= \low \alpha 3 b \oo \low \alpha 1 c \oo \low \alpha 2 d \oo k \oo l \, = \, c_{\alpha,v}(\alpha \oo b\oo c \oo d\oo k \oo l).
\end{align*}
As   edge slides from $w$ to $v$ are the inverses of edge slides from $v$ to $w$, the corresponding identities for those follow by 
pre-composing  with the inverses. The proof for vertices  with  different numbers of incident edge ends or  incident loops is analogous.
\end{proof}
 
Next, we consider the interaction of edge contractions with the (co)module  structures for the vertices (faces) of the graph. For this, note that the contraction of an edge $\alpha$ towards $v\in \{s(\alpha), t(\alpha)\}$ defines a bijection between the sets $F,F'$ of faces  before and after the contraction and likewise a bijection between the sets $V\setminus\{v\}$ and $V'$. 
If faces and vertices are identified via these bijections, the edge contraction becomes a (co)module  morphism. In contrast, the module structure $\rhd_v$ is coequalised.

\begin{lemma}\label{Lemma:EdgeContractionIsModuleMorphism}$ $ The  contraction  of  an edge $\alpha$ towards a ciliated vertex $v$ coequalises $\rhd_v$ and $\epsilon \oo 1_{H^{\oo E}}$  and is a (co)module morphism with respect to the (co)actions $\rhd_z$ and $\delta_f$ for all 
ciliated vertices $z\neq v$ and  ciliated faces $f\in F$ that do not start at $v$: 
\begin{align}
c_{\alpha,v}\circ \rhd_v &= c_{\alpha,v}\circ (\epsilon \, \oo \, 1_{H^{\oo E}}), \label{eq:EdgeContractionVertexActionAtV}\\
c_{\alpha,v}\circ \rhd_z &= \rhd_z'\circ (1_H \, \oo \, c_{\alpha,v}),\label{eq:EdgeContractionModuleMorphism}\\
\delta_f'\circ c_{\alpha,v} &= (1_H \, \oo \, c_{\alpha,v})\circ \delta_f. \label{eq:EdgeContractionComoduleMorphism}
\end{align}
\end{lemma}
\begin{proof}
As edge slides along $\alpha$ are  module and comodule isomorphisms by Proposition \ref{Prop:EdgeSlideIsIso}  and commute with the contraction of $\alpha$ by Lemma \ref{Lemma:EdgeContractionAndSlidesCompatible}, we can assume that $v$ is univalent. With Lemma \ref{lemma:EdgeContractionReversalCommutes} we can assume that $v=t(\alpha)$ and that all edge ends  at  $w=s(\alpha)$ are incoming: 
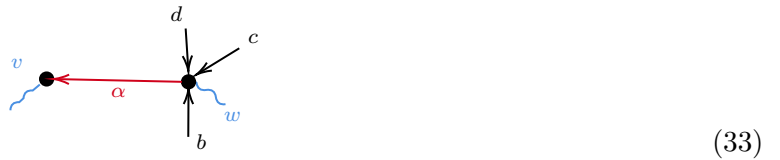
\begin{align}\label{pic:ribbonproof}
\tikzset{every picture/.style={line width=0.75pt}}  
\begin{tikzpicture}[x=0.75pt,y=0.75pt,yscale=-1.2,xscale=1.2]
\draw  [fill={rgb, 255:red, 0; green, 0; blue, 0 }  ,fill opacity=1 ] (34.43,45.75) .. controls (34.43,44.23) and (35.66,43) .. (37.18,43) .. controls (38.7,43) and (39.93,44.23) .. (39.93,45.75) .. controls (39.93,47.27) and (38.7,48.5) .. (37.18,48.5) .. controls (35.66,48.5) and (34.43,47.27) .. (34.43,45.75) -- cycle ;
\draw [color={rgb, 255:red, 74; green, 144; blue, 226 }  ,draw opacity=1 ]   (111.56,56.3) .. controls (109.23,56.65) and (107.9,55.65) .. (107.55,53.32) .. controls (107.2,50.99) and (105.86,49.99) .. (103.53,50.34) .. controls (101.2,50.69) and (99.86,49.69) .. (99.51,47.36) -- (98.93,46.93) -- (98.93,46.93) ;
\draw  [fill={rgb, 255:red, 0; green, 0; blue, 0 }  ,fill opacity=1 ] (93.43,46.93) .. controls (93.43,45.41) and (94.66,44.18) .. (96.18,44.18) .. controls (97.7,44.18) and (98.93,45.41) .. (98.93,46.93) .. controls (98.93,48.45) and (97.7,49.68) .. (96.18,49.68) .. controls (94.66,49.68) and (93.43,48.45) .. (93.43,46.93) -- cycle ;
\draw [color={rgb, 255:red, 74; green, 144; blue, 226 }  ,draw opacity=1 ]   (22.13,58.87) .. controls (22.22,56.51) and (23.45,55.38) .. (25.8,55.47) -- (27.58,53.81) -- (27.58,53.81) .. controls (27.78,51.46) and (29.05,50.39) .. (31.4,50.58) -- (34.36,48.07) -- (34.36,48.07) ;
\draw    (96.08,70.01) -- (96.17,51.68) ;
\draw [shift={(96.18,49.68)}, rotate = 90.3] [color={rgb, 255:red, 0; green, 0; blue, 0 }  ][line width=0.75]    (6.56,-1.97) .. controls (4.17,-0.84) and (1.99,-0.18) .. (0,0) .. controls (1.99,0.18) and (4.17,0.84) .. (6.56,1.97)   ;
\draw    (117.28,32.87) -- (100.98,42.96) ;
\draw [shift={(99.28,44.01)}, rotate = 328.24] [color={rgb, 255:red, 0; green, 0; blue, 0 }  ][line width=0.75]    (6.56,-1.97) .. controls (4.17,-0.84) and (1.99,-0.18) .. (0,0) .. controls (1.99,0.18) and (4.17,0.84) .. (6.56,1.97)   ;
\draw    (94.99,24.58) -- (96.01,40.58) ;
\draw [shift={(96.13,42.58)}, rotate = 266.37] [color={rgb, 255:red, 0; green, 0; blue, 0 }  ][line width=0.75]    (6.56,-1.97) .. controls (4.17,-0.84) and (1.99,-0.18) .. (0,0) .. controls (1.99,0.18) and (4.17,0.84) .. (6.56,1.97)   ;
\draw [color={rgb, 255:red, 208; green, 2; blue, 27 }  ,draw opacity=1 ]   (93.43,46.93) -- (41.93,45.8) ;
\draw [shift={(39.93,45.75)}, rotate = 1.26] [color={rgb, 255:red, 208; green, 2; blue, 27 }  ,draw opacity=1 ][line width=0.75]    (6.56,-1.97) .. controls (4.17,-0.84) and (1.99,-0.18) .. (0,0) .. controls (1.99,0.18) and (4.17,0.84) .. (6.56,1.97)   ;
\draw (62.71,48.01) node [anchor=north west][inner sep=0.75pt]  [font=\scriptsize]  {$\textcolor[rgb]{0.82,0.01,0.11}{\alpha }$};
\draw (21.14,36.26) node [anchor=north west][inner sep=0.75pt]  [font=\scriptsize]  {$\textcolor[rgb]{0.29,0.56,0.89}{v}$};
\draw (119.71,25.4) node [anchor=north west][inner sep=0.75pt]  [font=\scriptsize]  {$c$};
\draw (109.71,58.26) node [anchor=north west][inner sep=0.75pt]  [font=\scriptsize]  {$\textcolor[rgb]{0.29,0.56,0.89}{w}$};
\draw (87.43,13.97) node [anchor=north west][inner sep=0.75pt]  [font=\scriptsize]  {$d$};
\draw (98.29,67.69) node [anchor=north west][inner sep=0.75pt]  [font=\scriptsize]  {$b$};
\end{tikzpicture}
\end{align}
For the vertices $v$ and $w$ in \eqref{pic:ribbonproof} we compute \begin{align*}
c_{\alpha,v}\circ \rhd_v (h \oo \alpha\oo b\oo c\oo d) &= c_{\alpha,v}(h\alpha\oo b\oo c\oo d)= \epsilon(h\alpha) b\oo c\oo d = c_{\alpha,v}(\epsilon(h) \alpha \oo b\oo c\oo d)\\ 
&= c_{\alpha,v}\circ (\epsilon\oo 1_{H^{\oo E}}) ( h \oo \alpha\oo b\oo c\oo d)\\
 c_{\alpha,v}\circ \rhd_w(h \oo \alpha \oo b \oo c \oo d ) \, &=\,  c_{\alpha,v}(\alpha S(\low h 3) \oo \low h 4 b \oo \low h 1 c \oo \low h 2 d)\\
&= \epsilon(\alpha) \low h 3 b \oo \low h 1 c \oo \low h 2 d 
= \rhd_w'\circ (1_H \oo c_{\alpha,v})(h \oo \alpha\oo b \oo c\oo d).
\end{align*}
The computations for graphs with a different number of edge ends or loops incident at $w$ are analogous.
For vertices $z\in V\setminus\{v,w\}$ the action $\rhd_z$ does not affect the copy of $H$ for $\alpha$  and commutes with $\rhd_{v,\alpha}$ and hence with $c_{v,\alpha}$. 
This  proves \eqref{eq:EdgeContractionVertexActionAtV} and \eqref{eq:EdgeContractionModuleMorphism}.

If $f$ is a face that contains $\alpha$, but does not start at $v$, then the associated coaction is of the form
$$
\delta_f(\alpha\oo b\oo c\oo d\oo \ldots)=(\cdots S(\low d 2)S(\low\alpha 3) \low\alpha 1\low b 1\cdots)\oo \low\alpha 2\oo \low b 2\oo c\oo \low d 1\oo \ldots,
$$
where the dots stand for contributions of parts of $\Gamma$ that are not drawn in \eqref{pic:ribbonproof}. This yields
\begin{align*}
&(1_H\oo c_{\alpha, v})\circ \delta_f(\alpha\oo b\oo c\oo d\oo \ldots)
=\epsilon(\low\alpha 2) (\cdots S(\low d 2)S(\low\alpha 3) \low\alpha 1\low b 1\cdots) \oo \low b 2\oo c\oo \low d 1\oo \ldots\\
&\overset{\eqref{eq:EightWithAntipode}}{=}\epsilon(\alpha ) (\cdots S(\low d 2)\low b 1\cdots) \oo \low b 2\oo c\oo \low d 1\oo \ldots=\delta'_f\circ c_{\alpha,v}(\alpha\oo b\oo c\oo d\oo \ldots).
\end{align*}
If $f$ does not contain $\alpha$,  the edge $\alpha$ does not contribute to the coaction $\delta_f$, which proves \eqref{eq:EdgeContractionComoduleMorphism}. 
\end{proof}

With these results, we  investigate how edge contractions interact with the (co)invariants  of the $H$-(co)module structures at ciliated vertices and  faces of $\Gamma$. For subsets $\emptyset\neq \mathcal V\subset V$ and $\emptyset\neq \mathcal F\subset F$ we denote by $\rhd_{\mathcal V}$ and $\delta_{\mathcal F}$ the associated $H^{\oo \mathcal V}$-module structure and $H^{\oo \mathcal F}$-comodule structure 
from  \eqref{eq:ComposedVertexAction}
and by $\pi_{\mathcal V}$ and $\iota_{\mathcal F}$ their invariants and coinvariants from Definition \ref{def:invariants}.

We then find that edge contractions send coinvariants for $\delta_{\mathcal F}$ to coinvariants for the corresponding face set in the contracted graph. 
The same holds for the invariants of the action $\rhd_{\mathcal V}$, as long as $\mathcal V$ contains the starting and target vertex of the contracted edge.  
The morphism $\eta_\alpha$ that creates a copy of $H$ assigned to $\alpha$ by applying the unit of $H$ is right inverse to 
the edge contraction $c_{v,\alpha}$ and a left inverse
on the coinvariants. 
This corresponds to  the following technical lemma.

\begin{lemma} \label{Lemma:EdgeContractionAndCoInvarianceMorphisms} Let $\Gamma'$ be obtained from $\Gamma$ by contracting an edge $\alpha$ incident at $v,w\in V$. 
Then $\eta_\alpha: H^{\oo(E-1)}\to H^{\oo E}$ is right inverse to
 the edge contraction $c_{\alpha, v}: H^{\oo E}\to H^{\oo(E-1)}$, and for all subsets $\{v,w\}\subset\mathcal V\subset V$, $\emptyset \neq \mathcal F\subset F$ one has
\begin{align}
\delta_{\mathcal F}' \circ c_{\alpha,v}\circ \iota_{\mathcal F} &= (\eta^{\oo \vert \mathcal F\vert} \oo c_{\alpha,v})\circ \iota_{\mathcal F} \label{eq:WholeIotaFAndContraction}\\
\delta_{\mathcal F}\circ \eta_{\alpha}\circ \iota_{\mathcal F}' &= (\eta^{\oo \vert \mathcal F\vert} \oo \eta_{\alpha}) \circ \iota_{\mathcal F}' \label{eq:WholeIotaFAndEtaAlpha}\\
\pi_{\mathcal V}\circ \eta_{\alpha}\circ c_{\alpha,v}&= \pi_{\mathcal V} \label{eq:WholePiVAndContraction}\\
\pi_{\mathcal V}'\circ c_{\alpha,v}\circ \rhd_{\mathcal V} &= \pi_{\mathcal V}'\circ (\epsilon^{\oo \vert \mathcal V \vert}\oo c_{\alpha,v}) \label{eq:WholePiVWholeAction}\\
 \pi_{\mathcal V} \circ \eta_{\alpha} \circ \rhd_{\mathcal V}' &=\pi_{\mathcal V}\circ (\epsilon^{\oo \vert \mathcal V \vert -1}\oo \eta_{\alpha}).
\label{eq:WholePiVWholeActionDash}
\end{align}
\end{lemma}
\begin{proof}
1.~It follows directly from Definition \ref{def:EdgeContraction} that the morphism   $\eta_\alpha$   is a right inverse to $c_{\alpha, v}$. From the formula  for the (co)action in Definition \ref{def:vertexfaceops} it is apparent that $\eta_\alpha$ is a comodule morphism for the coactions $\delta_f$ at all ciliated faces 
 and a module morphism with respect to the actions $\rhd_z$ at all vertices $z\in V\setminus\{v,w\}$.
Moreover, it is clear from Definition \ref{def:EdgeSlide}  that sliding edge ends over $\alpha$ after applying $\eta_\alpha$ yields a morphism $\eta_\alpha''$ which splits the vertex $w$ in a different way.
Thus, we have 
\begin{align}\label{eq:helpids}
&c_{\alpha,v}\circ \eta_\alpha=1_{H^{\oo E}}, \quad
\delta_f\circ \eta_\alpha=\eta_\alpha\circ \delta'_f,\quad
\eta_\alpha\circ \rhd'_z=\rhd_z\circ \eta_\alpha,\quad 
S_{\alpha,\beta}\circ \eta_{\alpha}=\eta''_{\alpha}
\end{align}
for all vertices $z\in V\setminus\{v,w\}$ and faces $f\in F$ and edge slides $S_{\alpha,\beta}$ along $\alpha$. We can therefore assume that the vertex $v=t(\alpha)$ is univalent, all edge ends at $w=s(\alpha)$ are incoming, the graph $\Gamma$ is locally given by \eqref{pic:ribbonproof} and the edge contraction  by $c_{\alpha,v}=\epsilon_\alpha$, as in the proof of Lemma \ref{Lemma:EdgeContractionIsModuleMorphism}. 

2.~We prove the auxiliary identities
\begin{align}
\pi_v \circ \eta_{\alpha}\circ c_{\alpha,v} &= \pi_v \text{,}  \label{eq:PiVEtaAlphaContraction} \\
\pi_{\{v, w\}} \circ \eta_{\alpha} \circ (\epsilon\, \oo \, 1_{H^{\oo (E-1)}}) &= \pi_{\{v, w\}} \circ \eta_{\alpha}\circ \rhd_w'\text{,} \label{eq:PiVWAndEtaAlpha} \\
\delta_f'\circ c_{\alpha,v} \circ \iota_f &= (\eta \, \oo \, 1_{H^{\oo (E-1)}}) \circ c_{\alpha,v}\circ \iota_f \qquad \forall f\in F.\label{eq:IotaFAndContraction}
\end{align}
Omitting all copies of $H$ in $H^{\oo E}$ except the one for $\alpha$, we verify \eqref{eq:PiVEtaAlphaContraction}
\begin{align*}
\pi_v(\alpha\oo \ldots)=\pi_v\circ (\alpha \, \rhd_v)(1\oo \ldots)=\pi_v (\epsilon(\alpha) \, 1\oo\ldots)=\pi_v\circ \eta_\alpha\circ\epsilon_\alpha(1\oo \ldots)=\pi_v\circ \eta_\alpha\circ c_{\alpha,v}(1\oo \ldots).
\end{align*}
To show \eqref{eq:PiVWAndEtaAlpha}, we consider the graph  \eqref{pic:ribbonproof} and compute with  Lemma \ref{Lemma:RelatingTheInvariants}
\begin{align*}
&\pi_{\{v,w\}}\circ \eta_\alpha\circ \rhd'_w(h\oo b\oo c\oo d)=\pi_{\{v,w\}}\circ \eta_\alpha(\low h 3 b\oo \low h 1 c\oo \low h 2 d)=\pi_{\{v,w\}}(1\oo \low h 3 b\oo \low h 1 c\oo \low h 2 d)\\
&=\pi_{\{v,w\}}( \low h 3 S(\low h 4)\oo \low h 5 b\oo \low h 1 c\oo \low h 2 d)=\pi_{\{v,w\}}\circ \low h 3\rhd_v(S(\low h 4)\oo \low h 5 b\oo \low h 1 c\oo \low h 2 d)\\
&=\pi_{\{v,w\}}(\epsilon (\low h 3)\, S(\low h 4)\oo \low h 5 b\oo \low h 1 c\oo \low h 2 d)=\pi_{\{v,w\}}(S(\low h 3)\oo \low h 4 b\oo \low h 1 c\oo \low h 2 d)\\
&=\pi_{\{v,w\}}\circ h\rhd_w \, (1\oo  b\oo  c\oo  d)= \pi_{\{v,w\}}(\epsilon(h) \, 1\oo  b\oo  c\oo  d)= \pi_{\{v,w\}}\circ \eta_\alpha\circ (\epsilon\oo 1_{H^{\oo (E-1)}})(b\oo  c\oo  d).
\end{align*}
Identity \eqref{eq:IotaFAndContraction} follows from identity \eqref{eq:EdgeContractionComoduleMorphism} in Lemma \ref{Lemma:EdgeContractionIsModuleMorphism} for all faces $f\in F$ that do not start at $v$. If $f$ starts at $v$ one has for the graph in \eqref{pic:ribbonproof}
\begin{align*}
&\delta'_f\circ c_{\alpha, v}(\alpha\oo b\oo c\oo d)=\epsilon(\alpha)\delta'_f(b\oo c\oo d)=\epsilon(\alpha) \low b 1 \cdots S(\low d 2)\oo \low b 2\oo c\oo \low d 1\nonumber\\
&= S(\low\alpha 2) \low\alpha 1\low b 1 \cdots S(\low d 2)S(\low \alpha 4)\low \alpha 3\oo \low b 2\oo c\oo \low d 1\nonumber\\
&= (\lhd_{ad} \oo 1_{H^{\oo (E-1)}}) \circ (1_H\oo \tau_\alpha)\circ \delta_f (\alpha\oo b\oo c\oo d),
\end{align*}
where $\lhd_{ad}: H\oo H\to H$, $h\oo\alpha\mapsto S(\low \alpha 1) h\low\alpha 2$. In this case, contracting $\alpha$ deletes the cilium of $f$, but  Lemma \ref{lemma:InvariantsIndependentFromCilium} allows one to place a new cilium for  $f$ in any position. As $\lhd_{ad}\circ (\eta \oo 1_H)=\eta\circ \epsilon: H\to H$ this yields
\begin{align*}
&\delta'_f\circ c_{\alpha, v}\circ \iota_f=(\lhd_{ad} \oo 1_{H^{\oo (E-1)}}) \circ (1_H\oo \tau_\alpha)\circ \delta_f\circ \iota_f=(\lhd_{ad} \oo 1_{H^{\oo (E-1)}}) \circ (1_H\oo \tau_\alpha)\circ (\eta\oo 1_{H^{\oo E}})\circ \iota_f\\
&=(\eta\circ \epsilon\oo 1_{H^{\oo (E-1)}})\circ \tau_\alpha\circ \iota_f=(\eta\oo 1_{H^{\oo (E-1)}})\circ \epsilon_\alpha\circ \iota_f=(\eta\oo 1_{H^{\oo (E-1)}})\circ c_{\alpha, v}\circ \iota_f.
\end{align*}

3.~We prove the identities in the Lemma. Identity \eqref{eq:WholeIotaFAndContraction} follows  by pre-composing \eqref{eq:IotaFAndContraction} with the morphism $\xi_{f,\mathcal F}:=\xi_{\{f\},\mathcal F}$ from Lemma \ref{Lemma:RelatingTheInvariants} and inductively applying this equation for all $f\in \mathcal F$. Likewise, identity \eqref{eq:WholeIotaFAndEtaAlpha} follows by 
applying the identity $\delta_f \circ \eta_{\alpha}\circ \iota_f' = (\eta \, \oo \, 1_{H^{\oo E}}) \circ \eta_{\alpha}\circ \iota_f'$ obtained from  the second identity in \eqref{eq:helpids}
and pre-composing it with  $\xi_{f,\mathcal F}'$. Post-composing \eqref{eq:PiVEtaAlphaContraction} with the morphism $\chi_{v,\mathcal V}:=\chi_{\{v\},\mathcal V}$ from  Lemma \ref{Lemma:RelatingTheInvariants} yields \eqref{eq:WholePiVAndContraction}.  
From \eqref{eq:EdgeContractionModuleMorphism}, we obtain for all $z\in V\setminus\{v,w\}$ 
\begin{align}\label{eq:WholePiVAndActionZ}
\pi_{\mathcal V}'\circ c_{\alpha,v}\circ \rhd_z = \chi_{z,\mathcal V}'\circ \pi_z'\circ c_{\alpha,v}\circ \rhd_z = \chi_{z,\mathcal V}'\circ \pi_z'\circ \rhd_z'\circ (1_H\oo c_{\alpha,v}) =\pi_{\mathcal V}'\circ (\epsilon \oo c_{\alpha,v}).
\end{align}
Together with  the identity $\pi_{\mathcal V}'\circ c_{\alpha,v}\circ \rhd_w\circ (1_H \oo \rhd_v) =\pi_{\mathcal V}' \circ (\epsilon^{\oo 2}\oo c_{\alpha, v}),$ which follows from \eqref{eq:EdgeContractionVertexActionAtV} and \eqref{eq:EdgeContractionModuleMorphism} with $z=w$ and the identity  $\pi_{\mathcal V}'=\chi_{w,\mathcal V}'\circ \pi_w'$, this yields
 \eqref{eq:WholePiVWholeAction}. Identity \eqref{eq:WholePiVWholeActionDash} follows by  post-composing \eqref{eq:PiVWAndEtaAlpha} with $\chi_{\{v,w\},\mathcal V}$ and    the third identity  in \eqref{eq:helpids}  
with $\pi_{\mathcal V}=\chi_{z,\mathcal V}\circ \pi_z$.
\end{proof}

We now apply Lemma \ref{Lemma:EdgeContractionAndCoInvarianceMorphisms} to show that edge contractions induce morphisms between the coinvariants for $\emptyset\neq \mathcal F\subset F$. If $\mathcal V$ contains the starting and target vertex of the contracted edge, they also induce
 isomorphisms between the invariants  and isomorphisms between the protected objects.

For this, we consider a ciliated ribbon graph $\Gamma$ and the graph $\Gamma'$ obtained by contracting an edge $\alpha$ in $\Gamma$. We denote by $\mathcal M^{coH}$, $\mathcal M^H$, $\mathcal M_{inv}$ the coinvariants, invariants and biinvariants of $\delta_{\mathcal F}$, $\rhd_{\mathcal V}$ for $\Gamma$ and by $\mathcal M'^{coH}$, $\mathcal M'^H$, $\mathcal M'_{inv}$ the corresponding quantities for $\Gamma'$. As in Lemma \ref{Lemma:RelatingTheInvariants}  we write $\iota_{\mathcal F}$ and $\pi_{\mathcal V}$ for the associated equaliser and coequaliser and 
  $I: \mathcal M_{inv}\to \mathcal M^H$
and $P: \mathcal M^{coH}\to\mathcal M_{inv}$ for  the monomorphism and epimorphism that characterise $\mathcal M_{inv}$ as the image of $\pi_{\mathcal V}\circ \iota_{\mathcal F}$. 
The corresponding morphisms for $\Gamma'$ are denoted $\iota'_{\mathcal F}$, $\pi'_{\mathcal V}$, $I'$ and $P'$.

\begin{proposition}\label{prop:EdgeContractionInducesIso}
Let $\Gamma'$ be obtained from a ciliated ribbon graph  $\Gamma$ by contracting an edge $\alpha$ incident at $v,w$ towards $v$. Then for all  $\{v,w\}\subset \mathcal V \subset V$, $\emptyset \neq \mathcal F \subset F$  the  contraction of $\alpha$ induces 
\begin{compactitem}
\item a morphism $u:M_{\mathcal F}^{coH}\to M'^{coH}_{\mathcal F}$ with a right inverse that satisfies $\iota_{\mathcal F}'\circ u= c_{\alpha,v}\circ \iota_{\mathcal F}$, 
\item an isomorphism $r: M_{\mathcal V}^H \to M'^H_{\mathcal V}$ that satisfies $r\circ \pi_{\mathcal V}= \pi_{\mathcal V}'\circ c_{\alpha,v}$,
\item an isomorphism $\phi_{inv}:M_{inv}\to M'_{inv}$ with $I= r^\inv \circ I'\circ \phi_{inv}$.
\end{compactitem}
\end{proposition}
\begin{proof}
Using equation \eqref{eq:WholeIotaFAndContraction} together with the universal property of the equaliser $ \iota_{\mathcal F}' $ yields a unique morphism $u: M^{coH}_{\mathcal F} \to M'^{coH}_{\mathcal F}$ with $\iota_{\mathcal F}' \circ u = c_{\alpha,v} \circ \iota_{\mathcal F}$.
Equation \eqref{eq:WholeIotaFAndEtaAlpha} and the  equaliser $\iota_{\mathcal F}$ yield a unique morphism $u^\inv: M'^{coH}_{\mathcal F} \to M^{coH}_{\mathcal F}$ with 
$
\iota_{\mathcal F} \circ u^\inv= \eta_{\alpha} \circ \iota'_{\mathcal F} 
$.
To show that $u^\inv$ is a right inverse of $u$ note that 
$\iota_{\mathcal F}' \circ u \circ u^\inv= c_{\alpha,v} \circ \iota_{\mathcal F} \circ u^\inv = c_{\alpha,v} \circ \eta_{\alpha}\circ \iota_{\mathcal F}' =  \iota_{\mathcal F}' $, since $\eta_\alpha$ is right inverse to $c_{\alpha,v}$.
As $\iota_{\mathcal F}'$ is a monomorphism, this implies $u\circ u^\inv = 1_{M'^{coH}_{\mathcal F}}$.

Analogously, \eqref{eq:WholePiVWholeAction} and the universal property of the coequaliser $\pi_{\mathcal V}$ define a unique morphism $ r: M^H_{\mathcal V} \to M'^H_{\mathcal V}$ with $ r \circ \pi_{\mathcal V}= \pi_{\mathcal V}' \circ c_{\alpha,v}$. 
The coequaliser $\pi_{\mathcal V}'$ together with \eqref{eq:WholePiVWholeActionDash} yields a unique morphism $r^\inv: M'^H_{\mathcal V} \to M^H_{\mathcal V}$ with 
$
r^\inv \circ \pi'_{\mathcal V} = \pi_{\mathcal V} \circ \eta_{\alpha} 
$.
The morphisms $ r$ and $r^\inv$ are mutually inverse isomorphisms, since $\pi_{\mathcal V}'$, $\pi_{\mathcal V}$ are epimorphisms with
\begin{align*}
 r\circ r^\inv \circ \pi_{\mathcal V}' &=  r\circ \pi_{\mathcal V} \circ \eta_{\alpha} = \pi_{\mathcal V}' \circ c_{\alpha,v} \circ \eta_{\alpha} =\pi_{\mathcal V}',  \\
r^\inv \circ  r \circ \pi_{\mathcal V} &= r^\inv \circ \pi_{\mathcal V}' \circ c_{\alpha,v} = \pi_{\mathcal V} \circ \eta_{\alpha} \circ c_{\alpha,v} \overset{\eqref{eq:WholePiVAndContraction}}{=} \pi_{\mathcal V}.
\end{align*}
Hence, we constructed commuting  diagrams 
$$
\xymatrix{
M^{coH}_{\mathcal F} \ar[r]^{\iota_{\mathcal F}} \ar[d]^{u} & M  \ar[r]^{\pi_{\mathcal V}} \ar[d]^{c_{\alpha,v}}  & M^H_{\mathcal V} \ar[d]^{r}\\
M'^{coH}_{\mathcal F} \ar[r]_{\iota_{\mathcal F}'}  & M' \ar[r]_{\pi_{\mathcal V}'} & M'^H_{\mathcal V}.
}\qquad\qquad
\xymatrix{
M'^{coH}_{\mathcal F} \ar[r]^{\iota'_{\mathcal F}} \ar[d]^{u^\inv} & M'  \ar[r]^{\pi'_{\mathcal V}} \ar[d]^{\eta_{\alpha}}  & M'^H_{\mathcal V} \ar[d]^{r^\inv}\\
M^{coH}_{\mathcal F} \ar[r]_{\iota_{\mathcal F}}  & M \ar[r]_{\pi_{\mathcal V}} & M^H_{\mathcal V}
}
$$
To construct the isomorphism $\phi_{inv}$, we set $j:= r^\inv \circ I' : M'_{inv} \to M^H_{\mathcal V}$ and $q:= P'\circ u: M^{coH}_{\mathcal F} \to M'_{inv}$.
As $r^\inv$ is an isomorphism and $I'$ a monomorphism, the  morphism $j$ is a monomorphism. The composite $j \circ q $ satisfies
\begin{align*}
j \circ q = r^\inv \circ I' \circ P' \circ u = r^\inv \circ \pi'_{\mathcal V} \circ \iota'_{\mathcal F} \circ u = \pi_{\mathcal V}  \circ \eta_{\alpha} \circ  \iota'_{\mathcal F} \circ u = \pi_{\mathcal V} \circ \eta_{\alpha} \circ c_{\alpha,v} \circ \iota_{\mathcal F} \overset{\eqref{eq:WholePiVAndContraction}}{=} \pi_{\mathcal V} \circ \iota_{\mathcal F} \text{.}
\end{align*}
The universal property of the image $M_{inv}$ then yields a unique morphism $\phi_{inv}: M_{inv} \to M'_{inv}$ with $I = j \circ \phi_{inv}=r^\inv\circ I'\circ \phi_{inv}$.  
To construct its inverse we set $j':= r \circ I: M_{inv} \to M'^H_{\mathcal V}$ and $q':= P \circ u^\inv : M'^{coH}_{\mathcal F} \to M_{inv}$. As
$r$ is an isomorphism and $I$ a monomorphism, $j'$ is  a monomorphism, and we have
\begin{align*}
j' \circ q' =  r \circ I \circ P \circ u^\inv =  r \circ \pi_{\mathcal V} \circ \iota_{\mathcal F} \circ u^\inv = \pi'_{\mathcal V} \circ c_{\alpha,v} \circ \iota_{\mathcal F} \circ u^\inv = \pi'_{\mathcal V} \circ c_{\alpha,v} \circ \eta_{\alpha} \circ \iota'_{\mathcal F} = \pi'_{\mathcal V} \circ \iota'_{\mathcal F},
\end{align*}
where we used that $\eta_\alpha$ is right inverse to $c_{\alpha,v}$ in the last step. 
By the universal property of the image $M'_{inv}$ there is a unique morphism $\phi_{inv}^\inv: M'_{inv} \to M_{inv}$ with $I' = j' \circ \phi_{inv}^\inv=r\circ I\circ \phi_{inv}^\inv$ and 
\begin{align*}
I \circ \phi_{inv}^\inv \circ \phi_{inv} &= r^\inv \circ  r\circ I \circ \phi_{inv}^\inv \circ \phi_{inv} = r^\inv \circ I' \circ \phi_{inv} = I, \\
I' \circ \phi_{inv}\circ \phi_{inv}^\inv &=  r\circ r^\inv \circ I' \circ \phi_{inv}\circ \phi_{inv}^\inv =  r \circ I \circ \phi_{inv}^\inv = I'.
\end{align*}
As $I,I'$  are monomorphisms, it follows that $\phi_{inv}$ and $\phi_{inv}^\inv$ are mutually  inverse isomorphisms. 
\end{proof}

\begin{corollary}\label{Coro:EdgeContracionIso}
Edge contractions  induce  isomorphisms between  protected objects. 
\end{corollary}

\subsection{Deleting isolated loops}
\label{subsec:loops}

We now consider the last graph transformation from Definition \ref{def:graphtrafos},  the deletion of isolated loops. The morphism associated to the deletion of an isolated loop $\alpha$    applies the counit to the corresponding copy of the Hopf monoid $H$. Just as  edge contractions,  this is in general not an isomorphism in $\mac$. The morphism $\eta_\alpha$ that creates a copy of $H$ for $\alpha$ by applying the unit is a right inverse and corresponds to inserting a loop.

\begin{definition}\label{def:LoopRemoval}
The morphism induced by {\bf deleting an isolated loop} $\alpha$  is  $\epsilon_{\alpha}: H^{\oo E}\to H^{\oo E\setminus\{\alpha\}}$.
\end{definition}

As for edge contractions we investigate how these morphisms interact with the coinvariants for the $H^{\oo \mathcal F}$-comodule structure $\delta_{\mathcal F}$ and the $H^{\oo\mathcal V}$-module structure $\rhd_{\mathcal V}$ from Definition \ref{def:setact}  for subsets $\emptyset\neq \mathcal F\subset F$ and $\emptyset\neq  \mathcal V\subset V$.
We find that loop deletions send the invariants for $\rhd_{\mathcal V}$ to invariants for the corresponding vertex set  of the graph with the loop removed. The same holds for coinvariants of $\delta_{\mathcal F}$, as long as the two faces incident to the loop are contained in $\mathcal F$. Analogous statements hold for the right inverse $\eta_\alpha$, and on the coinvariants $\eta_\alpha$ is also a left inverse. This is a consequence of  the following technical lemma.

\begin{lemma}\label{Lemma:LoopRemovalAndCoInvarianceMorphisms}
Let $\Gamma^+$ be obtained from a ciliated ribbon graph $\Gamma$ by removing an isolated loop $\alpha$   with adjacent faces $f_1,f_2$ at a vertex $v$. 
 Then for all subsets $\emptyset\neq \mathcal V \subset V$ and $\{f_1, f_2\}\subset \mathcal F\subset F$ 
\begin{align}
\pi_{\mathcal V}^+ \circ  \epsilon_{\alpha} \circ \rhd_{\mathcal V} &= \pi_{\mathcal V}^+  \circ (\epsilon^{\oo \vert \mathcal V \vert} \oo \epsilon_{\alpha}) \label{eq:LoopRemovalWholePiPlus}\\
\delta_{\mathcal F}^+ \circ \epsilon_{\alpha} \circ \iota_{\mathcal F} &= (\eta^{\oo \vert \mathcal F\vert-1} \oo \epsilon_{\alpha})\circ \iota_{\mathcal F} \label{eq:LoopRemovalWholeIota}\\
\eta_{\alpha}\circ \epsilon_{\alpha} \circ \iota_{\mathcal F} &= \iota_{\mathcal F} \label{eq:LoopEtaLoopIotaIsIota}\\
\pi_{\mathcal V}\circ \eta_{\alpha}\circ \rhd_{\mathcal V }^+ &= \pi_{\mathcal V}\circ (\epsilon^{\oo \vert \mathcal V\vert }\oo \eta_{\alpha}) \label{eq:LoopRemovalWholePi}\\
\delta_{\mathcal F}\circ \eta_{\alpha}\circ \iota_{\mathcal F}^+ &= (\eta^{\oo \vert \mathcal F \vert} \oo \eta_{\alpha}) \circ \iota_{\mathcal F}^+. \label{eq:LoopRemovalWholeIotaPlus}
\end{align}
\end{lemma}
\begin{proof}  1.~We first prove some auxiliary identities for the interaction of the morphisms $\epsilon_\alpha$ and $\eta_\alpha$ with the module and comodules structures at the vertices and faces. 

1.(a) As $\eta_\alpha$ and $\epsilon_\alpha$ affect only the copy of $H$ for $\alpha$, we have for any vertex $z\neq v$ and any ciliated face $f$ that does not contain $\alpha$
\begin{align}
&\epsilon_\alpha \circ \rhd_z=\rhd^+_z\circ (1_H\oo \epsilon_\alpha) & &\eta_{\alpha}\circ \rhd_z^+ = \rhd_z  \circ (1_H\oo \eta_{\alpha})  \label{eq:LoopActionEtaAlpha}\\
&(1_H\oo \epsilon_\alpha)\circ \delta_f=\delta_f^+\circ \epsilon_\alpha & &(1_H\oo \eta_{\alpha})\circ \delta_f^+ = \delta_f \circ \eta_{\alpha}\label{eq:LoopCoactionEtaAlpha}. 
\end{align}
1.(b) For the $H$-module structure at the vertex $v$, we show that
\begin{align}
\pi_v^+ \circ \epsilon_{\alpha}\circ \rhd_v &= \pi_v^+\circ (\epsilon \oo \epsilon_{\alpha})\label{eq:LoopRemovalPiPlusAtV}\\
\pi_v \circ \eta_{\alpha}\circ \rhd_v^+ &= \pi_v \circ (\epsilon \oo \eta_{\alpha})\label{eq:LoopRemovalPiAtV}.
\end{align}
As reversing edge orientations  commutes with $\epsilon_{\alpha}, \eta_{\alpha}$ and $\rhd_v$, we
can assume that all edges $\beta\neq \alpha$ at $v$ are incoming and that $s(\alpha)$ is directly before $t(\alpha)$ with respect to the cyclic ordering at $v$:  
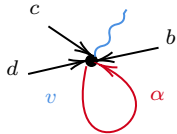
\begin{align*}
\tikzset{every picture/.style={line width=0.75pt}} 
\begin{tikzpicture}[x=0.75pt,y=0.75pt,yscale=-1.3,xscale=1.3]
\draw  [fill={rgb, 255:red, 0; green, 0; blue, 0 }  ,fill opacity=1 ] (65.14,42.18) .. controls (65.14,41.06) and (66.06,40.14) .. (67.18,40.14) .. controls (68.31,40.14) and (69.22,41.06) .. (69.22,42.18) .. controls (69.22,43.31) and (68.31,44.22) .. (67.18,44.22) .. controls (66.06,44.22) and (65.14,43.31) .. (65.14,42.18) -- cycle ;
\draw [color={rgb, 255:red, 208; green, 2; blue, 27 }  ,draw opacity=1 ]   (65.22,44.3) .. controls (54.81,91.01) and (107.04,62.88) .. (72.02,44.55) ;
\draw [shift={(70.36,43.72)}, rotate = 25.53] [color={rgb, 255:red, 208; green, 2; blue, 27 }  ,draw opacity=1 ][line width=0.75]    (6.56,-1.97) .. controls (4.17,-0.84) and (1.99,-0.18) .. (0,0) .. controls (1.99,0.18) and (4.17,0.84) .. (6.56,1.97)   ;
\draw    (42.86,47.32) -- (63.19,42.63) ;
\draw [shift={(65.14,42.18)}, rotate = 167.01] [color={rgb, 255:red, 0; green, 0; blue, 0 }  ][line width=0.75]    (6.56,-1.97) .. controls (4.17,-0.84) and (1.99,-0.18) .. (0,0) .. controls (1.99,0.18) and (4.17,0.84) .. (6.56,1.97)   ;
\draw    (50.65,28.87) -- (65.53,39.02) ;
\draw [shift={(67.18,40.14)}, rotate = 214.3] [color={rgb, 255:red, 0; green, 0; blue, 0 }  ][line width=0.75]    (6.56,-1.97) .. controls (4.17,-0.84) and (1.99,-0.18) .. (0,0) .. controls (1.99,0.18) and (4.17,0.84) .. (6.56,1.97)   ;
\draw    (92.93,37.15) -- (71.18,41.77) ;
\draw [shift={(69.22,42.18)}, rotate = 348.03] [color={rgb, 255:red, 0; green, 0; blue, 0 }  ][line width=0.75]    (6.56,-1.97) .. controls (4.17,-0.84) and (1.99,-0.18) .. (0,0) .. controls (1.99,0.18) and (4.17,0.84) .. (6.56,1.97)   ;
\draw [color={rgb, 255:red, 74; green, 144; blue, 226 }  ,draw opacity=1 ]   (68.65,40.64) .. controls (68.17,38.33) and (69.08,36.94) .. (71.39,36.46) .. controls (73.7,35.98) and (74.61,34.59) .. (74.13,32.28) .. controls (73.66,29.97) and (74.57,28.58) .. (76.88,28.1) .. controls (79.19,27.62) and (80.1,26.23) .. (79.62,23.92) -- (80.65,22.36) -- (80.65,22.36) ;
\draw (88.57,52.54) node [anchor=north west][inner sep=0.75pt]  [font=\scriptsize]  {$\textcolor[rgb]{0.82,0.01,0.11}{\alpha }$};
\draw (48.29,53.97) node [anchor=north west][inner sep=0.75pt]  [font=\scriptsize]  {$\textcolor[rgb]{0.29,0.56,0.89}{v}$};
\draw (94.57,28.83) node [anchor=north west][inner sep=0.75pt]  [font=\scriptsize]  {$b$};
\draw (42.29,19.4) node [anchor=north west][inner sep=0.75pt]  [font=\scriptsize]  {$c$};
\draw (33.43,40.83) node [anchor=north west][inner sep=0.75pt]  [font=\scriptsize]  {$d$};
\end{tikzpicture}\\[-10ex]
\end{align*}
For this graph we compute
\begin{align*}
&\pi_v^+ \circ \epsilon_{\alpha}\circ \rhd_v \,  (h \oo \alpha \oo b \oo c \oo d ) = \pi_v^+ \circ \epsilon_{\alpha} \, (\low h 4 \alpha S( \low h 3 ) \oo \low h 5 b \oo \low h 1 c \oo \low h 2 d)\\
&= \pi_v^+ \,  (\epsilon(\alpha) \oo \low h 3 b \oo \low h 1 c \oo \low h 2 d)=\pi_v^+ \circ (\epsilon \oo \rhd_v^+) \, (\alpha \oo h \oo b \oo c\oo d)\\
&= \pi_v^+ \circ (\epsilon \oo \epsilon \oo 1_{H^{\oo 3}}) \, (h \oo \alpha \oo b \oo c \oo d) = \pi_v^+ \circ (\epsilon \oo \epsilon_{\alpha}) \, (h \oo \alpha \oo b \oo c \oo d),\\[+1ex]
&\pi_v \circ \eta_{\alpha}\circ \rhd_v^+ \, (h \oo b \oo c\oo d) = \pi_v\circ \eta_\alpha\,( \low h 3 b\oo \low h 1 c\oo \low h 2 d)\\
&=\pi_v \, (\low h 4 S(\low h 3 ) \oo \low h 5 b \oo \low h 1 c \oo \low h 2 d )= \pi_v \circ \rhd_v \, (h \oo 1 \oo b \oo c \oo d)\\
& = \pi_v \circ (\epsilon\oo 1_{H^{\oo 4}}) \, (h \oo 1\oo b \oo c \oo d)= \pi_v \circ (\epsilon \oo \eta_{\alpha}) \, (h \oo b\oo c \oo d),
\end{align*}
which proves \eqref{eq:LoopRemovalPiPlusAtV} and \eqref{eq:LoopRemovalPiAtV}. The computation for graphs with a different number of edge ends at $v$ are analogous. The claim for the case where the cilium is between the edge ends of $\alpha$ follows, because the invariants of the $H$-module structure at $v$ do not depend on the choice of the cilium by Lemma \ref{lemma:InvariantsIndependentFromCilium}. It can also be verified directly by analogous computations.

1.(c) We consider the $H$-comodule structures at the faces $f_1,f_2$. Under the assumption  that the starting end of $\alpha$ comes directly before its target end with respect to the cyclic ordering at $v$, one of these faces coincides with $\alpha$, and we assume it is $f_1=\alpha$. We then have
\begin{align}
&\delta_{f_1}\circ \eta_{\alpha} = \eta \oo \eta_{\alpha}\label{eq:LoopDeltaF1EtaAlpha}\\
&\delta_{f_2}\circ \eta_\alpha=(1_H\oo \eta_\alpha)\circ \delta^+_{f_2}   \qquad \delta_{f_2}\circ \eta_{\alpha}\circ \iota_{f_2}^+ = (\eta \oo \eta_{\alpha}) \circ \iota_{f_2}^+.\label{eq:LoopDeltaF2EtaAlpha}
\end{align}
Equation \eqref{eq:LoopDeltaF1EtaAlpha} is obvious, and to prove \eqref{eq:LoopDeltaF2EtaAlpha}, we can assume    that $\Gamma$ is locally given by
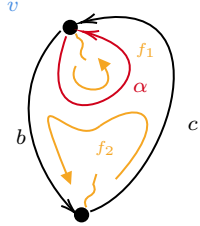
\begin{align}\label{eq:facegraphhelp}
\tikzset{every picture/.style={line width=0.75pt}} 
\begin{tikzpicture}[x=0.75pt,y=0.75pt,yscale=-1.2,xscale=1.2]
\draw  [fill={rgb, 255:red, 0; green, 0; blue, 0 }  ,fill opacity=1 ] (53.57,22.32) .. controls (53.57,20.8) and (54.8,19.57) .. (56.32,19.57) .. controls (57.84,19.57) and (59.08,20.8) .. (59.08,22.32) .. controls (59.08,23.84) and (57.84,25.08) .. (56.32,25.08) .. controls (54.8,25.08) and (53.57,23.84) .. (53.57,22.32) -- cycle ;
\draw [color={rgb, 255:red, 208; green, 2; blue, 27 }  ,draw opacity=1 ]   (54.65,26.07) .. controls (34.56,83.94) and (115.57,39.43) .. (63.13,24.51) ;
\draw [shift={(61.5,24.07)}, rotate = 14.54] [color={rgb, 255:red, 208; green, 2; blue, 27 }  ,draw opacity=1 ][line width=0.75]    (6.56,-1.97) .. controls (4.17,-0.84) and (1.99,-0.18) .. (0,0) .. controls (1.99,0.18) and (4.17,0.84) .. (6.56,1.97)   ;
\draw [color={rgb, 255:red, 245; green, 166; blue, 35 }  ,draw opacity=1 ]   (63.04,97.75) .. controls (61.92,95.68) and (62.4,94.08) .. (64.47,92.96) .. controls (66.54,91.84) and (67.02,90.24) .. (65.91,88.17) -- (67.22,83.79) -- (67.22,83.79) ;
\draw [color={rgb, 255:red, 245; green, 166; blue, 35 }  ,draw opacity=1 ]   (68.65,63.67) .. controls (56.06,74.6) and (35.28,35.72) .. (55.69,86) ;
\draw [shift={(56.65,88.36)}, rotate = 247.77] [fill={rgb, 255:red, 245; green, 166; blue, 35 }  ,fill opacity=1 ][line width=0.08]  [draw opacity=0] (5.36,-2.57) -- (0,0) -- (5.36,2.57) -- cycle    ;
\draw    (52.36,23.79) .. controls (29.18,48.61) and (39.8,80.71) .. (57.29,98.98) ;
\draw [shift={(58.65,100.36)}, rotate = 224.55] [color={rgb, 255:red, 0; green, 0; blue, 0 }  ][line width=0.75]    (6.56,-1.97) .. controls (4.17,-0.84) and (1.99,-0.18) .. (0,0) .. controls (1.99,0.18) and (4.17,0.84) .. (6.56,1.97)   ;
\draw    (64.65,99.5) .. controls (98.31,87.9) and (123.01,4.5) .. (62.51,19.86) ;
\draw [shift={(60.65,20.36)}, rotate = 344.4] [color={rgb, 255:red, 0; green, 0; blue, 0 }  ][line width=0.75]    (6.56,-1.97) .. controls (4.17,-0.84) and (1.99,-0.18) .. (0,0) .. controls (1.99,0.18) and (4.17,0.84) .. (6.56,1.97)   ;
\draw  [fill={rgb, 255:red, 0; green, 0; blue, 0 }  ,fill opacity=1 ] (58.29,100.93) .. controls (58.29,99.41) and (59.52,98.18) .. (61.04,98.18) .. controls (62.56,98.18) and (63.79,99.41) .. (63.79,100.93) .. controls (63.79,102.45) and (62.56,103.68) .. (61.04,103.68) .. controls (59.52,103.68) and (58.29,102.45) .. (58.29,100.93) -- cycle ;
\draw [color={rgb, 255:red, 245; green, 166; blue, 35 }  ,draw opacity=1 ]   (75.22,85.95) .. controls (98.72,59.91) and (87.01,51.06) .. (68.65,63.67) ;
\draw [color={rgb, 255:red, 245; green, 166; blue, 35 }  ,draw opacity=1 ]   (62.93,35.95) .. controls (60.78,35) and (60.17,33.45) .. (61.12,31.29) .. controls (62.07,29.13) and (61.47,27.58) .. (59.31,26.63) -- (58.65,24.93) -- (58.65,24.93) ;
\draw [color={rgb, 255:red, 245; green, 166; blue, 35 }  ,draw opacity=1 ]   (58.65,38.52) .. controls (52.37,51.08) and (80.49,52.68) .. (69.6,37.37) ;
\draw [shift={(67.79,35.1)}, rotate = 48.95] [fill={rgb, 255:red, 245; green, 166; blue, 35 }  ,fill opacity=1 ][line width=0.08]  [draw opacity=0] (5.36,-2.57) -- (0,0) -- (5.36,2.57) -- cycle    ;
\draw (81.29,44.3) node [anchor=north west][inner sep=0.75pt]  [font=\scriptsize]  {$\textcolor[rgb]{0.82,0.01,0.11}{\alpha }$};
\draw (28.86,10.26) node [anchor=north west][inner sep=0.75pt]  [font=\scriptsize]  {$\textcolor[rgb]{0.29,0.56,0.89}{v}$};
\draw (32.57,63.69) node [anchor=north west][inner sep=0.75pt]  [font=\scriptsize]  {$b$};
\draw (104,60.26) node [anchor=north west][inner sep=0.75pt]  [font=\scriptsize]  {$c$};
\draw (65.43,68.54) node [anchor=north west][inner sep=0.75pt]  [font=\tiny]  {$\textcolor[rgb]{0.96,0.65,0.14}{f_{2}}$};
\draw (82.29,27.69) node [anchor=north west][inner sep=0.75pt]  [font=\tiny]  {$\textcolor[rgb]{0.96,0.65,0.14}{f_{1}}$};
\end{tikzpicture}
\end{align}
as edge reversals commute with the module and comodule structures at the vertices and faces and with the morphisms $\eta_\alpha$ and $\epsilon_\alpha$. 
We then compute
\begin{align*}
\delta_{f_2}\circ \eta_\alpha(b\oo c)=\delta_{f_2}(1\oo b\oo c)=\low b 1\low c 1\oo 1\oo \low b 2\oo \low c 2=(1_H\oo\eta_\alpha)\circ \delta^+_{f_2}(b\oo c).
\end{align*}
The computations for graphs with different numbers of edges in $f_2$ are analogous, and
with the identity $\delta_{f_2}^+\circ \iota_{f_2}^+=(\eta\oo 1_{H^{\oo(E-1)}})\circ \iota_{f_2}^+$  we obtain the second identity in \eqref{eq:LoopDeltaF2EtaAlpha}.

2.~We prove the identities \eqref{eq:LoopRemovalWholePiPlus} to \eqref{eq:LoopRemovalWholeIotaPlus}.
To show \eqref{eq:LoopEtaLoopIotaIsIota}
it is  sufficient to consider  the graph \eqref{eq:facegraphhelp} with
\begin{align*}
&\delta_{f_1}(\alpha\oo b\oo c)=\low\alpha 1\oo \low \alpha 2\oo b\oo c\\
&\delta_{f_2}(\alpha\oo b\oo c)=\low b 1 S(\low
\alpha 2)\low c 1\oo \low \alpha 1\oo \low b 2\oo \low c 2.
\end{align*}
As $f_1, f_2 \in \mathcal F$, this yields 
\begin{align*}
\iota_{\mathcal F}=((\epsilon\circ \eta) \oo 1_{H^{\oo E}})\circ \iota_{\mathcal F}=(\epsilon\oo 1_{H^{\oo E}})\circ \delta_{f_2}\circ \iota_{\mathcal F}\overset{(\ast)}=\eta_\alpha\circ \epsilon_\alpha\circ \iota_{\mathcal F},
\end{align*}
where we apply in  $(\ast)$  the coinvariance 
under $\delta_{f_1}$.

Identity \eqref{eq:LoopRemovalWholePiPlus} follows inductively from the identity $\pi_{\mathcal V}^+ \circ \epsilon_{\alpha}\circ \rhd_z = \pi_{\mathcal V}^+ \circ (\epsilon \oo \epsilon_{\alpha})$ for all vertices $z\in V$, which is obtained for $z\neq v$ by post-composing the first identity in \eqref{eq:LoopActionEtaAlpha} with $\pi_{\mathcal V}^+ = \chi_{z,\mathcal V}^+ \circ \pi_z^+$ and for $z=v$ by post-composing \eqref{eq:LoopRemovalPiPlusAtV} with $\chi_{v,\mathcal V}$. 

Identity \eqref{eq:LoopRemovalWholeIota} follows from \eqref{eq:LoopEtaLoopIotaIsIota} and the first identities in  \eqref{eq:LoopCoactionEtaAlpha}, \eqref{eq:LoopDeltaF2EtaAlpha}, which yield for all $f\in \mathcal F\setminus\{f_1\}$
\begin{align*}
\delta_f^+\circ \epsilon_\alpha\circ \iota_{\mathcal F}&=(1_H\oo (\epsilon_\alpha\circ \eta_\alpha))\circ \delta_f^+\circ \epsilon_\alpha\circ \iota_{\mathcal F}\overset{\eqref{eq:LoopCoactionEtaAlpha}, \eqref{eq:LoopDeltaF2EtaAlpha}}=(1_H\oo \epsilon_\alpha)\circ \delta_f\circ \eta_\alpha\circ \epsilon_\alpha\circ \iota_{\mathcal F}\\&\overset{\eqref{eq:LoopEtaLoopIotaIsIota}}= (1_H\oo \epsilon_\alpha)\circ \delta_f\circ \iota_{\mathcal F}= (1_H \oo \epsilon_{\alpha} )\circ \delta_f \circ \iota_f \circ \xi_{f,\mathcal F}= (\eta \oo \epsilon_{\alpha}) \circ \iota_{\mathcal F}.
\end{align*}
Identity \eqref{eq:LoopRemovalWholePi} follows inductively by applying the identity
 $\pi_{\mathcal V}\circ \eta_{\alpha}\circ \rhd_z^+ = \pi_{\mathcal V}\circ (\epsilon \oo \eta_{\alpha})$ for $z\in V$, obtained by post-composing \eqref{eq:LoopActionEtaAlpha} with $\pi_{\mathcal V}=\chi_{z,\mathcal V}\circ\pi_z$  for $z\neq v$ and  \eqref{eq:LoopRemovalPiAtV} with $\chi_{v,\mathcal V}$ for $z=v$. 

 Analogously, \eqref{eq:LoopRemovalWholeIotaPlus} follows from the identity  $\delta_f \circ \eta_{\alpha}\circ \iota_{\mathcal F}^+ = (\eta\oo\eta_{\alpha}) \circ \iota_{\mathcal F}^+$ for  $f\in F$, which is obtained for  $f=f_1$  by pre-composing \eqref{eq:LoopDeltaF1EtaAlpha} with $\iota_{\mathcal F}^+$, for $f=f_2$ by pre-composing \eqref{eq:LoopDeltaF2EtaAlpha} with $\xi_{f_2,\mathcal F}^+$ and for $f\notin \{f_1, f_2\}$ by pre-composing \eqref{eq:LoopCoactionEtaAlpha} with $\iota_{\mathcal F}^+ = \iota_f^+ \circ \xi_{f,\mathcal F}^+$.
\end{proof}

We now apply Lemma \ref{Lemma:LoopRemovalAndCoInvarianceMorphisms} to show that loop deletions induce morphisms between the  invariants of  $\rhd_{\mathcal V}$ for
subsets $\emptyset\neq \mathcal V\subset V$. If $\mathcal F$ contains the two faces adjacent to the loop, they
also induce isomorphisms between  the
coinvariants of $\delta_\mathcal F$  and isomorphisms between the protected objects.

For this we denote by $\Gamma^+$ the graph obtained by deleting a loop $\alpha$ in $\Gamma$. For $\Gamma$ we use the  notation from Proposition \ref{prop:EdgeContractionInducesIso}. For $\Gamma^+$ we  denote by $\mathcal M^{+coH}$, $\mathcal M^{+H}$, $\mathcal M^+_{inv}$ the coinvariants, invariants and biinvariants of $\delta^+_{\mathcal F}$, $\rhd^+_{\mathcal V}$, by 
 $\iota^+_{\mathcal F}$ and $\pi^+_{\mathcal V}$ the associated equaliser and coequaliser and 
by  $I^+: \mathcal M^+_{inv}\to \mathcal M^{+H}$
and $P^+: \mathcal M^{+coH}\to\mathcal M^+_{inv}$  the monomorphism and epimorphism that characterise $\mathcal M^+_{inv}$ as the image of $\pi^+_{\mathcal V}\circ \iota^+_{\mathcal F}$.

\begin{proposition}\label{prop:LoopRemovalInducesIso}
Let $ \Gamma^+$ be obtained from   $\Gamma$ by removing an isolated loop $\alpha$ with  incident faces $f_1,f_2$.
Then for all subsets $\emptyset\neq\mathcal V \subset V$, $\{f_1, f_2\}\subset \mathcal F \subset F$ the loop removal induces 
\begin{compactitem}
\item an isomorphism $y: M_{\mathcal F}^{coH} \to M^{+coH}_{\mathcal F}$ with $\iota_{\mathcal F}^+ \circ y=\epsilon_{\alpha}\circ \iota_{\mathcal F}$,
\item a morphism $t:M_{\mathcal V}^{H}\to M_{\mathcal V}^{+H}$ with a right inverse and  $t\circ \pi_{\mathcal V}= \pi_{\mathcal V}^+ \circ \epsilon_{\alpha}$
\item an isomorphism $\psi_{inv}:M_{inv}\to M^+_{inv}$ with $I= t^\inv \circ I^+\circ \psi_{inv}$. 
\end{compactitem}
\end{proposition}
\begin{proof}
Equation \eqref{eq:LoopRemovalWholeIota} and the universal property of the equaliser $ \iota_{\mathcal F}^+$ yield a unique morphism $y: M_{\mathcal F}^{coH}\to M_{\mathcal F}^{+coH}$  with $\iota_{\mathcal F}^+ \circ y=\epsilon_{\alpha}\circ \iota_{\mathcal F}$. The universal property  of the equaliser $\iota_{\mathcal F}$ and \eqref{eq:LoopRemovalWholeIotaPlus} provide a unique morphism $y^\inv: M_{\mathcal F}^{+coH}\to M_{\mathcal F}^{coH}$ with
$
\iota_{\mathcal F}\circ y^\inv = \eta_{\alpha}\circ \iota_{\mathcal F}^+ 
$. 
The two morphisms are inverse to each other, as $\iota_{\mathcal F}$, $\iota_{\mathcal F}^+$ are monomorphisms and 
\begin{align*}
\iota_{\mathcal F }^+ \circ y \circ y^\inv &= \epsilon_{\alpha}\circ \iota_{\mathcal F}\circ y^\inv = \epsilon_{\alpha}\circ \eta_{\alpha}\circ \iota_{\mathcal F}^+ =\iota_{\mathcal F}^+\\
\iota_{\mathcal F}\circ y^\inv \circ y &= \eta_{\alpha}\circ \iota_{\mathcal F}^+ \circ y = \eta_{\alpha}\circ \epsilon_{\alpha}\circ \iota_{\mathcal F} \overset{\eqref{eq:LoopEtaLoopIotaIsIota}}{=} \iota_{\mathcal F}.
\end{align*}
Similarly, equation \eqref{eq:LoopRemovalWholePiPlus} and the universal  property of the coequaliser $ \pi_{\mathcal V}$ yield a unique morphism $t:M^H_{\mathcal V}\to M_{\mathcal V}^{+H}$ with $t\circ \pi_{\mathcal V}=\pi_{\mathcal V}^+\circ \epsilon_{\alpha}$
and equation \eqref{eq:LoopRemovalWholePi} with the universal property of the 
coequaliser $ \pi_{\mathcal V}^+$  a unique morphism $t^\inv: M_{\mathcal V}^{+H}\to M_{\mathcal V}^H$ with
$
t^\inv \circ \pi_{\mathcal V}^+ = \pi_{\mathcal V}\circ \eta_{\alpha}. 
$ 
The morphism $t^\inv$ is a right inverse of $t$, since $\pi_{\mathcal V}^+$ is an epimorphism and
\begin{align*}
t \circ t^\inv \circ \pi_{\mathcal V}^+ = t \circ \pi_{\mathcal V}\circ \eta_{\alpha} = \pi_{\mathcal V}^+ \circ \epsilon_{\alpha}\circ \eta_{\alpha} = \pi_{\mathcal V}^+.
\end{align*}
We have constructed commuting diagrams 
$$
\xymatrix{
M^{coH}_{\mathcal F} \ar[r]^{\iota_{\mathcal F}} \ar[d]^{y} & M  \ar[r]^{\pi_{\mathcal V}} \ar[d]^{\epsilon_{\alpha}}  & M^H_{\mathcal V} \ar[d]^{t}\\
M^{+coH}_{\mathcal F} \ar[r]_{\iota_{\mathcal F}^+}  & M^+ \ar[r]_{\pi_{\mathcal V}^+} & M^{+H}_{\mathcal V}.
}\qquad\qquad 
\xymatrix{
M^{+coH}_{\mathcal F} \ar[r]^{\iota^+_{\mathcal F}} \ar[d]^{y^\inv} & M^+  \ar[r]^{\pi^+_{\mathcal V}} \ar[d]^{\eta_{\alpha}}  & M^{+H}_{\mathcal V} \ar[d]^{t^\inv}\\
M^{coH}_{\mathcal F} \ar[r]_{\iota_{\mathcal F}}  & M \ar[r]_{\pi_{\mathcal V}} & M^H_{\mathcal V}
}
$$
The construction of $\psi_{inv}: M_{inv}\to M^+_{inv}$ and its inverse is as in the proof of Proposition \ref{prop:EdgeContractionInducesIso} with the monomorphisms $j=t^\inv\circ I^+$, $j^+=t\circ I$ and epimorphisms $q=P^+\circ y$, $q^+=P\circ y^\inv$.
\end{proof}

\begin{corollary}\label{Cor:loopdel}
Deletions of isolated loops induce isomorphisms between the protected objects.
\end{corollary}

\subsection{Protected objects}
\label{subsec:topinv}

Combining the results from Section \ref{subsec:edge slides} to \ref{subsec:loops} one has that ciliated ribbon graphs related by moving cilia, edge reversals, edge contractions and deletions of isolated loops have isomorphic protected objects. As these are sufficient to relate any connected ribbon graph to the standard graph from \eqref{eq:standardgraph},  the protected object  of a ciliated ribbon graph is determined up to isomorphisms by the genera of the connected components of the associated surface. 

\begin{theorem}\label{th:topinv}
The isomorphism class of the protected object for an involutive Hopf monoid $H$ and a ciliated ribbon graph $\Gamma$ depends only on $H$ and the homeomorphism class of the surface for $\Gamma$.
\end{theorem}
\begin{proof} 
By Lemma \ref{lemma:InvariantsIndependentFromCilium} the invariants, coinvariants  and hence the protected object of a ciliated ribbon graph are independent of the choice of the cilia. 
By Proposition \ref{prop:standardgraph} every  ribbon graph can be transformed into a disjoint union of standard graphs by edge reversals, edge contractions, edge slides and removing isolated loops.  
In each step the cilia can be arranged in such a way that no edge ends slide over cilia. 
By Corollaries \ref{cor:edgereverssal}, \ref{Coro:EdgeSlideIsoOnProtectedObject}, \ref{Coro:EdgeContracionIso} and \ref{Cor:loopdel} these graph transformations induce  isomorphisms between the protected objects.
\end{proof}

As the protected object is a topological invariant, one can use any embedded graph whose complement is a disjoint union of discs to compute the protected object.
For a sphere, the simplest such graph consists of a single isolated vertex. This is associated with the trivial $H$-(co)module  structure on $e$ given by the (co)unit of $H$ and yields the tensor unit as protected object.

\begin{example}\label{ex:sphere}  The protected object for a sphere $S^2$ is the tensor unit of $H$: $\mathcal M_{inv}=e$.
\end{example}

We now focus on  oriented surfaces $\Sigma$ of genus $g\geq 1$ and use  the standard graphs \eqref{eq:standardgraph} to determine their protected objects. The associated module and comodule structures are given in Example \ref{ex:modcomodpi_1gen} and form a Yetter-Drinfeld module.

\begin{example}\label{ex:set}
For a group $H$ as a Hopf monoid in $\mac=\Set$ the coinvariants are the set of group homomorphisms from $\pi_1(\Sigma)$ to $H$
\begin{align}
M^{coH}= \{(a_1, b_1, \ldots, a_g, b_g)\in H^{\times 2g}: [b_g^\inv, a_g]\cdot \ldots \cdot [b_1^\inv, a_1]=1\} \cong \mathrm{Hom}(\pi_1(\Sigma), H). \label{eq:CoinvariantsGroupHomos}
\end{align}
The invariants are the set of orbits for the  conjugation action $\rhd$ from \eqref{eq:YDmodulegroup} on $H^{\times 2g}$, and the protected object  is the {\em representation variety}  or {\em moduli space of flat $H$-bundles}
$M_{inv}\cong \mathrm{Hom}(\pi_1(\Sigma), H)/H$.
\end{example}

\begin{example}\label{ex:top}
For a topological group $H$ as a Hopf monoid in $\mathcal C=\mathrm{Top}$ the protected object is $M_{inv}\cong \mathrm{Hom}(\pi_1(\Sigma), H)/H$ as a set by Example \ref{Ex:ImageInSet,Top}. It is equipped with the quotient topology induced by the canonical surjection $\pi:\mathrm{Hom}(\pi_1(\Sigma), H)\to \mathrm{Hom}(\pi_1(\Sigma), H)/H$ and the compact-open topology on $\mathrm{Hom}_{\mathrm{Top}}(\pi_1(\Sigma), H)$ for the discrete topology on $\pi_1(\Sigma)$.
\end{example}

\begin{example}\label{ex:GSet}
For a Hopf monoid $H$ in $\mathcal C=G\mathrm{-Set}\, = \mathrm{Set}^{\mathrm{B}G}$ the coinvariants for the  comodule structure  $\delta$ from Example \ref{ex:modcomodpi_1gen} are  the set \eqref{eq:CoinvariantsGroupHomos} 
with the diagonal $G$-action. The invariants for the module structure  $\rhd$ 
are the associated orbit space. By  Example \ref{Ex:ImageInSet,Top}, 2. 
the protected object is the representation variety $M_{inv}\cong \mathrm{Hom}(\pi_1(\Sigma), H)/H$ with the induced $G$-set structure.
\end{example}

\begin{example}\label{ex:k-Mod}
Let $k$ be a commutative ring, $\mac=k\text{-Mod}$ and $G$ a finite group.

For the group algebra $H=k[G]$ as a Hopf monoid in $\mac$ and the standard graph in \eqref{eq:standardgraph} one has $M=k[G]^{2g}\cong k[G^{\times 2g}]$. The Yetter-Drinfeld module structure of $M$ is given by  \eqref{eq:YDmodulegroup}  on a basis. The coinvariants and invariants are 
\begin{align*}
M^{coH}&=\langle\{(a_1,b_1,\ldots, a_g, b_g)\mid [b_g^\inv,a_g]\cdots[b_1^\inv, a_1]=1\}\rangle_k\cong \langle \Hom(\pi_1(\Sigma),G)\rangle_k\\
M^H&=k[G^{\times 2g}]/\langle  \{(a_1,\ldots, b_g)-(ha_1h^\inv, \ldots ,hb_g h^\inv)\mid a_1,b_1,\ldots, a_g,b_g,h\in G\}\rangle,
\end{align*} 
and the protected object is the free $k$-module generated by the representation variety $\Hom(\pi_1(\Sigma),G)/G$
\begin{align}
M_{inv}=\langle \Hom(\pi_1(\Sigma), G)/G\rangle_k.
\end{align}
For the dual Hopf monoid $H=k[G]^*=\mathrm{Map}(G,k)$ of maps from $G$ to $k$ with Hopf monoid structure
\begin{align}\label{eq:dualhopf}
&\delta_g\cdot \delta_h=\delta_g(h)\delta_g, & &1=\sum_{g\in G} \delta_g, & &\Delta(\delta_g)=\sum_{x,y\in G, xy=g} \delta_x\oo \delta_y, & &\epsilon(\delta_g)=\delta_g(e), & &S(\delta_g)=\delta_{g^\inv}
\end{align}
one has $M=\mathrm{Map}(G,k)^{\oo 2g}\cong \mathrm{Map}(G^{\times 2g},k)$ with the Yetter-Drinfeld module structure
\begin{align}
\delta_h\rhd(\delta_{a_1}\oo\delta_{b_1}\oo\ldots\oo \delta_{a_g}\oo \delta_{b_g})=\delta_h([b_g^\inv, a_g]\cdots [b_1^\inv, a_1]) \,\delta_{a_1}\oo\delta_{b_1}\oo\ldots\oo \delta_{a_g}\oo \delta_{b_g}\\
\delta(\delta_{a_1}\oo\delta_{b_1}\oo\ldots\oo \delta_{a_g}\oo \delta_{b_g})=\Sigma_{h\in G} \delta_{h^\inv} \oo \delta_{ha_1h^\inv}\oo \delta_{hb_1h^\inv}\oo\ldots\oo \delta_{ha_gh^\inv}\oo \delta_{hb_gh^\inv}\nonumber
\end{align} 
computed from \eqref{eq:ydgen} and \eqref{eq:dualhopf}. It follows that the coinvariants and invariants are given by
\begin{align}
M^{coH}&=\mathrm{Map}(G^{\times 2g},k)^G\\
M^H
&=\{f: G^{\times 2g}\to k\mid \mathrm{supp}(f)\subseteq\{(a_1,\ldots, b_g)\mid [b_g^\inv, a_g]\cdots [b_1^\inv, a_1]=1\}\},\nonumber
\end{align}
and the protected object is the set of functions 
\begin{align}
M_{inv}=\mathrm{Map}(\Hom(\pi_1(\Sigma), G)/G, k).
\end{align}
\end{example}

Example \ref{ex:k-Mod} shows that the protected object in this article indeed generalises the protected space of Kitaev's quantum double models. If one sets $k=\C$ in Example \ref{ex:k-Mod} one obtains precisely the protected space for Kitaev's quantum double model for the group algebra $\C[G]$ and its dual, see \cite[Sec.~4]{Ki}. However, Example \ref{ex:k-Mod} also yields an analogous result for any commutative ring $k$, for  which the usual quantum double models are not defined.

\section{Protected objects in SSet}
\label{sec:sset}

In this section, we investigate protected objects for group objects in the category $\SSet$. We denote by $\Delta$ the simplex category with finite ordinals $[n]=\{0,1,\ldots, n\}$ for $n\in\N_0$ as objects and weakly monotonic maps $\alpha:[m]\to [n]$ as morphisms from $[m]$ to $[n]$. 

Objects in $\SSet=\Set^{\Delta^{op}}$ are simplicial sets,  functors $X:\Delta^{op}\to\Set$ that are specified by  sets $X_n$,  face maps $d_i: X_{n+1}\to X_{n}$  and degeneracies $s_i: X_n\to X_{n+1}$ for $n\in\N_0$ and $i\in\{0,\ldots, n\}$ that satisfy the simplicial relations 
\begin{align}\label{eq:simprels}
&d_j\circ d_i=d_i\circ d_{j+1}\text{ if } i\leq j, &  &s_i\circ s_j=s_{j+1}\circ s_i\text{ if }i\leq j, \\
&d_i\circ s_j=s_{j-1}\circ d_i \text{ if }  i<j, & 
&d_i\circ s_j=\id  \text{ if } i\in \{j,j+1\}, &
&d_i\circ s_j=s_j\circ d_{i-1} \text{ if } i>j+1.\nonumber
\end{align}

Morphisms in $\SSet$ are simplicial maps,  natural transformations $f: X\to Y$  specified by  component maps $f_n: X_n\to Y_n$ satisfying $f_{n-1}\circ d_i=d_{i} \circ f_n$ and $f_{n+1}\circ s_i=s_i\circ f_n$ for  $n\in\N_0$ and admissible $i$. 
The category $\SSet$ is cartesian monoidal with the objectwise product induced by the product in $\Set$. 

Unpacking the definition of a group object in a cartesian monoidal category from Example \ref{ex:groupobjex}  yields

\begin{definition} \label{lem:simplicialgroup} $\quad$
\begin{compactenum}
\item A group object in $\SSet$ is a {\bf simplicial group}: a simplicial set $H:\Delta^{op}\to \Set$ with  group structures on the sets $H_n$  such that all face maps  and degeneracies  are group homomorphisms. 
\item A morphism of group objects in $\SSet$ is a {\bf morphism of simplicial groups}:  a simplicial map $f: H\to H'$ such that all maps $f_n: H_n\to H'_n$ are group homomorphisms. 
\end{compactenum}
\end{definition}

For examples of simplicial groups, see Section \ref{subsec:coequcat},  in particular Corollary \ref{cor:nervegroup}  and Example \ref{ex:crossedmod}. 
Modules, comodules and Yetter-Drinfeld modules over simplicial groups are given by Example  \ref{Ex:YDAndHopfModulesInCat}. 

\begin{lemma} Let $H:\Delta^{op}\to\Set$ be a simplicial group.
\begin{compactenum}
\item A {\bf module} over  $H$ is a  simplicial set $M:\Delta^{op}\to \Set$ together with a collection of $H_n$-actions $\rhd_n: H_n\times M_n\to M_n$ that define a simplicial map $\rhd: H\times M\to M$. 
\item A {\bf comodule} over $H$  is a  simplicial set $M:\Delta^{op}\to \Set$  
with a simplicial map $F: M\to H$. 

\item If $(M,\rhd)$ is a module and $(M,F)$ a comodule over $H$, then  $(M,\rhd, F)$ is a
 {\bf Yetter-Drinfeld module} over $H$ iff $F_n(g\rhd_n m)=g\cdot F_n(m) \cdot g^\inv$ for all $m\in M_n$, $g\in H_n$ and $n\in\N_0$. 

\end{compactenum}
\end{lemma}

As (co)limits in $\SSet$ are  objectwise, see for instance Riehl \cite[Prop.~3.3.9]{R} or Leinster \cite[Th.~6.2.5]{L}, (co)invariants of a (co)module over a group object in $\SSet$ are obtained from  (co)equalisers in $\Set$.  It is also straightforward to compute the biinvariants of a Yetter-Drinfeld module.

\begin{proposition}\label{prop:binvsimplicial} Let $H$ be a simplicial group.
\begin{compactenum}
\item The coinvariants $\mathcal M^{coH}$ of  a $H$-comodule $M$   defined by a simplicial map $F: M\to H$ are given by 
the sets $M^{coH}_n=\{m\in M_n\mid F_n(m)=e\}$ and the induced face maps and degeneracies.

\item The invariants $M^H$  of a $H$-module $(M,\rhd)$ are given by the sets  $M^H_n=\{H_n\rhd_n m\mid m\in M_n\}$ and the induced face maps and degeneracies.

\item The biinvariants $M_{inv}$ of a Yetter-Drinfeld module  $(M, \rhd, F)$ over $H$ are given by the sets  $(M_{inv})_n=\{H_n\rhd_n m\mid m\in M_n, F_n(m)=e\}$ and the induced  face maps and degeneracies.
\end{compactenum}
\end{proposition}
  
  \begin{proof} 1.~The coinvariant object of a $H$-comodule $(M,F)$ is the equaliser of the simplicial maps $\delta=F\times\id: M\to H\times M$ and $\eta\times \id: M\to H\times M$. As limits in $\SSet$ are objectwise, this is the simplicial set $M^{coH}:\Delta^{op}\to \Set$ that assigns to an ordinal $[n]$ the equaliser in $\Set$ of the maps $F_n\times \id: M_n\to H_n\times M_n$ and $\eta_n\times \id: M_n\to H_n\times M_n$, which is $M^{coH}_n=\{m\in M_n\mid F_n(m)=e\}$. 
  The face maps and degeneracies are induced by the ones of $M$, and
  the simplicial map $\iota: M^{coH}\to M$ is given by the maps  $\iota_n: M^{coH}_n\to M_n$, $m\mapsto m$.

2.~Analogously to 1., the invariant object of  $(M,\rhd)$ is the simplicial set $M^H:\Delta^{op}\to\Set$ that assigns to the ordinal $[n]$ the coequaliser in $\Set$ of the maps $\rhd_n: H_n\times M_n\to M_n$ and $\epsilon_n\times\id: H_n\times M_n\to M_n$. This is the set $M^H_n=M_n/\sim_n$  with $m\sim_n m'$ iff there is a $g\in H_n$ with $m'=g\rhd_n m$. 
 The simplicial map $\pi: M\to M^H$ is given by the maps  $\pi_n: M_n\to M^H_n$, $m\mapsto H_n\rhd_n m$. 
  
  3.~The simplicial maps $I: M_{inv}\to M^H$ and $P: M^{coH}\to M_{inv}$ with
$\pi\circ \iota=I\circ P$ that characterise $M_{inv}$ with $(M_{inv})_n=\{H_n \rhd_n m \mid m \in M^{coH}_n\}$ as the image of $\pi\circ\iota$  are given by 
  $$I_n: (M_{inv})_n\to M^H_n, H_n\rhd_n m\mapsto H_n\rhd_n m, \qquad P_n: M^{coH}_n\to (M_{inv})_n, m\mapsto H_n\rhd_n m.$$ 
  As monomorphisms and epimorphisms in $\SSet$ are those simplicial maps whose component morphisms are injective and surjective, see for instance \cite[Ex.~6.2.20]{L}, it follows directly that $I$ is a monomorphism and $P$ an epimorphism in $\SSet$. Every pair $(J,Q)$ of a monomorphism $J: X\to M^H$ and morphism $Q: M^{coH}\to X$ in $\SSet$ with $J\circ Q=\pi\circ \iota$ defines  injective maps $J_n: X_n\to M^H_n$ and thus identifies $Q(M^{coH}_n)$ with a subset of $M^H_n$. 
As $J_n$ is a monomorphism and due to the identity $J_n\circ Q_n(g\rhd_n m)=\pi_n\circ \iota_n(g\rhd_n m)=\pi_n\circ\iota_n(m)=J_n\circ Q_n(m)$,  we have $Q_n(g\rhd_n m)=Q_n(m)$  for all $m\in M_n^{coH}$ and $g\in H_n$. The maps
    $V_n: (M_{inv})_n\to X_n$,  $H_n\rhd_n m\mapsto Q_n(m)$ define a simplicial map $V: M_{inv}\to X$  with $I=J\circ V$.
 \end{proof}

We now determine the  coinvariants, invariants and the protected objects for Kitaev models on oriented surfaces $\Sigma$ of genus $g\geq 1$ and for a simplicial group $H$ as a Hopf monoid in $\SSet$. 

\begin{proposition}\label{th:simplicialprot} Let $H$ be a simplicial group and $\Sigma$ an oriented surface of genus $g\geq 1$. The  associated protected object  is the simplicial set $X:\Delta^{op}\to \Set$ with $X_n=\Hom(\pi_1(\Sigma), H_n)/H_n$, where the quotient is with respect to conjugation by $H_n$,  and face maps and degeneracies given by
$$
d_i: X_n\to X_{n-1},\; [\rho]\mapsto [d_i\circ \rho],\qquad s_i: X_n\to X_{n+1},\; [\rho]\mapsto [s_i\circ \rho].
$$ 
\end{proposition}

\begin{proof} By Theorem \ref{th:topinv} the protected object of $\Sigma$  can be computed from the standard graph in \eqref{eq:standardgraph}. This yields a Yetter-Drinfeld module $(M,\rhd,F)$ over $H$ given by formula \eqref{eq:YDmodulegroup} in Example \ref{ex:modcomodpi_1gen}. Hence, we have 
$M_n=H_n^{\times 2g}$ for all $n\in\N_0$ with the face maps and degeneracies of $H$ applied to each component simultaneously. The Yetter-Drinfeld module structure is given by
\begin{align*}
&F_n: H_n^{\times 2g}\to H_n, \;(a_1,b_1,\ldots, a_g,b_g)\mapsto[b_g^\inv, a_g]\cdots[b_1^\inv,a_1]\\
&\rhd_n: H_n\times H_n^{\times 2g}\to H_n^{\times 2g},\; (h, a_1,b_1,\ldots, a_g,b_g)\mapsto (ha_1h^\inv, hb_1h^\inv,\ldots, ha_g h^\inv, hb_gh^\inv). 
\end{align*}
  By Proposition  \ref{prop:binvsimplicial} the associated  protected object is the simplicial set $\mathcal M_{inv}$ with 
  $$
  (M_{inv})_n=\{ H_n\rhd_n (a_1,b_1,\ldots, a_g,b_g)\in H_n^{\times 2g}\mid [b_g^\inv, a_g]\cdots [b_1^\inv,a_1]=e\}\cong \Hom(\pi_1(\Sigma), H_n)/H_n.
  $$
Face maps and degeneracies are given by post-composing  group homomorphisms $\rho:\pi_1(\Sigma)\to H_n$ with the face maps $d_i: H_n\to H_{n-1}$ and degeneracies $s_i: H_n\to H_{n+1}$. 
\end{proof}

\section{Protected objects in Cat}
\label{sec:cat}

\subsection{Crossed modules as group objects in Cat}
\label{subsec:crossed modules}

We consider the category Cat of small categories and functors between them as a cartesian monoidal category with terminal object  
$\{\cdot\}$. For a finite product $\mathcal C_1\times\ldots\times \mathcal C_n$ of small categories, we denote by $\pi_i: \mathcal C_1\times\ldots\times \mathcal C_n\to \mathcal C_i$ the associated projection functors. For a small category $\mathcal C$ we denote by $\mathrm{Ob}(\mathcal C)$ the set of objects and by $\mathcal C^{(1)}=\bigcup_{X, Y \in \mathrm{Ob}(\mathcal C)}\mathrm{Hom}_{\mathcal C}(X, Y)$ the set of all morphisms in $\mathcal C$.

\begin{definition}\label{Def:GroupObjectInCat} $\quad$
\begin{compactenum}
\item A \textbf{group object} in Cat is a  small category $H$ together with functors $m: H \times H \to H$,  $\eta: \{\cdot\}\to H$ and $I:H \to H$ such that the  diagrams 
\eqref{eq:groupobjdiags} commute.

\item A \textbf{morphism} $F: (H, m, \eta, I)\to (H', m', \eta', I')$ \textbf{of group objects}  is a functor $F: H \to H'$ 
that satisfies \eqref{eq:groupmorphism}.
\end{compactenum}
We denote by $\mathcal G(\mathrm{Cat})$ the category of group objects  and morphisms of group objects in  Cat  and write $e:=\eta(\cdot)$,   $f^\inv=I(f)$, $g\cdot f=m(g,f)$ and likewise for multiple products.
\end{definition}

Brown und Spencer \cite{BS} showed that  group objects in $\Cat$ correspond to crossed modules. We summarise this correspondence for the convenience of the reader.

\begin{definition}\label{Def:CrossedModule}$ $\\
A \textbf{crossed module} is a quadruple $(B, A, \brhd, \partial)$ of groups $A$ and $B$, a group homomorphism $\partial:A \to B$ and a group action $\brhd: B\times A \to A$ by automorphisms that satisfy the Peiffer identities
\begin{align}\label{eq:crossedmod}
\partial(b\brhd a) =b \partial(a) b^\inv,\qquad
\partial(a)\brhd a'= aa'a^\inv \qquad \forall a, a'\in A, b\in B.
\end{align}
A \textbf{morphism of crossed modules} $f=(f_1, f_2):(B, A, \brhd, \partial)\to (B', A',  \brhd', \partial')$ is a pair of group homomorphisms $f_1: B\to B'$, $f_2: A \to A'$ such that 
\begin{equation*}
\partial' \circ f_2= f_1 \circ \partial\text{,} \qquad \brhd' \circ (f_1 \times f_2)=f_2 \circ \brhd\text{.}
\end{equation*}
We denote by $\mathcal{CM}$ the category of crossed modules and morphisms between them.
\end{definition}

\vspace*{0.2cm}

\begin{example} \label{ex:crossedmodules} $\quad$
\begin{compactenum}
\item A normal subgroup $A\subset B$ defines a crossed module with the inclusion  $\partial: A \to B$, $a\mapsto a$ and the conjugation action $\brhd: B \times A \to A $, $b \brhd a=b a b^\inv$.

\item Any crossed module $(A,B,\brhd,\partial)$ yields a crossed module $(A/\ker \partial, B,\brhd',\partial')$ with injective $\partial'$. 
This identifies $A/\ker\partial$ with a normal subgroup of $B$ and hence yields 1.

\item Any group action $\brhd: B\times A\to A$ by automorphisms of an abelian group $A$  yields a crossed module with $\partial \equiv e_B$. 
\item Any group $A$ defines a crossed module with $B=\mathrm{Aut}(A)$, $\brhd: \mathrm{Aut}(A)\times A \to A$, $\phi \brhd a=\phi (a)$ and $\partial: A \to \mathrm{Aut}(A)$, $g\mapsto C_g$, where $C_g(x)=gxg^\inv$. 
\item Every extension of a group $G$ by a group $X$
\begin{equation*}
\xymatrix{
1 \ar[r] & X  \ar@^{(->}[r]^{\iota} & E  \ar@{->>}[r]^{\pi} &G  \ar[r] & 1
}
\end{equation*}
defines a crossed module with $A:=X$, $B:=E$,  $\partial:=\iota$ and  $b\brhd a:= \iota^\inv (b \iota(a) b^\inv)$ for all $b\in B$, $a\in A$.  The group extension is central iff $\iota(X)\subset Z(E)$, which is equivalent to $\brhd$ trivial. 

\item Conversely, any crossed module $(B, A, \brhd, \partial)$ gives an extension of  $B/\partial(A)$ by $A/\ker (\partial)$:
\begin{equation*}
\xymatrix{
1 \ar[r] & A/\ker (\partial)  \ar@^{(->}[r]^{\qquad \partial'} & B  \ar@{->>}[r]^{\pi\qquad } & B/\partial(A)  \ar[r] & 1
}
\end{equation*}
with $\partial'(a \ker (\partial)):= \partial(a)$. For crossed modules with surjective $\partial$, the group $A$ is also an extension of $B$ by $\ker(\partial)$:
\begin{equation*}
\xymatrix{
1 \ar[r] & \ker (\partial)  \ar@^{(->}[r]^{\iota} & A  \ar@{->>}[r]^{\partial} &B  \ar[r] & 1\text{.}
}
\end{equation*}
\end{compactenum}
\end{example}

\begin{theorem}\label{Proposition:CrossedModulesAndGroupObjectsCatEquivalent}\cite[Th.~1]{BS}\\
The following functors $\bigtriangleup: \mathcal G(\mathrm{Cat})\to\mathcal{CM}$ and $\bigtriangledown: \mathcal{CM}\to\mathcal G(\mathrm{Cat})$ form an equivalence of categories.

The functor $\bigtriangleup$ sends a group object $(H, m, \eta, I)$ to the crossed module $(B, A, \brhd,  \partial)$ with
\begin{compactitem}
\item $A= \mathrm{Cost}_e:=\bigcup_{X\in \mathrm{Ob}(H)} \mathrm{Hom}_H(X, e)$ with  multiplication $m_A: A\times A \to A$, $(a, a')\mapsto m(a,a')$,
\item $B=\mathrm{Ob}(H)$ with  multiplication $m_B: B \times B  \to B$, $(b, b')\mapsto  m(b, b')$,
\item $\partial: A \to B$, $(f: X \to e) \mapsto X$,
\item $\brhd: B  \times A \to A$, $b\brhd a=1_b \cdot a \cdot 1_b^\inv$,
\end{compactitem}
and a morphism of group objects $F:(H, m, \eta, I)\to (H', m', \eta', I')$  to the pair of group homomorphisms $f_1:=F: \mathrm{Ob}(H)\to \mathrm{Ob}(H')$ and $f_2:=F: \mathrm{Cost}_e \to \mathrm{Cost}'_{e'}$.

The functor $\bigtriangledown$ sends a crossed module $(B, A, \brhd, \partial)$ to the group object $(H, m, \eta, I)$ with
\begin{compactitem}
\item $\mathrm{Ob}(H)=B$,
\item $\mathrm{Hom}_{H}(b, b')=\{(a, b)\in A \times B: \, \partial(a)b=b'\}$ with composition $(a', \partial(a)b)\circ (a, b)=(a'a, b)$,
\item $m: H\times H\to H$ with $m(b, b')=bb'$ and $m((a', b'), (a, b))= (a' (b'\brhd a), b'b)$,
\item $\eta: \{\cdot \} \to H$, $\eta(\cdot)=e_B$, $\eta(1_{\cdot})=(e_A, e_B)$,
\item $I:H\to H$, $I(b)=b^\inv$, $I((a, b))=(b^\inv \brhd a^\inv, b^\inv)$,
\end{compactitem}
and a morphism $(f_1, f_2): (B, A, \brhd,\partial )\to (B', A',  \brhd', \partial')$ of crossed modules  to the functor $F: H\to H'$ with $F(b)=f_1(b)$ for all $b\in B=\mathrm{Ob}(H)$ and $F((a, b))=(f_2(a), f_1(b))$.
\end{theorem}

By Theorem \ref{Proposition:CrossedModulesAndGroupObjectsCatEquivalent} the {\em group structure} on the set $H^{(1)}$ of morphisms of a group object $H$ is the semidirect product $A\rtimes B$ 
for the   group action $\brhd: B\times A\to A$. As a \emph{category}  $H$ is the action groupoid for the group action $\brhd': A\times B\to B$, $a\brhd' b=\partial(a)b$.

\subsection{Equalisers and coequalisers in Cat}
\label{subsec:coequcat}

To determine the coinvariants, invariants and the protected object for a group object in Cat, we require equalisers, coequalisers and images in Cat.
It is well-known that Cat is complete and cocomplete, see for instance \cite[Prop.~3.5.6, Cor.~4.5.16]{R}. The following result on  equalisers is standard, see for example Schubert \cite[Sec.~7.2]{Sch}.

\begin{lemma}\label{def:equCat} The equaliser of two functors $F,K:\mathcal C\to \mathcal D$ between small categories  is the subcategory $\mathcal E\subset\mathcal C$ with
\begin{compactitem}
\item $\mathrm{Ob}(\mathcal E)=\{C\in \mathrm{Ob}(\mac)\mid F(C)=K(C)\}$,
\item $\mathrm{Hom}_{\mathcal E}(C,C')=\{f\in \mathrm{Hom}_{\mathcal C}(C,C')\mid F(f)=K(f)\}$.
\end{compactitem}
\end{lemma}

To describe coequalisers in Cat  we use that $\mathrm{Cat}$ is a reflective subcategory of $\mathrm{SSet}$ with the inclusion given by the nerve functor  $N: \mathrm{Cat}\to \mathrm{SSet}$, which is full and faithful. Its left adjoint is the homotopy  functor $h: \mathrm{SSet}\to \mathrm{Cat}$, and  the composite $hN: \mathrm{Cat}\to\mathrm{Cat}$ is naturally isomorphic to the identity functor via the counit of the adjunction, see Riehl
\cite[Ex.~4.5.14 (vi)]{R} or Lurie \cite[Sec.~1.2]{K}. As a left adjoint,  $h$ preserves colimits. This allows one to compute  colimits in Cat  by applying the homotopy functor $h$ to the associated colimits in $\mathrm{SSet}$, see for instance \cite[Prop.~4.5.15]{R}.

\begin{lemma}\label{def:coequCat} The coequaliser of two functors $F,K:\mathcal C\to \mathcal D$ between small categories  is the functor $h(\pi): hN(\mathcal D)\to h(X)$, where $\pi: N(\mathcal D)\to X$ is the coequaliser of $N(F), N(K)$ in $\mathrm{SSet}$.
\end{lemma}

To compute such coequalisers, we require an explicit description of the nerve and the homotopy functor. We summarise the details from \cite[Ex.~4.5.14 (vi)]{R} and \cite[Sec.~1.2]{K}. 
For $n\in\mathbb N_0$ we denote by $[n]$  the  ordinals  in $\Delta$ as well as the associated categories with objects $0,1,\ldots, n$ and a single morphism from $i$ to  $j$ if $i\leq j$. As every weakly monotonic map $\alpha: [m]\to [n]$ defines a  functor $\alpha:[m]\to [n]$,  this defines an embedding $\iota:\Delta\to \Cat$. 

\begin{definition}\label{def:nerve}The {\bf nerve} $N:\mathrm{Cat}\to \mathrm{SSet}$ is the functor that sends a small category $\mac$ to the simplicial set $N(\mac):\Delta^{op}\to \mathrm{Set}$ with
\begin{compactitem}
\item $N(\mac)_n=\Hom_{\mathrm{Cat}}([n], \mac)$,
\item $N(\mac)(\alpha): N(\mac)_n\to N(\mac)_m$, $F\mapsto F\circ \alpha$ for every weakly monotonic  $\alpha:[m]\to [n]$,
\end{compactitem}
and a functor $F:\mac\to\mad$ to the simplicial map $N(F): N(\mac)\to N(\mad)$ that post-composes with $F$.
\end{definition}

By definition,  $N(\mac)_0=\Ob\, \mac$ and $N(\mathcal C)_n$  is the set of sequences 
$(f_1,\ldots, f_n): C_0\xrightarrow{f_1} \ldots\xrightarrow{f_n} C_n$ of  composable morphisms in $\mac$ for $n\in\N$. 
The simplicial set structure is given by the face maps $d_i: N(\mac)_n\to N(\mac)_{n-1}$ 
and  degeneracies $s_i: N(\mac)_n\to N(\mac)_{n+1}$ for $i\in\{0,...,n\}$.
The face maps act on a sequence $(f_1,\ldots, f_n)$
by removing $f_1$ and $f_n$ for $i=0$ and $i=n$, respectively,  and by replacing $(\ldots, f_i, f_{i+1},\ldots)$ with  $(\ldots,f_{i+1}\circ f_i, \ldots)$ for $1\leq i\leq n-1$. For $n=1$ and  $f_1: C_0\to C_1$  one has  $d_0(f_1)=C_1$ and $d_1(f_1)=C_0$. The degeneracies act on  $(f_1,\ldots, f_n)$ by inserting the identity morphism $1_{C_i}$. In particular, for $n=0$ one has $s_0(C)=1_C$ for every $C\in \Ob \, \mac$. The simplicial map $N(F)$ for a functor $F:\mac\to \mad$ applies $F$ to all morphisms in $(f_1,\ldots, f_n)$.

The left adjoint of the nerve $N:\Cat\to\SSet$ is the homotopy functor $h: \SSet\to\Cat$.  It is the left Kan extension along the Yoneda embedding $y: \Delta\to\SSet$
of the embedding functor $\iota: \Delta\to\Cat$. 
Concretely, it is given as follows.

\begin{definition}\label{def:homotopy}The {\bf homotopy functor} $h:\mathrm{SSet}\to \mathrm{Cat}$ sends a simplicial set $X$ to the  category $hX$ with  
$\Ob\, hX=X_0$, generating morphisms $\sigma: d_1(\sigma)\to d_0(\sigma)$ for $\sigma\in X_1$  and relations
\begin{align}\label{eq:relshom}
&s_0(x)=1_x \text{ for } x\in X_0, & &d_1(\sigma)=d_0(\sigma)\circ d_2(\sigma)\text{ for }\sigma\in X_2.
\end{align}
It sends a simplicial map $f: X\to Y$ to the functor $hf: hX\to hY$ given by $f$ on the generators.  
\end{definition}

The simplicial relations imply that  for elements of $X_2$ that are in the image of a degeneracy map, the second relation in \eqref{eq:relshom} is satisfied trivially. In this case one of the two morphisms on the right is an identity and the other coincides with the morphism on the left. Only non-degenerate elements of $X_2$ give rise to non-trivial relations in $hX$. 

In general, morphisms in the homotopy category of a simplicial set $X$ are finite sequences of composable elements of $X_1$. However, 
if the simplicial set $X$ is an $\infty$-category, which is always the case if $X=N(\mac)$ for some category $\mac$, every morphism in $hX$ is represented by a single element in $X_1$, see for instance  \cite[Sec.~1.2.5]{K}.  Most of  the simplicial sets considered in the following are even Kan complexes, as they are nerves of groupoids.

As  a right  adjoint, the nerve preserves limits, and as a left adjoint, the homotopy functor preserves colimits. It follows directly from its definition that the nerve also preserves coproducts, and the homotopy functor preserves finite products, see for instance Joyal \cite[Prop. 1.3]{Jo}. 
This  implies  with Examples \ref{ex:transport0} and \ref{Ex:groupobjectfunctor}

\begin{corollary} \label{cor:nervegroup}The nerve $N:\Cat\to \SSet$ and the homotopy functor $h: \SSet\to \Cat$ are symmetric monoidal with respect to the cartesian monoidal category structures of $\Cat$ and $\SSet$. In particular:
\begin{compactenum}
\item The nerve of a crossed module is a simplicial group.
\item The homotopy category of a simplicial group is a crossed module.
\item The nerve of a (co)module over a crossed module  is a (co)module over its nerve.
\item The homotopy category of a (co)module over a simplicial group  is a (co)module over its homotopy category.
\end{compactenum}
\end{corollary}

Concretely,   the nerve of  a crossed module $(B,A, \brhd,\partial)$ is the simplicial group $H$ with $H_n= A^{\times n}\times B$ for $n\in\N_0$ with group multiplication
\begin{align}\label{eq:groumultsimp}
(a_1,..., a_n, b)\cdot ( a'_1,..., a'_n, b')=(a_1(b\brhd a'_1), a_2(\partial(a_1)b\brhd a'_2), ... ,a_n(\partial(a_{n-1}\cdots a_1)b\brhd a'_n), bb')
\end{align}
and  face maps  and degeneracies \vspace{-.4cm}
\begin{align}\label{eq:facedegs}
&d_i: H_n\to H_{n-1}, \quad (a_1,..., a_n,b)\mapsto \begin{cases} ( a_2,..., a_n, \partial(a_1)b) & i=0\\
(a_1,..., a_{i+1}a_i, ..., a_n,b) & 1\leq i\leq n-1\\
(a_1,..., a_{n-1},b) & i=n
\end{cases}\\
&s_i: H_n\to H_{n+1}, \quad (a_1,..., a_n,b)\mapsto (a_1,\ldots, a_i, 1, a_{i+1}, ..., a_n,b) \quad 0\leq i\leq n .\nonumber
\end{align}

\smallskip
\begin{example} \label{ex:crossedmod} $\quad$
\begin{compactenum}
\item A group action $\brhd: B\times A\to A$  by automorphisms on an abelian group $A$  yields a simplicial group with $H_n=A^{\times n}\rtimes' B$, where $B$ acts diagonally via $\brhd$, and  with the face maps and degeneracies \eqref{eq:facedegs} for $\partial\equiv 1$.

\item Every injective group homomorphism $\partial: A\to B$  from an abelian group $A$ into the centre of a group $B$ yields a simplicial group, where $H_n=A^{\times n}\times B$  is the direct product, and the face maps and degeneracies are given by \eqref{eq:facedegs}.
\item Every abelian group $A$ is a simplicial group with $H_n=A^{\times n}$, the group multiplication of $A^{\times n}$, the face maps and degeneracies \eqref{eq:facedegs}
for $B=\{e\}$ and $\partial\equiv 1$.

\item Any normal subgroup $A\subset B$ determines a simplicial group with $H_n=A^{\times n}\times B$ and group multiplication \eqref{eq:groumultsimp}, face maps and degeneracies \eqref{eq:facedegs}, where $\partial: A\to B$ is the inclusion and $\brhd: B\times A\to A$ the conjugation action. 
\end{compactenum}

\end{example}

\subsection{(Co)invariants of (co)modules  over group objects in Cat}
\label{subsec:modcomodgrp}

The coinvariants of a comodule $(\mathcal M, \delta)$ over a group object  $(H, m, \eta, I)$ in Cat are given as the equaliser of $\delta=(F \times 1_{\mathcal M})\circ \Delta:\mathcal M\to H\times \mathcal M$ and $\eta\times 1_{\mathcal M}: \mathcal M\to H\times \mathcal M$. This is  the subcategory on which $\delta$ and $\eta\times 1_{\mathcal M}$ coincide, together with its inclusion functor, see Lemma \ref{def:equCat}. In terms of the associated  functor $F:\mathcal M\to H$ from 
Example \ref{Ex:YDAndHopfModulesInCat}
we have

\begin{lemma}\label{Lemma:CoinvariantsInCat}$ $\\
Let $(\mathcal M, \delta)$ be a comodule over a group object $(H, m, \eta, I)$ in $\mathrm{Cat}$. Then the coinvariants  are given by the  subcategory $\mathcal M^{coH}\subset\mathcal M$ with
\begin{compactitem}
\item $\mathrm{Ob}(\mathcal M^{coH})=\{A\in \mathrm{Ob}(\mathcal M)\mid F(A)=e\}$,
\item $\mathrm{Hom}_{\mathcal M^{coH}}(A,A')=\{f\in \mathrm{Hom}_{\mathcal M}(A, A')\mid F(f)=1_e\}$,
\end{compactitem}
and the  inclusion functor $\iota: \mathcal M^{coH}\to \mathcal M$.
\end{lemma}

The invariants of a module $(\mathcal M, \rhd)$ over a group object  $H$ in $\mathrm{Cat}$ are the coequaliser of the functors $\rhd,\pi_2: H \times \mathcal M \to \mathcal M$. They are computed with Lemma \ref{def:coequCat}.

\begin{proposition} \label{Prop:InvariantsInCat}Let $(\mathcal M, \rhd)$ be a module over a group object $H=\bigtriangledown(B,A,\brhd, \delta)$  in $\Cat$. 
Then its invariants are the category $\mathcal M^H$,
whose
\begin{compactitem}
\item objects are  orbits of the $B$-action on $\mathrm{Ob}(\mathcal M)$,
\item morphisms  are generated by orbits of the $A\rtimes B$-action on $\mathcal M^{(1)}$ subject to the relations
$
[f_2]\circ [f_1]=[f_2\circ f_1]
$ 
for all $A\rtimes B$-orbits $[f_1]$, $[f_2]$ of composable morphisms $f_1,f_2$ in $\mathcal M$.
\end{compactitem}
We denote by $\pi: \mathcal M\to\mathcal M^H$ the projection functor that sends each object  of $\mathcal M$ to its $B$-orbit and each morphism in $\mathcal M$ to the equivalence class of its $A\rtimes B$-orbit.
\end{proposition}

\begin{proof} By Corollary \ref{cor:nervegroup}, applying the nerve to the group object $H$ in $\Cat$ and to a module $(\mathcal M,\rhd)$ over $H$ yields a simplicial group $N(H)$ and a module $N(\mathcal M)$ over $N(H)$ in $\SSet$. By Lemma \ref{def:coequCat} the coequaliser of the morphisms $\rhd,\pi_2: H\times \mathcal M\to\mathcal M$ in $\Cat$ is obtained by applying the homotopy functor  to the   coequaliser of  $\rhd'=N(\rhd),\pi'_2=N(\pi_2): N(H)\times N(\mathcal M)\to N(\mathcal M)$ in $\SSet$. 

As  colimits in $\SSet$ are computed objectwise, see for instance \cite[Prop.~3.3.9]{R}, the coequaliser  of $\rhd', \pi'_2$ is 
   the simplicial set $N(\mathcal M)^H$ with
$N(\mathcal M)^H_n=N(\mathcal M)_n/\sim_n$, where $\sim_n$ is the equivalence relation on $N(\mathcal M)_n$ defined by the $N(H)$-action:  
 $m\sim_n m'$ iff there is a  $g\in N(H)_n$ with $m'=g\rhd' m$. 
The face maps and degeneracies of $N(\mathcal M)^H$ are induced by the ones of $N(\mathcal M)$. 

As $N(H)_0=\Ob\, H=B$ and $N(\mathcal M)_0=\Ob\, \mathcal M$, the elements of $N(\mathcal M)^H_0$ are the orbits of the $B$-action on $\Ob\,\mathcal M$. 
As $N(H)_1=H^{(1)}=A\rtimes B$, the set $N(\mathcal M)^H_1$ contains the orbits of the $A\rtimes B$-action on $\mathcal M^{(1)}$. 
Elements of $N(\mathcal M)_2$ and $N(H)_2$ are pairs  of composable morphisms in $\mathcal M$ and $H$. 
Thus, the set $N(\mathcal M)^H_2$ consists of equivalence classes  of pairs  $(f_1,f_2)$ of composable morphisms  in $\mathcal M$ with $(f_1,f_2)\sim (f'_1,f'_2)$ if there are $(a_1,b_1), (a_2,b_2)\in A\rtimes B$ with $\partial(a_1)b_1=b_2$  such that $f'_1=(a_1,b_1)\rhd' f_1$ and $f'_2=(a_2,b_2)\rhd' f_2$. 

For any composable pair $(f_1,f_2) \in N(\mathcal M)_2$, one has $d_0(f_1,f_2)=f_2$, $d_1(f_1,f_2)=f_2\circ f_1$ and $d_2(f_1,f_2)=f_1$. This  implies
$d_0 [(f_1,f_2)]=[f_2]$, $d_1[(f_1,f_2)]=[f_2\circ f_1]$ and $d_2[(f_1,f_2)]=[f_1]$ for their equivalence classes in $N(\mathcal M)^H_2$ and $N(\mathcal M)^H_1$. 

Applying the homotopy functor from Definition \ref{def:homotopy}  thus yields a category $\mathcal M^H$ with objects $\Ob\,\mathcal M^H=N(\mathcal M)^H_0=\Ob\,\mathcal M/B$. Its generating morphisms are $A\rtimes B$-orbits of morphisms in $\mathcal M$, and the second relation in \eqref{eq:relshom} translates into the relation $[f_2]\circ [f_1]=[f_2\circ f_1]$ for the $A\rtimes B$-orbits of  composable pairs $(f_1,f_2)$ of morphisms in $\mathcal M$. 
\end{proof}

We now restrict attention to Yetter-Drinfeld modules $(\mathcal M, \rhd, \delta)$ over group objects $H$ in Cat and determine their biinvariants. We denote again by $F:\mathcal M\to H$ the functor defined by $\delta$ from Example \ref{Ex:YDAndHopfModulesInCat}, by $\iota: \mathcal M^{coH}\to\mathcal M$ the inclusion functor from Lemma \ref{Lemma:CoinvariantsInCat} and by  $\pi: \mathcal M\to \mathcal M^H$  the projection functor from  Proposition \ref{Prop:InvariantsInCat}.

\begin{proposition}\label{Proposition:ImageInCat}$ $\\
Let $(\mathcal M, \rhd, F)$ be a Yetter-Drinfeld module over a group object $H$ in $\mathrm{Cat}$. Then $\mathcal M_{inv}$ is given by
\begin{align*}
&\Ob \mathcal M_{inv}=\{\pi(M)\mid M\in \Ob\mathcal M \text{ with } F(M)=e\},\\
&\Hom_{\mathcal M_{inv}}(\pi(M_1), \pi(M_2))=\{\pi(f)\mid f\in \mathcal M^{(1)} \text{ with } \pi(s(f))=\pi(M_1), \pi(t(f))=\pi(M_2), F(f)=1_e\}. 
\end{align*}
\end{proposition}

\begin{proof} 1.~We verify that $\mathcal M_{inv}$ is  a category. If $F(M)=e$ for an object $M$ in $\mathcal M$, then $F(g\rhd M)=g\cdot F(M)\cdot g^\inv=e$ for all objects $g$ in $H$ by  the Yetter-Drinfeld module condition in Example \ref{Ex:YDAndHopfModulesInCat}. Likewise, if  $f$ is a morphism in $\mathcal M$ with $F(f)=1_e$,  then $F(g\rhd f)=g \cdot F(f)\cdot g^\inv=1_e$ for all $g\in H^{(1)}$. This shows that for every object $M$ and morphism $f$ of $\mathcal M^{coH}$ the entire $\Ob\, H$-orbit of $M$ and $H^{(1)}$-orbit of $f$  is contained in $\mathcal M^{coH}$.  Any identity morphism  on an object $M\in \Ob\,\mathcal M^{coH}$ satisfies $F(1_M)=1_e$ and hence is contained in $\mathcal M^{coH}$. 
 If $(f_1,f_2)$ is a pair of composable morphisms in $\mathcal M^{coH}$, then $F(f_2\circ f_1)=F(f_2)\circ F(f_1)=1_e$ and hence $f_2\circ f_1\in \mathcal M^{coH}$ as well. 

 Suppose now that  $f_1: M_0\to M_1$ and  $f_2: M'_1\to M_2$ are morphisms in $\mathcal M^{coH}$ such that $\pi(f_1)$ and $\pi(f_2)$ are composable in $\mathcal M^H$. Then there is a $g\in \Ob\, H$ with $M'_1=g\rhd M_1$,  and the morphisms  $f_1$ and $g^\inv\rhd f_2$ are composable in $\mathcal M^{coH}$. With the relations of $\mathcal M_{inv}$ one obtains  $\pi(f_2)\circ \pi(f_1)=\pi(g^\inv\rhd f_2)\circ \pi(f_1)=\pi((g^\inv\rhd f_2)\circ f_1)$ with  $(g^\inv \rhd f_2)\circ f_1\in \mathcal M^{coH}$. 

2.~We show that $\mathcal M_{inv}$ has the universal property of the image in Cat. 
The inclusion functor 
 $I: \mathcal M_{inv}\to \mathcal M^H$  is a monomorphism in $\Cat$ and satisfies $IP=\pi\iota$, where
 $P: \mathcal M^{coH}\to \mathcal M_{inv}$ is the functor that sends an object $M$ in $\mathcal M^{coH}$ to $\pi(M)$ and a morphism $f$ in $\mathcal M^{coH}$ to $\pi(f)$.
  If $(J,Q)$ is a pair of a monomorphism $J: \mac\to \mathcal M^{H}$ and a functor $Q: \mathcal M^{coH}\to \mac$ with $JQ=\pi\iota$, 
 then  $J$ is a monomorphism in $\Cat$,  which allows one to identify $\mac$ with a subcategory of $\mathcal M^H$ and $J$ with its inclusion functor. As $JQ=\pi\iota$, the subcategory $\mac\subset \mathcal M^H$ contains $\mathcal M_{inv}$ as a subcategory $\mathcal M_{inv}\subset \mac$ and hence there is a unique functor, the inclusion $V:\mathcal M_{inv}\to\mac$,  with
  $I=JV$. 
\end{proof}

\begin{remark}\label{remark:CoequalisersViaGenerealisedCongruences}
Coequalisers in $\mathrm{Cat}$ can also be determined via the construction of Bednarczyk, Borzyszkowski and Pawlowski \cite{BBP}, using \textit{generalised congruences} and associated \textit{quotient categories}. For a summary of this construction, see also Bruckner \cite{Br} and Haucourt \cite{Ha}.
 For a module $(\mathcal M, \rhd)$ over a group object in $\mathrm{Cat}$ the associated quotient category gives the invariants of $(\mathcal M, \rhd)$ as in Proposition \ref{Prop:InvariantsInCat}. For a triple $(\mathcal M, \rhd, \delta)$ the generalised congruence $(\sim_0, \sim_m)$ restricts to a generalised congruence on $\mathcal M^{coH}$ whose quotient category are the biinvariants of $(\mathcal M, \rhd, \delta)$.
\end{remark}

\subsection{Protected objects for group objects in Cat}
\label{subsec:protected}

We now give a concrete description of the coinvariants and the protected objects for oriented surfaces $\Sigma$ of genus $g\geq 1$ and group objects $H=\bigtriangledown (B,A,\brhd,\partial)$ in $\Cat$. We start by considering the Yetter-Drinfeld module and the  coinvariants  for the standard graph from \eqref{eq:standardgraph} 
and show that they are given by group homomorphisms $\rho: F_{2g}\to A\rtimes B$ and $\rho: \pi_1(\Sigma)\to A\rtimes B$, respectively. To describe their category structure, we consider group-valued  1-cocycles.

\begin{definition}\label{Def:Cocycle}$ $ \\
Let $K, A$ be groups and $\brhd: K\times A\to A$ a group action of $K$ on $A$ by automorphisms. 
\begin{compactenum}[1.]
\item A \textbf{$1$-cocycle}   is a map $\phi: K\to A$ with 
$
\phi(\lambda \mu)=\phi(\lambda) \cdot (\lambda \brhd \phi(\mu)) 
$ for all $\lambda, \mu \in K$.
\item A \textbf{$1$-coboundary}   is a map $\eta_a: K\to A$, $\lambda \mapsto a(\lambda\brhd a^\inv)$ for some $a\in A$. 
\item  $\phi,\psi:K\to A$ are \textbf{related by a coboundary} if  $\psi(\lambda)=a \cdot \phi(\lambda)\cdot (\lambda\brhd a^\inv)$ for some $a\in A$.
\end{compactenum}
\end{definition}
If $A$ is abelian,   1-cocycles form a group $Z^1(K, A,\brhd)$ with  pointwise multiplication  and  coboundaries  a subgroup $B^1(K, A,\brhd)$. The factor group is the first cohomology group $H^1(K, A,\brhd)$. More generally,  1-cocycles with values in a (not necessarily abelian) group $A$ arise from group homomorphisms into a semidirect product $A\rtimes B$.

\begin{lemma}\label{Lemma:GrouphomoIsCocycleAndGrouphomoToB}$ $\\ 
Let  $\brhd: B\times A\to A$ a group action  by automorphisms and $A\rtimes B$ the associated semidirect product. 
\begin{compactenum}
\item Group homomorphisms $\sigma: K\to A\rtimes B$ correspond to  pairs $(\phi,\rho)$ of a group homomorphism $\rho:K\to B$ and a $1$-cocycle $\phi:K\to A$ for the action $\rho^*\brhd: K \times A\to A$, $(\lambda, a)\mapsto \rho(\lambda)\brhd a$.
\item Two 1-cocycles $\phi,\phi':K\to A$ for $\rho^*\brhd$ are related by a coboundary iff the group homomorphisms $(\phi, \rho), (\phi',\rho):K\to A\rtimes B$ are related by conjugation with $A\subset A\rtimes B$.
\end{compactenum}
\end{lemma}

If the semidirect product in Lemma \ref{Lemma:GrouphomoIsCocycleAndGrouphomoToB} arises from a crossed module
 $(B,A,\brhd,\partial)$, the group homomorphism $\partial: A\to B$,  allows one to organise the group homomorphisms $\rho: K\to B$ and  1-cocycles $\phi:K\to A$  from  Lemma  \ref{Lemma:GrouphomoIsCocycleAndGrouphomoToB}  into a groupoid.  
Denoting by $\phi\cdot\psi$ and $\phi^\inv$ the pointwise product and inverse of maps $\phi,\psi:K\to A$ we have

\begin{lemma}\label{Lemma:CategoryFromCocycles} $ $ \\
Any group $K$ and crossed module $(B,A,\brhd, \partial)$ defines a  groupoid 
$\Hom(K,B\brhd A)$  with
\begin{compactitem}
\item group homomorphisms $\rho:K\to B$ as objects,
\item $\mathrm{Hom}(\rho, \rho')=\{(\phi,\rho)\mid \phi:K\to A\text{ 1-cocycle for } \rho^*\brhd\text{ with } (\partial\circ \phi)\cdot \rho=\rho' \}$,
\item composition of morphisms: $(\psi, (\partial\circ\phi)\cdot \rho)\circ (\phi,\rho)=(\psi\cdot \phi,\rho) $,
\item inverse morphisms: $(\phi,\rho)^\inv=(\phi^\inv, (\partial\circ \phi)\cdot \rho)$.
\end{compactitem}
\end{lemma}
\begin{proof}
A direct computation using \eqref{eq:crossedmod} shows that for any pair $(\phi,\rho)$ of a group homomorphism $\rho:K\to B$ and a 1-cocycle $\phi:K\to A$ for $\rho^*\brhd$, the map $(\partial\circ \phi)\cdot \rho:K\to B$ is another group homomorphism. Similarly, if $\phi$ is a 1-cocycle for $\rho^*\brhd$ and $\psi$ a 1-cocycle for $((\partial\circ \phi)\cdot \rho)^*\brhd$, then $\psi\cdot\phi$ is another 1-cocycle  for $\rho^*\brhd$. The formula for the inverse morphism  follows directly.
\end{proof}

By applying this lemma to  Example \ref{ex:modcomodpi_1gen}, we obtain a groupoid that describes the Yetter-Drinfeld module for a group object $H=\bigtriangledown (B,A,\brhd,\partial)$  in $\Cat$  and the standard graph  \eqref{eq:standardgraph}, if we set $K=F_{2g}$ and identify the generators of $F_{2g}$ with the edges of the graph. An analogous result holds for the associated coinvariants for $K=\pi_1(\Sigma)$ and any properly embedded graph with a single vertex.

\begin{proposition}\label{Proposition:CoinvariantsAndCocycleCategoryEquivalent}$ $\\
Let $\Gamma$ be a properly embedded graph  with a single vertex on a surface $\Sigma$ of genus $g\geq 1$ and $H=\bigtriangledown (B, A, \brhd, \partial)$ a group object in $\mathrm{Cat}$. 
Then the associated coinvariants are the groupoid  from Lemma \ref{Lemma:CategoryFromCocycles} for $K=\pi_1(\Sigma)$. 
\end{proposition}
\begin{proof}
By Theorem \ref{th:topinv} it suffices to consider the graph  in \eqref{eq:standardgraph}. 

By Example \ref{ex:modcomodpi_1gen} the coinvariants are  the equaliser of the morphisms
$ \eta \, \epsilon, F: H^{\times 2g}\to H$ in Cat, where $\epsilon: H^{\times 2g}\to \{\cdot\}$ is the terminal morphism, $\eta:\{\cdot\}\to H$ is as in Definition \ref{Def:GroupObjectInCat} and $F: H^{\times 2g}\to H$ is given by $F(a_1,b_1,\ldots, a_g,b_g)=[b_g^\inv, a_g]\cdots [b_1^\inv,a_1]$. 
By Lemma \ref{def:equCat}, this equaliser is the subcategory $\mathcal E\subset H^{\times 2g}$ consisting of objects $C$ and morphisms $f$ with $F(C)=e$ and $F(f)=1_e$. 
For $H=\bigtriangledown(B,A,\brhd,\partial)$, this yields  with Theorem \ref{Proposition:CrossedModulesAndGroupObjectsCatEquivalent}
\begin{align*}
\mathrm{Ob}(\mathcal E)&=\{(a_1,b_1, \dots, a_g, b_g)\in B^{\times 2g}\mid [b_g^\inv, a_g]\cdots[b_1^\inv,a_1]=1\}\\
\mathcal E^{(1)}&=\{(a_1, b_1,\ldots, a_g,b_g)\in (A\rtimes B)^{2g}\mid [b_g^\inv, a_g]\cdots[b_1^\inv,a_1]=1\}.
\end{align*}
Thus, every object $\rho \in \mathrm{Ob}(\mathcal E)$ corresponds to a group homomorphism $\rho:\pi_1(\Sigma)\to B$ and every morphism $\sigma\in \mathcal E^{(1)}$ to a group homomorphism $\sigma: \pi_1(\Sigma)\to A\rtimes B$.
By Lemma \ref{Lemma:GrouphomoIsCocycleAndGrouphomoToB} the latter defines a pair $\sigma=(\phi,\rho)$ of a group homomorphism $\rho:\pi_1(\Sigma)\to B$ and a 1-cocycle 
 $\phi$ for $\rho^*\brhd$. 
\end{proof}

We  now use the  description of the coinvariants in Proposition \ref{Proposition:CoinvariantsAndCocycleCategoryEquivalent} and the description of the image object in Cat from Proposition \ref{Proposition:ImageInCat} to compute the protected object for a surface $\Sigma$ of genus $g\geq 1$ and a crossed module $(B,A,\brhd,\partial)$.

\begin{theorem} \label{th:kitaevexplicit}The protected object for a group object $H=\bigtriangledown (B,A,\brhd,\partial)$ in $\Cat$ and a surface $\Sigma$ of genus $g\geq 1$   is a groupoid  $\mathcal M_{H,\Sigma}$ with
\begin{compactitem}
\item conjugacy classes of group homomorphisms $\rho:\pi_1(\Sigma)\to B$ as objects,
\item equivalence classes of  group homomorphisms $\tau=(\phi,\rho): \pi_1(\Sigma)\to A\rtimes B$  as morphisms  
 from $[\rho]$ to $[(\partial\circ \phi)\cdot \rho]$. \end{compactitem}
 The equivalence relation  is given by
 $\tau_2\circ \tau_1\sim \tau'_2\circ \tau'_1$ for all composable pairs  $(\tau_1,\tau_2)$ and $(\tau'_1,\tau'_2)$ of group homomorphisms $\tau_i,\tau'_i: F_{2g}\to A\rtimes B$ such that $\tau_i, \tau'_i$ are conjugate and  $\tau_2\circ \tau_1, \tau'_2\circ \tau'_1$ define group homomorphisms $\pi_1(\Sigma)\to A\rtimes B$.

\end{theorem}

\begin{proof}  By Theorem \ref{th:topinv} the protected object of $\Sigma$ is a topological invariant and can be computed from the standard  graph in \eqref{eq:standardgraph}. This  yields a Yetter-Drinfeld module $(\mathcal M,\rhd,\delta)$ over $\bigtriangledown(B,A,\brhd,\partial)$ given by  formula \eqref{eq:YDmodulegroup}. Hence, we have  $\mathcal M^{(1)}=(A\rtimes B)^{2g}\cong \Hom(F_{2g}, A\rtimes B)$ with the module structure given by conjugation  and the comodule structure by the defining relation of $\pi_1(\Sigma)$.
  
By Proposition  \ref{Proposition:CoinvariantsAndCocycleCategoryEquivalent} the associated coinvariants form a groupoid $\mathcal M^{coH}$ with group homomorphisms $\rho:\pi_1(\Sigma)\to B$ as objects and group homomorphisms $\tau=(\phi,\rho):\pi_1(\Sigma)\to A\rtimes B$  as morphisms from $\rho$ to $(\partial\circ\phi)\cdot \rho$.  By Propositions  \ref{Prop:InvariantsInCat}  and  \ref{Proposition:ImageInCat} the associated  image object is the groupoid, whose objects 
 are orbits  of group homomorphisms $\rho:\pi_1(\Sigma)\to B$ under the conjugation action of  $B$ and whose morphisms are the images of
 group homomorphisms $\tau:\pi_1(\Sigma)\to A\rtimes B$ under the projection functor $\pi: \mathcal M\to \mathcal M^H$. 
 The latter is given by  the equivalence relation in the theorem. 
\end{proof}

There are a number of cases in which the protected object has a particularly simple form. They correspond to crossed modules in which part of the data is trivial. The first  corresponds to the case, where the Moore complex of the crossed module has trivial non-abelian homologies, namely $\ker(\partial)=\{1\}$ and $B/\partial(A)=1$. The second is the case where the action of $B$ on $A$ is trivial.

\begin{example}\label{ex:deltaiso} Let $\Sigma$ be a surface of genus $g\geq 1$ and  $(B,A, \brhd, \partial)$ a crossed module, where $\partial$ is an isomorphism. Then the  protected object has
\begin{compactitem}
\item conjugacy classes of group homomorphisms $\rho:\pi_1(\Sigma)\to B$ as objects,
\item exactly one morphism between any two objects.
\end{compactitem}
\end{example}

\begin{proof}
All morphism sets in the groupoid from Lemma \ref{Lemma:CategoryFromCocycles} contain exactly one morphism, since $\Hom(\rho,\sigma)=\{(\partial^\inv(\sigma\cdot\rho^\inv), \rho)\}$ for all group homomorphisms $\rho,\sigma: \pi_1(\Sigma)\to B$.  Conjugating a morphism in $ \Hom(\rho,\sigma)$
 with an element of $(a,b)\in A\rtimes B$ yields the unique  morphism from $b\rho b^\inv$ to $(\partial (a)b)\, \sigma \, (\partial(a)b)^\inv$. This shows that all morphisms from conjugates of a group homomorphism  $\rho:\pi_1(\Sigma)\to B$ to a conjugate of a group homomorphism  $\sigma:\pi_1(\Sigma)\to B$ are conjugated  and hence identified  in $\mathcal M^H$ and in $\mathcal M_{inv}$.
 \end{proof}

\begin{example}\label{cor:catsimple} Let $\Sigma$ be a surface of genus $g\geq 1$ and  $(B,A,\brhd,\partial)$ a crossed module with a trivial group action $\brhd$. Then
  the   protected object  is  $\mathrm{Hom}(\pi_1(\Sigma), A\times B)/A\times B$
with
\begin{compactitem}
\item conjugacy classes of group homomorphisms $\rho:\pi_1(\Sigma)\to B$ as objects,
\item group homomorphisms $\phi:\pi_1(\Sigma)\to A$ as morphisms from $[\rho]$ to $[(\partial\circ \phi)\cdot \rho]$.
\end{compactitem}
\end{example}

\begin{proof} If $\brhd: B\times A\to A$ is  trivial, then conditions \eqref{eq:crossedmod} imply that $A$ is abelian with $\partial(A)\subset Z(B)$.
As $A$ is abelian and $\brhd$ trivial, the 1-cocycles from Definition \ref{Def:Cocycle} are  simply  group homomorphisms $\phi: \pi_1(\Sigma)\to A$ and any 1-coboundary is trivial. 
The groupoid $\mathcal M^{coH}$ from  Lemma \ref{Lemma:CategoryFromCocycles} thus has as objects group homomorphisms $\rho:\pi_1(\Sigma)\to B$ and as morphisms  $\tau=(\phi,\rho): \rho\to (\partial\circ \phi)\cdot \rho$ group homomorphisms $\tau=(\phi,\rho): \pi_1(\Sigma)\to A\times B$.

 As $A$ is abelian and $\brhd$ trivial, 
two group homomorphisms $\tau=(\phi,\rho), \tau'=(\phi',\rho'): \pi_1(\Sigma)\to A\rtimes B$  are conjugate iff $\phi'=\phi$ and $\rho'=b\rho b^\inv$ for some $b\in B$. Thus, the relation on morphisms in Theorem \ref{th:kitaevexplicit} identifies $\tau$ and $\tau'$ iff $\phi=\phi'$ and $[\rho]=[\rho']$.
\end{proof}

In the case of a trivial  group homomorphism $\partial: A\to B$  all morphisms in $\mathcal M$, $\mathcal M^{coH}$, $\mathcal M^H$ and $\mathcal M_{inv}$ are automorphisms. This yields

\begin{example} Let $\Sigma$ be a surface of genus $g\geq 1$ and  $H=\bigtriangledown(B,A,\brhd,\partial)$  with $A$ abelian and a trivial group homomorphism $\partial\equiv 1$. Then
  the associated  protected object  is 
  $$
  \mathcal M_{H,\Sigma}=\amalg_{[\rho]\in \Hom(\pi_1(\Sigma), B)/B}\; G_{[\rho]},
  $$
  where $G_{[\rho]}$ is a  factor group of  $H^1(\pi_1(\Sigma), A,\rho^*\brhd)$.
\end{example}

\begin{proof}
If $\partial$ is trivial and $A$ abelian, then every 1-cocycle $\phi:\pi_1(\Sigma)\to A$ for $\rho^*\brhd$ defines an automorphism of $\rho$ in $\mathcal M^{coH}$, which implies
$\mathcal M^{coH}=\amalg_{\rho\in \Hom(\pi_1(\Sigma), B)} Z^1(\pi_1(\Sigma), A, \rho^*\brhd)$. 

As all morphisms in $\mathcal M^{coH}$ are automorphisms, two morphisms given by group homomorphisms $\tau=(\phi,\rho):\pi_1(\Sigma)\to A\rtimes B$ and $\tau'=(\phi',\rho'):\pi_1(\Sigma)\to A\rtimes B$  are composable iff $\rho=\rho'$.
By Lemma \ref{Lemma:GrouphomoIsCocycleAndGrouphomoToB}, 2.~two group homomorphisms $(\phi,\rho):\pi_1(\Sigma)\to A\rtimes B$ and $(\phi',\rho):\pi_1(\Sigma)\to A\rtimes B$ are related by conjugation with $A\subset A\rtimes B$ iff $\phi,\phi'$ are related by a 1-coboundary. Thus for a group homomorphism $\rho:\pi_1(\Sigma)\to B$, the  automorphism group of  $[\rho]$ in $\mathcal M_{inv}$ is a factor group of $H^1(\pi_1(\Sigma), A, \rho^*\brhd)$.
\end{proof}

By  Theorem \ref{th:kitaevexplicit}   group homomorphisms $\tau,\tau': \pi_1(\Sigma)\to A\rtimes B$ that are conjugated define the same morphism in $\mathcal M_{inv}$. 
This implies in particular  that the morphism in $\mathcal M_{H,\Sigma}$  defined by a group homomorphism $\sigma=(\phi,\rho):\pi_1(\Sigma)\to A\rtimes B$ depends on $\phi$ only up to coboundaries. Modifying $\phi$ with a coboundary yields a group homomorphism $\sigma'=(\phi',\rho)$ conjugated to $\sigma$  by Lemma \ref{Lemma:GrouphomoIsCocycleAndGrouphomoToB}, 2. 

However, except for the  situation in Examples \ref{ex:deltaiso} and  \ref{cor:catsimple}, it is  difficult to describe the category $\mathcal M_{H,\Sigma}$ explicitly, even for genus $g=1$ and crossed modules given by normal subgroups. This is due to the fact that the equivalence relation in Theorem \ref{th:kitaevexplicit} also identifies morphisms in $\mathcal M^{coH}$ in different $A\rtimes B$-orbits.  This is illustrated by the following two examples.

\begin{example} \label{ex:minv}Let $\Sigma$ be the torus with  $\pi_1(\Sigma)=\Z\times\Z$ and consider the crossed module
 $(S_3,A_3,\brhd, \iota)$, where $\iota: A_3\to S_3$ is the inclusion and $\brhd: S_3\times A_3\to A_3$, $b\brhd a=bab^\inv$ the conjugation action. 
 
We specify group homomorphisms $\rho:\Z\times\Z\to S_3$ and 1-cocycles $\phi: \Z\times \Z\to A_3$  by the images of $(1,0)$ and $(0,1)$ and write $\rho=(\rho(1,0),\rho(0,1))$  
for the former and $\phi=\langle \phi(1,0), \phi(0,1)\rangle$ for the latter.  Then the conjugacy classes  of group homomorphisms $\rho:\Z\times\Z\to S_3$ are given by
 
\begin{center}
 \begin{tabular}{|l|l|}
 \hline
$ C_1=\{(\id,\id)\}$ &\\
\hline
 $C_2=\{(\id, c)\mid c\in A_3\setminus\{\id\}\}$ & $C'_2=\{(c,\id)\mid  c\in A_3\setminus\{\id\} \}$\\
 \hline
 $C_3=\{(c,c)\mid c\in A_3\setminus\{\id\}\}$ & \\
 \hline
 $C_4=\{(c,c')\mid c\neq c'\in A_3\setminus\{\id\}\}$ &\\
 \hline
 $C_5=\{(\id,\sigma)\mid \sigma\in S_3\setminus A_3\}$ &  $C'_5=\{(\sigma,\id)\mid \sigma\in S_3\setminus A_3\}$\\
 \hline
 $C_6=\{(\sigma,\sigma)\mid \sigma\in S_3\setminus A_3\}$ &\\
 \hline
 \end{tabular}
 \end{center}

If $\rho(\Z\times\Z)\subset A_3$, then 1-cocycles for $\rho^*\brhd$ are simply group homomorphisms $\phi: \Z\times\Z\to A_3$. If $\rho\in C_5$, $\rho\in C'_5$ or $\rho\in C_6$ then the   1-cocycles  for $\rho^*\brhd$ are given by
\begin{align*}
\langle \id, c \rangle(k,l)=\begin{cases} c & l\text{ odd}\\  \id & l\text{ even}\end{cases}, \; \langle c,\id\rangle (k,l)=\begin{cases} c & k\text{ odd}\\  \id & k\text{ even}\end{cases},\; \langle c,c\rangle (k,l)=\begin{cases} c & k+l\text{ odd}\\  \id & k+l\text{ even}\end{cases},
\end{align*}
respectively, where $c\in A_3$, $k, l\in \Z$.
  
The protected object $\mathcal M_{inv}$ has the objects $C_1$, $C_2$, $C'_2$, $C_3$, $C_4$, $C_5$, $C'_5$, $C_6$. Its morphisms are  equivalence classes of morphisms in $\mathcal M^{coH}$, that is, of 1-cocycles.
 The morphisms in $\mathcal M^{coH}$ starting in $\rho\in C_1$ are  the trivial 1-cocycle $\phi \equiv \id$ as identity morphism and  the   conjugate pairs
 \begin{align*}
&\langle \id, (123)\rangle: (\id,\id)\to (\id, (123)) & &\sim & &\langle \id, (132)\rangle: (\id,\id)\to (\id, (132))\\
&\langle (123), \id\rangle: (\id,\id)\to ((123),\id) & &\sim & &\langle(132), \id\rangle: (\id,\id)\to ((132), \id )\\
&\langle (123), (123)\rangle: (\id,\id)\to ((123),(123)) & &\sim & &\langle(132), (132)\rangle: (\id,\id)\to ((132), (132) )\\
&\langle(123), (132)\rangle: (\id,\id)\to ((123),(132)) & &\sim & &\langle(132), (123)\rangle: (\id,\id)\to ((132), (123) ),
 \end{align*}
 where we use cycle notation for elements of $S_3$.  As each of these pairs defines a single morphism  in $\mathcal M_{inv}$,  there is exactly one morphism from $C_1$ to each of the conjugacy classes $C_2$, $C'_2$, $C_3$, $C_4$. As $\mathcal M_{inv}$ is a groupoid, there is exactly one morphism between any two of these conjugacy classes.
 
Each of the 1-cocycles $\langle \id, c \rangle$, $\langle c,\id \rangle$, $\langle c,c \rangle$ with $c\in A_3$  defines a morphism in $\mathcal M^{coH}$ within  the conjugacy classes $C_5$, $C'_5$, $C_6$.  The morphisms between  objects in $C_5$ in  $\mathcal M^{coH}$ are
 \begin{align*}
&\langle \id,\id\rangle: (\id, (12))\to (\id,(12))  & &\langle\id,\id\rangle: (\id, (13))\to (\id,(13))  & &\langle\id,\id\rangle: (\id, (23))\to (\id,(23))\\ 
&\langle \id, (123)\rangle: (\id,(12))\to (\id,(13)) & &\langle\id, (123)\rangle: (\id, (13))\to (\id, (23)) & &\langle\id, (123)\rangle: (\id, (23))\to (\id, (12))\\
&\langle\id, (132)\rangle: (\id,(12))\to (\id,(23)) & &\langle\id, (132)\rangle: (\id, (13))\to (\id, (12)) & &\langle\id, (132)\rangle: (\id, (23))\to (\id, (13)). 
 \end{align*}
All  morphisms in the first line are conjugate.  The first morphism in the second line is conjugate to  the  morphisms in the second and  third line via cyclic permutations and transpositions.
As
 \begin{align*}
 \langle\id,\id\rangle&=\langle\id, (123)\rangle\circ \langle\id, (132)\rangle: (\id, (12))\to (\id, (12)) \\
\langle\id, (123)\rangle&=\langle\id, (132)\rangle\circ \langle\id, (132)\rangle: (\id, (12))\to (\id, (13)) . 
 \end{align*}
 with $\langle\id, (132)\rangle\sim \langle\id, (123)\rangle$, all morphisms are identified by the relation  in Theorem \ref{th:kitaevexplicit} and define a single morphism in $\mathcal M_{inv}$. Hence, the identity morphism 
  is the only automorphism of $C_5$ in $\mathcal M_{inv}$ and likewise  for $C'_5$ and $C_6$.
Thus $\mathcal M_{inv}$ is the groupoid in Figure \ref{fig:conjpic}. 
\end{example}

\begin{figure}
\begin{center}
\tikzset{every picture/.style={line width=0.75pt}} 
\begin{tikzpicture}[x=0.75pt,y=0.75pt,yscale=-1.8,xscale=1.8]
\draw   (65.95,38.27) .. controls (65.95,34.14) and (69.3,30.8) .. (73.42,30.8) .. controls (77.54,30.8) and (80.89,34.14) .. (80.89,38.27) .. controls (80.89,42.39) and (77.54,45.73) .. (73.42,45.73) .. controls (69.3,45.73) and (65.95,42.39) .. (65.95,38.27) -- cycle ;
\draw   (119.95,45.47) .. controls (119.95,41.34) and (123.3,38) .. (127.42,38) .. controls (131.54,38) and (134.89,41.34) .. (134.89,45.47) .. controls (134.89,49.59) and (131.54,52.93) .. (127.42,52.93) .. controls (123.3,52.93) and (119.95,49.59) .. (119.95,45.47) -- cycle ;
\draw   (125.55,96.27) .. controls (125.55,92.14) and (128.9,88.8) .. (133.02,88.8) .. controls (137.14,88.8) and (140.49,92.14) .. (140.49,96.27) .. controls (140.49,100.39) and (137.14,103.73) .. (133.02,103.73) .. controls (128.9,103.73) and (125.55,100.39) .. (125.55,96.27) -- cycle ;
\draw   (75.84,115.75) .. controls (75.84,111.63) and (79.18,108.29) .. (83.31,108.29) .. controls (87.43,108.29) and (90.77,111.63) .. (90.77,115.75) .. controls (90.77,119.88) and (87.43,123.22) .. (83.31,123.22) .. controls (79.18,123.22) and (75.84,119.88) .. (75.84,115.75) -- cycle ;
\draw   (38.35,79.35) .. controls (38.35,75.23) and (41.7,71.89) .. (45.82,71.89) .. controls (49.94,71.89) and (53.29,75.23) .. (53.29,79.35) .. controls (53.29,83.48) and (49.94,86.82) .. (45.82,86.82) .. controls (41.7,86.82) and (38.35,83.48) .. (38.35,79.35) -- cycle ;
\draw [color={rgb, 255:red, 27; green, 16; blue, 224 }  ,draw opacity=1 ]   (80,33.8) .. controls (88.73,26.97) and (109.05,25.5) .. (120.09,35.98) ;
\draw [shift={(122.09,38.13)}, rotate = 230.91] [fill={rgb, 255:red, 27; green, 16; blue, 224 }  ,fill opacity=1 ][line width=0.08]  [draw opacity=0] (7.14,-3.43) -- (0,0) -- (7.14,3.43) -- cycle    ;
\draw   (187.55,41.07) .. controls (187.55,36.94) and (190.9,33.6) .. (195.02,33.6) .. controls (199.14,33.6) and (202.49,36.94) .. (202.49,41.07) .. controls (202.49,45.19) and (199.14,48.53) .. (195.02,48.53) .. controls (190.9,48.53) and (187.55,45.19) .. (187.55,41.07) -- cycle ;
\draw   (188.75,77.87) .. controls (188.75,73.74) and (192.1,70.4) .. (196.22,70.4) .. controls (200.34,70.4) and (203.69,73.74) .. (203.69,77.87) .. controls (203.69,81.99) and (200.34,85.33) .. (196.22,85.33) .. controls (192.1,85.33) and (188.75,81.99) .. (188.75,77.87) -- cycle ;
\draw   (189.95,116.67) .. controls (189.95,112.54) and (193.3,109.2) .. (197.42,109.2) .. controls (201.54,109.2) and (204.89,112.54) .. (204.89,116.67) .. controls (204.89,120.79) and (201.54,124.13) .. (197.42,124.13) .. controls (193.3,124.13) and (189.95,120.79) .. (189.95,116.67) -- cycle ;
\draw [color={rgb, 255:red, 27; green, 16; blue, 224 }  ,draw opacity=1 ]   (132,52.2) .. controls (138.89,57.05) and (143.39,69.88) .. (139.97,86.07) ;
\draw [shift={(139.29,88.93)}, rotate = 284.76] [fill={rgb, 255:red, 27; green, 16; blue, 224 }  ,fill opacity=1 ][line width=0.08]  [draw opacity=0] (7.14,-3.43) -- (0,0) -- (7.14,3.43) -- cycle    ;
\draw [color={rgb, 255:red, 27; green, 16; blue, 224 }  ,draw opacity=1 ]   (130.57,103.23) .. controls (123.97,115.88) and (112.03,121.41) .. (93.95,119.98) ;
\draw [shift={(91.05,119.69)}, rotate = 6.61] [fill={rgb, 255:red, 27; green, 16; blue, 224 }  ,fill opacity=1 ][line width=0.08]  [draw opacity=0] (7.14,-3.43) -- (0,0) -- (7.14,3.43) -- cycle    ;
\draw [color={rgb, 255:red, 27; green, 16; blue, 224 }  ,draw opacity=1 ]   (74.48,117.12) .. controls (60.85,112.04) and (51.55,103.98) .. (46.36,91.27) ;
\draw [shift={(45.33,88.55)}, rotate = 70.94] [fill={rgb, 255:red, 27; green, 16; blue, 224 }  ,fill opacity=1 ][line width=0.08]  [draw opacity=0] (7.14,-3.43) -- (0,0) -- (7.14,3.43) -- cycle    ;
\draw [color={rgb, 255:red, 27; green, 16; blue, 224 }  ,draw opacity=1 ]   (44.29,70.94) .. controls (45.26,57.56) and (52.49,48.71) .. (63.15,42.02) ;
\draw [shift={(65.62,40.55)}, rotate = 150.26] [fill={rgb, 255:red, 27; green, 16; blue, 224 }  ,fill opacity=1 ][line width=0.08]  [draw opacity=0] (7.14,-3.43) -- (0,0) -- (7.14,3.43) -- cycle    ;
\draw [color={rgb, 255:red, 27; green, 16; blue, 224 }  ,draw opacity=1 ]   (119.62,41.4) .. controls (110.89,35.36) and (100.8,32.06) .. (85.49,37.27) ;
\draw [shift={(82.76,38.26)}, rotate = 338.7] [fill={rgb, 255:red, 27; green, 16; blue, 224 }  ,fill opacity=1 ][line width=0.08]  [draw opacity=0] (7.14,-3.43) -- (0,0) -- (7.14,3.43) -- cycle    ;
\draw [color={rgb, 255:red, 27; green, 16; blue, 224 }  ,draw opacity=1 ]   (133.02,88.8) .. controls (137.99,76.61) and (135.83,65.81) .. (131.1,56.86) ;
\draw [shift={(129.62,54.26)}, rotate = 58.78] [fill={rgb, 255:red, 27; green, 16; blue, 224 }  ,fill opacity=1 ][line width=0.08]  [draw opacity=0] (7.14,-3.43) -- (0,0) -- (7.14,3.43) -- cycle    ;
\draw [color={rgb, 255:red, 27; green, 16; blue, 224 }  ,draw opacity=1 ]   (91.71,114.66) .. controls (103.19,115.87) and (116.44,111.61) .. (124.72,104.25) ;
\draw [shift={(126.76,102.26)}, rotate = 133.03] [fill={rgb, 255:red, 27; green, 16; blue, 224 }  ,fill opacity=1 ][line width=0.08]  [draw opacity=0] (7.14,-3.43) -- (0,0) -- (7.14,3.43) -- cycle    ;
\draw [color={rgb, 255:red, 27; green, 16; blue, 224 }  ,draw opacity=1 ]   (51.05,85.69) .. controls (54.81,96.97) and (60.08,102.44) .. (72.33,110.18) ;
\draw [shift={(74.76,111.69)}, rotate = 211.48] [fill={rgb, 255:red, 27; green, 16; blue, 224 }  ,fill opacity=1 ][line width=0.08]  [draw opacity=0] (7.14,-3.43) -- (0,0) -- (7.14,3.43) -- cycle    ;
\draw [color={rgb, 255:red, 27; green, 16; blue, 224 }  ,draw opacity=1 ]   (68.48,45.12) .. controls (59.61,48.61) and (53.02,54.63) .. (50.13,68.86) ;
\draw [shift={(49.62,71.69)}, rotate = 278.97] [fill={rgb, 255:red, 27; green, 16; blue, 224 }  ,fill opacity=1 ][line width=0.08]  [draw opacity=0] (7.14,-3.43) -- (0,0) -- (7.14,3.43) -- cycle    ;
\draw [color={rgb, 255:red, 27; green, 16; blue, 224 }  ,draw opacity=1 ]   (80.57,42.94) .. controls (98.03,45) and (117.52,71.95) .. (128.38,85.72) ;
\draw [shift={(130.19,87.98)}, rotate = 230.91] [fill={rgb, 255:red, 27; green, 16; blue, 224 }  ,fill opacity=1 ][line width=0.08]  [draw opacity=0] (7.14,-3.43) -- (0,0) -- (7.14,3.43) -- cycle    ;
\draw [color={rgb, 255:red, 27; green, 16; blue, 224 }  ,draw opacity=1 ]   (126.76,90.55) .. controls (110.22,72.35) and (102.99,60.27) .. (80.94,46.61) ;
\draw [shift={(78.48,45.12)}, rotate = 30.76] [fill={rgb, 255:red, 27; green, 16; blue, 224 }  ,fill opacity=1 ][line width=0.08]  [draw opacity=0] (7.14,-3.43) -- (0,0) -- (7.14,3.43) -- cycle    ;
\draw [color={rgb, 255:red, 27; green, 16; blue, 224 }  ,draw opacity=1 ]   (75.71,46.37) .. controls (79.15,56.31) and (83.71,84.66) .. (84.4,103.39) ;
\draw [shift={(84.48,106.26)}, rotate = 269.12] [fill={rgb, 255:red, 27; green, 16; blue, 224 }  ,fill opacity=1 ][line width=0.08]  [draw opacity=0] (7.14,-3.43) -- (0,0) -- (7.14,3.43) -- cycle    ;
\draw [color={rgb, 255:red, 27; green, 16; blue, 224 }  ,draw opacity=1 ]   (79.43,108.66) .. controls (79.87,86.24) and (73.72,64.16) .. (72.86,50.83) ;
\draw [shift={(72.76,47.98)}, rotate = 90] [fill={rgb, 255:red, 27; green, 16; blue, 224 }  ,fill opacity=1 ][line width=0.08]  [draw opacity=0] (7.14,-3.43) -- (0,0) -- (7.14,3.43) -- cycle    ;
\draw [color={rgb, 255:red, 27; green, 16; blue, 224 }  ,draw opacity=1 ]   (118,45.51) .. controls (102.36,47.59) and (71.69,59.93) .. (55.02,71.74) ;
\draw [shift={(52.76,73.4)}, rotate = 322.48] [fill={rgb, 255:red, 27; green, 16; blue, 224 }  ,fill opacity=1 ][line width=0.08]  [draw opacity=0] (7.14,-3.43) -- (0,0) -- (7.14,3.43) -- cycle    ;
\draw [color={rgb, 255:red, 27; green, 16; blue, 224 }  ,draw opacity=1 ]   (54.76,76.26) .. controls (74.37,64.71) and (100.28,54.93) .. (115.9,50) ;
\draw [shift={(118.76,49.12)}, rotate = 163.2] [fill={rgb, 255:red, 27; green, 16; blue, 224 }  ,fill opacity=1 ][line width=0.08]  [draw opacity=0] (7.14,-3.43) -- (0,0) -- (7.14,3.43) -- cycle    ;
\draw [color={rgb, 255:red, 27; green, 16; blue, 224 }  ,draw opacity=1 ]   (121.43,51.23) .. controls (108.93,64.91) and (100.78,80.55) .. (90.98,105.77) ;
\draw [shift={(89.9,108.55)}, rotate = 290.96] [fill={rgb, 255:red, 27; green, 16; blue, 224 }  ,fill opacity=1 ][line width=0.08]  [draw opacity=0] (7.14,-3.43) -- (0,0) -- (7.14,3.43) -- cycle    ;
\draw [color={rgb, 255:red, 27; green, 16; blue, 224 }  ,draw opacity=1 ]   (91.43,111.51) .. controls (100.25,103.61) and (101.56,87.11) .. (124.13,57.16) ;
\draw [shift={(125.9,54.83)}, rotate = 127.69] [fill={rgb, 255:red, 27; green, 16; blue, 224 }  ,fill opacity=1 ][line width=0.08]  [draw opacity=0] (7.14,-3.43) -- (0,0) -- (7.14,3.43) -- cycle    ;
\draw [color={rgb, 255:red, 27; green, 16; blue, 224 }  ,draw opacity=1 ]   (55.14,79.51) .. controls (72.87,74.01) and (106,85.19) .. (121.4,92.71) ;
\draw [shift={(123.9,93.98)}, rotate = 207.95] [fill={rgb, 255:red, 27; green, 16; blue, 224 }  ,fill opacity=1 ][line width=0.08]  [draw opacity=0] (7.14,-3.43) -- (0,0) -- (7.14,3.43) -- cycle    ;
\draw [color={rgb, 255:red, 27; green, 16; blue, 224 }  ,draw opacity=1 ]   (125.62,99.98) .. controls (103.95,89.28) and (80.97,83.85) .. (57.17,83.43) ;
\draw [shift={(54.19,83.4)}, rotate = 360] [fill={rgb, 255:red, 27; green, 16; blue, 224 }  ,fill opacity=1 ][line width=0.08]  [draw opacity=0] (7.14,-3.43) -- (0,0) -- (7.14,3.43) -- cycle    ;
\draw [color={rgb, 255:red, 27; green, 16; blue, 224 }  ,draw opacity=1 ]   (66.29,32.94) .. controls (56.58,16.01) and (70.44,15.41) .. (76.01,27.08) ;
\draw [shift={(77.05,29.69)}, rotate = 252.26] [fill={rgb, 255:red, 27; green, 16; blue, 224 }  ,fill opacity=1 ][line width=0.08]  [draw opacity=0] (7.14,-3.43) -- (0,0) -- (7.14,3.43) -- cycle    ;
\draw [color={rgb, 255:red, 27; green, 16; blue, 224 }  ,draw opacity=1 ]   (127.42,38) .. controls (136.84,16.53) and (151.12,31.39) .. (138.67,43.27) ;
\draw [shift={(136.48,45.12)}, rotate = 323.26] [fill={rgb, 255:red, 27; green, 16; blue, 224 }  ,fill opacity=1 ][line width=0.08]  [draw opacity=0] (7.14,-3.43) -- (0,0) -- (7.14,3.43) -- cycle    ;
\draw [color={rgb, 255:red, 27; green, 16; blue, 224 }  ,draw opacity=1 ]   (140.86,94.37) .. controls (159.42,81.1) and (158.84,113.48) .. (140.31,104.47) ;
\draw [shift={(137.9,103.12)}, rotate = 32.22] [fill={rgb, 255:red, 27; green, 16; blue, 224 }  ,fill opacity=1 ][line width=0.08]  [draw opacity=0] (7.14,-3.43) -- (0,0) -- (7.14,3.43) -- cycle    ;
\draw [color={rgb, 255:red, 27; green, 16; blue, 224 }  ,draw opacity=1 ]   (41.62,87.4) .. controls (32.52,100.64) and (13.31,88.33) .. (34.32,78.98) ;
\draw [shift={(36.76,77.98)}, rotate = 159.23] [fill={rgb, 255:red, 27; green, 16; blue, 224 }  ,fill opacity=1 ][line width=0.08]  [draw opacity=0] (7.14,-3.43) -- (0,0) -- (7.14,3.43) -- cycle    ;
\draw [color={rgb, 255:red, 27; green, 16; blue, 224 }  ,draw opacity=1 ]   (89.05,122.83) .. controls (90.65,142.33) and (75.77,137.86) .. (76.16,124.35) ;
\draw [shift={(76.48,121.4)}, rotate = 100.3] [fill={rgb, 255:red, 27; green, 16; blue, 224 }  ,fill opacity=1 ][line width=0.08]  [draw opacity=0] (7.14,-3.43) -- (0,0) -- (7.14,3.43) -- cycle    ;
\draw [color={rgb, 255:red, 27; green, 16; blue, 224 }  ,draw opacity=1 ]   (192.29,34.37) .. controls (200.23,15.7) and (230.31,28.04) .. (206.29,41.31) ;
\draw [shift={(203.9,42.55)}, rotate = 334.13] [fill={rgb, 255:red, 27; green, 16; blue, 224 }  ,fill opacity=1 ][line width=0.08]  [draw opacity=0] (7.14,-3.43) -- (0,0) -- (7.14,3.43) -- cycle    ;
\draw [color={rgb, 255:red, 27; green, 16; blue, 224 }  ,draw opacity=1 ]   (196.22,70.4) .. controls (212.92,50.52) and (229.13,71.58) .. (207.9,78.86) ;
\draw [shift={(205.05,79.69)}, rotate = 346.12] [fill={rgb, 255:red, 27; green, 16; blue, 224 }  ,fill opacity=1 ][line width=0.08]  [draw opacity=0] (7.14,-3.43) -- (0,0) -- (7.14,3.43) -- cycle    ;
\draw [color={rgb, 255:red, 27; green, 16; blue, 224 }  ,draw opacity=1 ]   (200.29,108.66) .. controls (201.3,89.71) and (233.86,102.58) .. (207.5,115.47) ;
\draw [shift={(204.89,116.67)}, rotate = 336.85] [fill={rgb, 255:red, 27; green, 16; blue, 224 }  ,fill opacity=1 ][line width=0.08]  [draw opacity=0] (7.14,-3.43) -- (0,0) -- (7.14,3.43) -- cycle    ;
\draw (70.74,35.5) node [anchor=north west][inner sep=0.75pt]   [align=left] {{\footnotesize $C_1$}};
\draw (124.31,42.91) node [anchor=north west][inner sep=0.75pt]   [align=left] {{\footnotesize $C_2$}};
\draw (129.9,93) node [anchor=north west][inner sep=0.75pt]   [align=left] {{\footnotesize $C'_2$}};
\draw (80.3,113.2) node [anchor=north west][inner sep=0.75pt]   [align=left] {{\footnotesize $C_3$}};
\draw (42.7,76.4) node [anchor=north west][inner sep=0.75pt]   [align=left] {{\footnotesize $C_4$}};
\draw (192.2,37.8) node [anchor=north west][inner sep=0.75pt]   [align=left] {{\footnotesize $C_5$}};
\draw (192.9,74.6) node [anchor=north west][inner sep=0.75pt]   [align=left] {{\footnotesize $C'_5$}};
\draw (192.1,113.4) node [anchor=north west][inner sep=0.75pt]   [align=left] {{\footnotesize $C_6$}};
\end{tikzpicture}
\end{center}
\caption{The groupoid $\mathcal M_{inv}$ from Example \ref{ex:minv}}
\label{fig:conjpic}
\end{figure}

\begin{figure}
\begin{center}
\tikzset{every picture/.style={line width=0.75pt}}     
\begin{tikzpicture}[x=0.75pt,y=0.75pt,yscale=-1.8,xscale=1.8]
\draw   (33.95,47.7) .. controls (33.95,43.57) and (37.3,40.23) .. (41.42,40.23) .. controls (45.54,40.23) and (48.89,43.57) .. (48.89,47.7) .. controls (48.89,51.82) and (45.54,55.16) .. (41.42,55.16) .. controls (37.3,55.16) and (33.95,51.82) .. (33.95,47.7) -- cycle ;
\draw [color={rgb, 255:red, 27; green, 16; blue, 224 }  ,draw opacity=1 ]   (37.99,40) .. controls (39.54,22.01) and (60.68,27.13) .. (51.07,40.55) ;
\draw [shift={(49.29,42.73)}, rotate = 312.34] [fill={rgb, 255:red, 27; green, 16; blue, 224 }  ,fill opacity=1 ][line width=0.08]  [draw opacity=0] (5.36,-2.57) -- (0,0) -- (5.36,2.57) -- cycle    ;
\draw   (89.67,47.41) .. controls (89.67,43.29) and (93.01,39.94) .. (97.13,39.94) .. controls (101.26,39.94) and (104.6,43.29) .. (104.6,47.41) .. controls (104.6,51.53) and (101.26,54.88) .. (97.13,54.88) .. controls (93.01,54.88) and (89.67,51.53) .. (89.67,47.41) -- cycle ;
\draw [color={rgb, 255:red, 27; green, 16; blue, 224 }  ,draw opacity=1 ]   (93.71,39.71) .. controls (85.25,23.76) and (97.8,21.35) .. (97.35,36.78) ;
\draw [shift={(97.13,39.66)}, rotate = 276.76] [fill={rgb, 255:red, 27; green, 16; blue, 224 }  ,fill opacity=1 ][line width=0.08]  [draw opacity=0] (5.36,-2.57) -- (0,0) -- (5.36,2.57) -- cycle    ;
\draw [color={rgb, 255:red, 27; green, 16; blue, 224 }  ,draw opacity=1 ]   (105.9,48.55) .. controls (125.34,49.9) and (121.06,59.67) .. (107.5,53.02) ;
\draw [shift={(105.05,51.69)}, rotate = 30.65] [fill={rgb, 255:red, 27; green, 16; blue, 224 }  ,fill opacity=1 ][line width=0.08]  [draw opacity=0] (5.36,-2.57) -- (0,0) -- (5.36,2.57) -- cycle    ;
\draw [color={rgb, 255:red, 27; green, 16; blue, 224 }  ,draw opacity=1 ]   (103.9,53.12) .. controls (117.06,63.86) and (109.52,69.3) .. (102.4,57.58) ;
\draw [shift={(101.05,55.12)}, rotate = 63.43] [fill={rgb, 255:red, 27; green, 16; blue, 224 }  ,fill opacity=1 ][line width=0.08]  [draw opacity=0] (5.36,-2.57) -- (0,0) -- (5.36,2.57) -- cycle    ;
\draw [color={rgb, 255:red, 27; green, 16; blue, 224 }  ,draw opacity=1 ]   (98.48,55.4) .. controls (103.09,69.25) and (91.98,75.1) .. (95.02,58.5) ;
\draw [shift={(95.62,55.69)}, rotate = 103.65] [fill={rgb, 255:red, 27; green, 16; blue, 224 }  ,fill opacity=1 ][line width=0.08]  [draw opacity=0] (5.36,-2.57) -- (0,0) -- (5.36,2.57) -- cycle    ;
\draw [color={rgb, 255:red, 27; green, 16; blue, 224 }  ,draw opacity=1 ]   (93.62,54.26) .. controls (88.81,68.69) and (79.51,63.38) .. (88.13,54.22) ;
\draw [shift={(90.19,52.26)}, rotate = 139.4] [fill={rgb, 255:red, 27; green, 16; blue, 224 }  ,fill opacity=1 ][line width=0.08]  [draw opacity=0] (5.36,-2.57) -- (0,0) -- (5.36,2.57) -- cycle    ;
\draw [color={rgb, 255:red, 27; green, 16; blue, 224 }  ,draw opacity=1 ]   (89.13,49.71) .. controls (73.31,61.36) and (70.71,44.15) .. (86.95,46.57) ;
\draw [shift={(89.67,47.12)}, rotate = 194.22] [fill={rgb, 255:red, 27; green, 16; blue, 224 }  ,fill opacity=1 ][line width=0.08]  [draw opacity=0] (5.36,-2.57) -- (0,0) -- (5.36,2.57) -- cycle    ;
\draw [color={rgb, 255:red, 27; green, 16; blue, 224 }  ,draw opacity=1 ]   (89.42,44) .. controls (68.55,38.04) and (79.65,27.5) .. (89.37,39.87) ;
\draw [shift={(91.05,42.26)}, rotate = 237.99] [fill={rgb, 255:red, 27; green, 16; blue, 224 }  ,fill opacity=1 ][line width=0.08]  [draw opacity=0] (5.36,-2.57) -- (0,0) -- (5.36,2.57) -- cycle    ;
\draw [color={rgb, 255:red, 27; green, 16; blue, 224 }  ,draw opacity=1 ]   (105.33,42.83) .. controls (117.82,36.86) and (123.86,40.43) .. (107.39,46.44) ;
\draw [shift={(104.6,47.41)}, rotate = 341.69] [fill={rgb, 255:red, 27; green, 16; blue, 224 }  ,fill opacity=1 ][line width=0.08]  [draw opacity=0] (5.36,-2.57) -- (0,0) -- (5.36,2.57) -- cycle    ;
\draw [color={rgb, 255:red, 27; green, 16; blue, 224 }  ,draw opacity=1 ]   (100.85,40.29) .. controls (107.55,22.08) and (119.23,28.39) .. (106.42,38.62) ;
\draw [shift={(104.19,40.26)}, rotate = 325.62] [fill={rgb, 255:red, 27; green, 16; blue, 224 }  ,fill opacity=1 ][line width=0.08]  [draw opacity=0] (5.36,-2.57) -- (0,0) -- (5.36,2.57) -- cycle    ;

\begin{scope}[shift={(150, 0)}]
\draw   (138.81,47.41) .. controls (138.81,43.29) and (142.15,39.94) .. (146.28,39.94) .. controls (150.4,39.94) and (153.74,43.29) .. (153.74,47.41) .. controls (153.74,51.53) and (150.4,54.88) .. (146.28,54.88) .. controls (142.15,54.88) and (138.81,51.53) .. (138.81,47.41) -- cycle ;
\draw [color={rgb, 255:red, 27; green, 16; blue, 224 }  ,draw opacity=1 ]   (142.85,39.71) .. controls (148.61,21.12) and (161.2,27.01) .. (156.24,41.64) ;
\draw [shift={(155.19,44.28)}, rotate = 294.44] [fill={rgb, 255:red, 27; green, 16; blue, 224 }  ,fill opacity=1 ][line width=0.08]  [draw opacity=0] (5.36,-2.57) -- (0,0) -- (5.36,2.57) -- cycle    ;
\end{scope}
\draw   (189.67,47.12) .. controls (189.67,43) and (193.01,39.66) .. (197.13,39.66) .. controls (201.26,39.66) and (204.6,43) .. (204.6,47.12) .. controls (204.6,51.25) and (201.26,54.59) .. (197.13,54.59) .. controls (193.01,54.59) and (189.67,51.25) .. (189.67,47.12) -- cycle ;
\draw [color={rgb, 255:red, 27; green, 16; blue, 224 }  ,draw opacity=1 ]   (193.71,39.43) .. controls (185.25,23.48) and (197.8,21.07) .. (197.35,36.5) ;
\draw [shift={(197.13,39.37)}, rotate = 276.76] [fill={rgb, 255:red, 27; green, 16; blue, 224 }  ,fill opacity=1 ][line width=0.08]  [draw opacity=0] (5.36,-2.57) -- (0,0) -- (5.36,2.57) -- cycle    ;
\draw [color={rgb, 255:red, 27; green, 16; blue, 224 }  ,draw opacity=1 ]   (205.9,48.26) .. controls (225.34,49.61) and (221.06,59.38) .. (207.5,52.74) ;
\draw [shift={(205.05,51.4)}, rotate = 30.65] [fill={rgb, 255:red, 27; green, 16; blue, 224 }  ,fill opacity=1 ][line width=0.08]  [draw opacity=0] (5.36,-2.57) -- (0,0) -- (5.36,2.57) -- cycle    ;
\draw [color={rgb, 255:red, 27; green, 16; blue, 224 }  ,draw opacity=1 ]   (203.9,52.83) .. controls (217.06,63.58) and (209.52,69.02) .. (202.4,57.29) ;
\draw [shift={(201.05,54.83)}, rotate = 63.43] [fill={rgb, 255:red, 27; green, 16; blue, 224 }  ,fill opacity=1 ][line width=0.08]  [draw opacity=0] (5.36,-2.57) -- (0,0) -- (5.36,2.57) -- cycle    ;
\draw [color={rgb, 255:red, 27; green, 16; blue, 224 }  ,draw opacity=1 ]   (198.48,55.12) .. controls (203.09,68.96) and (191.98,74.81) .. (195.02,58.21) ;
\draw [shift={(195.62,55.4)}, rotate = 103.65] [fill={rgb, 255:red, 27; green, 16; blue, 224 }  ,fill opacity=1 ][line width=0.08]  [draw opacity=0] (5.36,-2.57) -- (0,0) -- (5.36,2.57) -- cycle    ;
\draw [color={rgb, 255:red, 27; green, 16; blue, 224 }  ,draw opacity=1 ]   (193.62,53.98) .. controls (188.81,68.4) and (179.51,63.1) .. (188.13,53.94) ;
\draw [shift={(190.19,51.98)}, rotate = 139.4] [fill={rgb, 255:red, 27; green, 16; blue, 224 }  ,fill opacity=1 ][line width=0.08]  [draw opacity=0] (5.36,-2.57) -- (0,0) -- (5.36,2.57) -- cycle    ;
\draw [color={rgb, 255:red, 27; green, 16; blue, 224 }  ,draw opacity=1 ]   (189.13,49.43) .. controls (173.31,61.08) and (170.71,43.87) .. (186.95,46.29) ;
\draw [shift={(189.67,46.84)}, rotate = 194.22] [fill={rgb, 255:red, 27; green, 16; blue, 224 }  ,fill opacity=1 ][line width=0.08]  [draw opacity=0] (5.36,-2.57) -- (0,0) -- (5.36,2.57) -- cycle    ;
\draw [color={rgb, 255:red, 27; green, 16; blue, 224 }  ,draw opacity=1 ]   (189.42,43.71) .. controls (168.55,37.75) and (179.65,27.22) .. (189.37,39.58) ;
\draw [shift={(191.05,41.98)}, rotate = 237.99] [fill={rgb, 255:red, 27; green, 16; blue, 224 }  ,fill opacity=1 ][line width=0.08]  [draw opacity=0] (5.36,-2.57) -- (0,0) -- (5.36,2.57) -- cycle    ;
\draw [color={rgb, 255:red, 27; green, 16; blue, 224 }  ,draw opacity=1 ]   (205.33,42.55) .. controls (217.82,36.58) and (223.86,40.15) .. (207.39,46.16) ;
\draw [shift={(204.6,47.12)}, rotate = 341.69] [fill={rgb, 255:red, 27; green, 16; blue, 224 }  ,fill opacity=1 ][line width=0.08]  [draw opacity=0] (5.36,-2.57) -- (0,0) -- (5.36,2.57) -- cycle    ;
\draw [color={rgb, 255:red, 27; green, 16; blue, 224 }  ,draw opacity=1 ]   (200.85,40) .. controls (207.55,21.79) and (219.23,28.1) .. (206.42,38.33) ;
\draw [shift={(204.19,39.98)}, rotate = 325.62] [fill={rgb, 255:red, 27; green, 16; blue, 224 }  ,fill opacity=1 ][line width=0.08]  [draw opacity=0] (5.36,-2.57) -- (0,0) -- (5.36,2.57) -- cycle    ;
\draw   (241.38,47.41) .. controls (241.38,43.29) and (244.72,39.94) .. (248.85,39.94) .. controls (252.97,39.94) and (256.32,43.29) .. (256.32,47.41) .. controls (256.32,51.53) and (252.97,54.88) .. (248.85,54.88) .. controls (244.72,54.88) and (241.38,51.53) .. (241.38,47.41) -- cycle ;
\draw [color={rgb, 255:red, 27; green, 16; blue, 224 }  ,draw opacity=1 ]   (245.42,39.71) .. controls (236.97,23.76) and (249.52,21.35) .. (249.07,36.78) ;
\draw [shift={(248.85,39.66)}, rotate = 276.76] [fill={rgb, 255:red, 27; green, 16; blue, 224 }  ,fill opacity=1 ][line width=0.08]  [draw opacity=0] (5.36,-2.57) -- (0,0) -- (5.36,2.57) -- cycle    ;
\draw [color={rgb, 255:red, 27; green, 16; blue, 224 }  ,draw opacity=1 ]   (257.62,48.55) .. controls (277.06,49.9) and (272.77,59.67) .. (259.22,53.02) ;
\draw [shift={(256.76,51.69)}, rotate = 30.65] [fill={rgb, 255:red, 27; green, 16; blue, 224 }  ,fill opacity=1 ][line width=0.08]  [draw opacity=0] (5.36,-2.57) -- (0,0) -- (5.36,2.57) -- cycle    ;
\draw [color={rgb, 255:red, 27; green, 16; blue, 224 }  ,draw opacity=1 ]   (255.62,53.12) .. controls (268.78,63.86) and (261.24,69.3) .. (254.12,57.58) ;
\draw [shift={(252.76,55.12)}, rotate = 63.43] [fill={rgb, 255:red, 27; green, 16; blue, 224 }  ,fill opacity=1 ][line width=0.08]  [draw opacity=0] (5.36,-2.57) -- (0,0) -- (5.36,2.57) -- cycle    ;
\draw [color={rgb, 255:red, 27; green, 16; blue, 224 }  ,draw opacity=1 ]   (250.19,55.4) .. controls (254.8,69.25) and (243.69,75.1) .. (246.73,58.5) ;
\draw [shift={(247.33,55.69)}, rotate = 103.65] [fill={rgb, 255:red, 27; green, 16; blue, 224 }  ,fill opacity=1 ][line width=0.08]  [draw opacity=0] (5.36,-2.57) -- (0,0) -- (5.36,2.57) -- cycle    ;
\draw [color={rgb, 255:red, 27; green, 16; blue, 224 }  ,draw opacity=1 ]   (245.33,54.26) .. controls (240.52,68.69) and (231.22,63.38) .. (239.84,54.22) ;
\draw [shift={(241.9,52.26)}, rotate = 139.4] [fill={rgb, 255:red, 27; green, 16; blue, 224 }  ,fill opacity=1 ][line width=0.08]  [draw opacity=0] (5.36,-2.57) -- (0,0) -- (5.36,2.57) -- cycle    ;
\draw [color={rgb, 255:red, 27; green, 16; blue, 224 }  ,draw opacity=1 ]   (240.85,49.71) .. controls (225.02,61.36) and (222.43,44.15) .. (238.66,46.57) ;
\draw [shift={(241.38,47.12)}, rotate = 194.22] [fill={rgb, 255:red, 27; green, 16; blue, 224 }  ,fill opacity=1 ][line width=0.08]  [draw opacity=0] (5.36,-2.57) -- (0,0) -- (5.36,2.57) -- cycle    ;
\draw [color={rgb, 255:red, 27; green, 16; blue, 224 }  ,draw opacity=1 ]   (241.13,44) .. controls (220.26,38.04) and (231.36,27.5) .. (241.08,39.87) ;
\draw [shift={(242.76,42.26)}, rotate = 237.99] [fill={rgb, 255:red, 27; green, 16; blue, 224 }  ,fill opacity=1 ][line width=0.08]  [draw opacity=0] (5.36,-2.57) -- (0,0) -- (5.36,2.57) -- cycle    ;
\draw [color={rgb, 255:red, 27; green, 16; blue, 224 }  ,draw opacity=1 ]   (257.05,42.83) .. controls (269.53,36.86) and (275.57,40.43) .. (259.1,46.44) ;
\draw [shift={(256.32,47.41)}, rotate = 341.69] [fill={rgb, 255:red, 27; green, 16; blue, 224 }  ,fill opacity=1 ][line width=0.08]  [draw opacity=0] (5.36,-2.57) -- (0,0) -- (5.36,2.57) -- cycle    ;
\draw [color={rgb, 255:red, 27; green, 16; blue, 224 }  ,draw opacity=1 ]   (252.56,40.29) .. controls (259.27,22.08) and (270.95,28.39) .. (258.14,38.62) ;
\draw [shift={(255.9,40.26)}, rotate = 325.62] [fill={rgb, 255:red, 27; green, 16; blue, 224 }  ,fill opacity=1 ][line width=0.08]  [draw opacity=0] (5.36,-2.57) -- (0,0) -- (5.36,2.57) -- cycle    ;

\begin{scope}[shift={(-154, 0)}]
\draw   (293.71,47.41) .. controls (293.71,43.29) and (297.05,39.94) .. (301.17,39.94) .. controls (305.3,39.94) and (308.64,43.29) .. (308.64,47.41) .. controls (308.64,51.53) and (305.3,54.88) .. (301.17,54.88) .. controls (297.05,54.88) and (293.71,51.53) .. (293.71,47.41) -- cycle ;
\draw [color={rgb, 255:red, 27; green, 16; blue, 224 }  ,draw opacity=1 ]   (297.74,39.71) .. controls (289.29,23.76) and (301.84,21.35) .. (301.39,36.78) ;
\draw [shift={(301.17,39.66)}, rotate = 276.76] [fill={rgb, 255:red, 27; green, 16; blue, 224 }  ,fill opacity=1 ][line width=0.08]  [draw opacity=0] (5.36,-2.57) -- (0,0) -- (5.36,2.57) -- cycle    ;
\draw [color={rgb, 255:red, 27; green, 16; blue, 224 }  ,draw opacity=1 ]   (309.94,48.55) .. controls (329.38,49.9) and (325.09,59.67) .. (311.54,53.02) ;
\draw [shift={(309.09,51.69)}, rotate = 30.65] [fill={rgb, 255:red, 27; green, 16; blue, 224 }  ,fill opacity=1 ][line width=0.08]  [draw opacity=0] (5.36,-2.57) -- (0,0) -- (5.36,2.57) -- cycle    ;
\draw [color={rgb, 255:red, 27; green, 16; blue, 224 }  ,draw opacity=1 ]   (307.94,53.12) .. controls (321.1,63.86) and (313.56,69.3) .. (306.44,57.58) ;
\draw [shift={(305.09,55.12)}, rotate = 63.43] [fill={rgb, 255:red, 27; green, 16; blue, 224 }  ,fill opacity=1 ][line width=0.08]  [draw opacity=0] (5.36,-2.57) -- (0,0) -- (5.36,2.57) -- cycle    ;
\draw [color={rgb, 255:red, 27; green, 16; blue, 224 }  ,draw opacity=1 ]   (302.51,55.4) .. controls (307.13,69.25) and (296.01,75.1) .. (299.05,58.5) ;
\draw [shift={(299.66,55.69)}, rotate = 103.65] [fill={rgb, 255:red, 27; green, 16; blue, 224 }  ,fill opacity=1 ][line width=0.08]  [draw opacity=0] (5.36,-2.57) -- (0,0) -- (5.36,2.57) -- cycle    ;
\draw [color={rgb, 255:red, 27; green, 16; blue, 224 }  ,draw opacity=1 ]   (297.66,54.26) .. controls (292.85,68.69) and (283.54,63.38) .. (292.16,54.22) ;
\draw [shift={(294.23,52.26)}, rotate = 139.4] [fill={rgb, 255:red, 27; green, 16; blue, 224 }  ,fill opacity=1 ][line width=0.08]  [draw opacity=0] (5.36,-2.57) -- (0,0) -- (5.36,2.57) -- cycle    ;
\draw [color={rgb, 255:red, 27; green, 16; blue, 224 }  ,draw opacity=1 ]   (293.17,49.71) .. controls (277.35,61.36) and (274.75,44.15) .. (290.99,46.57) ;
\draw [shift={(293.71,47.12)}, rotate = 194.22] [fill={rgb, 255:red, 27; green, 16; blue, 224 }  ,fill opacity=1 ][line width=0.08]  [draw opacity=0] (5.36,-2.57) -- (0,0) -- (5.36,2.57) -- cycle    ;
\draw [color={rgb, 255:red, 27; green, 16; blue, 224 }  ,draw opacity=1 ]   (293.46,44) .. controls (272.59,38.04) and (283.69,27.5) .. (293.41,39.87) ;
\draw [shift={(295.09,42.26)}, rotate = 237.99] [fill={rgb, 255:red, 27; green, 16; blue, 224 }  ,fill opacity=1 ][line width=0.08]  [draw opacity=0] (5.36,-2.57) -- (0,0) -- (5.36,2.57) -- cycle    ;
\draw [color={rgb, 255:red, 27; green, 16; blue, 224 }  ,draw opacity=1 ]   (309.37,42.83) .. controls (321.86,36.86) and (327.9,40.43) .. (311.43,46.44) ;
\draw [shift={(308.64,47.41)}, rotate = 341.69] [fill={rgb, 255:red, 27; green, 16; blue, 224 }  ,fill opacity=1 ][line width=0.08]  [draw opacity=0] (5.36,-2.57) -- (0,0) -- (5.36,2.57) -- cycle    ;
\draw [color={rgb, 255:red, 27; green, 16; blue, 224 }  ,draw opacity=1 ]   (304.89,40.29) .. controls (311.59,22.08) and (323.27,28.39) .. (310.46,38.62) ;
\draw [shift={(308.23,40.26)}, rotate = 325.62] [fill={rgb, 255:red, 27; green, 16; blue, 224 }  ,fill opacity=1 ][line width=0.08]  [draw opacity=0] (5.36,-2.57) -- (0,0) -- (5.36,2.57) -- cycle    ;
\end{scope}

\begin{scope}[shift={(-15, 0)}]
\draw   (338.13,47.74) .. controls (338.13,43.62) and (341.48,40.28) .. (345.6,40.28) .. controls (349.72,40.28) and (353.07,43.62) .. (353.07,47.74) .. controls (353.07,51.87) and (349.72,55.21) .. (345.6,55.21) .. controls (341.48,55.21) and (338.13,51.87) .. (338.13,47.74) -- cycle ;
\draw [color={rgb, 255:red, 27; green, 16; blue, 224 }  ,draw opacity=1 ]   (341.6,40.61) .. controls (342.64,23.43) and (359.21,27.17) .. (354.49,41.98) ;
\draw [shift={(353.45,44.68)}, rotate = 294.51] [fill={rgb, 255:red, 27; green, 16; blue, 224 }  ,fill opacity=1 ][line width=0.08]  [draw opacity=0] (5.36,-2.57) -- (0,0) -- (5.36,2.57) -- cycle    ;
\draw   (367.29,48.36) .. controls (367.29,44.24) and (370.63,40.9) .. (374.75,40.9) .. controls (378.88,40.9) and (382.22,44.24) .. (382.22,48.36) .. controls (382.22,52.49) and (378.88,55.83) .. (374.75,55.83) .. controls (370.63,55.83) and (367.29,52.49) .. (367.29,48.36) -- cycle ;
\draw [color={rgb, 255:red, 27; green, 16; blue, 224 }  ,draw opacity=1 ]   (370.75,40.56) .. controls (368.23,21.41) and (394.79,26.25) .. (383.04,41.97) ;
\draw [shift={(381.11,44.26)}, rotate = 312.88] [fill={rgb, 255:red, 27; green, 16; blue, 224 }  ,fill opacity=1 ][line width=0.08]  [draw opacity=0] (5.36,-2.57) -- (0,0) -- (5.36,2.57) -- cycle    ;
\end{scope}
\draw (37.77,43.4) node [anchor=north west][inner sep=0.75pt]   [align=left] {{\footnotesize $C_1$}};
\draw (92.99,43.4) node [anchor=north west][inner sep=0.75pt]   [align=left] {{\footnotesize $C_2$}};
\draw (142.83,43.4) node [anchor=north west][inner sep=0.75pt]   [align=left] {{\footnotesize $C'_2$}};
\draw (193.34,43.14) node [anchor=north west][inner sep=0.75pt]   [align=left] {{\footnotesize $C_3$}};
\draw (244.6,43.43) node [anchor=north west][inner sep=0.75pt]   [align=left] {{\footnotesize $C_4$}};
\begin{scope}[shift={(-15, 0)}]
\draw (306.42,43.7) node [anchor=north west][inner sep=0.75pt]   [align=left] {{\footnotesize $C_5$}};
\draw (342.12,43.7) node [anchor=north west][inner sep=0.75pt]   [align=left] {{\footnotesize $C'_5$}};
\draw (369.44,44.71) node [anchor=north west][inner sep=0.75pt]   [align=left] {{\footnotesize $C_6$}};
\end{scope}
\end{tikzpicture}

\end{center}
\caption{The groupoid $\mathcal M_{inv}$ from Example \ref{ex:minv2}}
\label{fig:conjpic2}
\end{figure}

\begin{example}\label{ex:minv2} Let $\Sigma$ be the torus  and consider the crossed module
 $(S_3,A_3,\brhd, \partial)$ with the trivial group homomorphism $\partial: A_3\to S_3$, $a\mapsto \id$ and $\brhd: S_3\times A_3\to A_3$, $b\brhd a=bab^\inv$. 
 
 Then the protected object  $\mathcal M_{inv}$ has  the same objects as in Example \ref{ex:minv}. As $\partial$ is trivial, all morphisms in $\mathcal M^{coH}$ and $\mathcal M_{inv}$ are automorphisms. The object $(\id,\id)$ in $\mathcal M^{coH}$ has the  identity morphism $\langle\id,\id\rangle$ and the following four conjugate pairs of automorphisms
 \begin{align*}
&\langle\id, (123)\rangle\sim\langle\id, (132)\rangle, & &\langle(123),\id\rangle\sim \langle(132), \id\rangle,\\
&\langle(123), (123)\rangle\sim\langle(132), (132)\rangle, & &\langle(123), (132)\rangle\sim\langle(132),(123)\rangle. 
 \end{align*}
The  relation for morphisms in Theorem \ref{th:kitaevexplicit} then implies
 \begin{align*}
&\langle\id,\id\rangle=\langle\id, (123)\rangle\circ \langle\id, (132)\rangle \sim\langle\id,(132)\rangle\circ \langle\id,(132)\rangle=\langle\id,(123)\rangle\\
&\langle\id,\id\rangle=\langle(123),\id\rangle\circ \langle(132),\id\rangle \sim\langle(132), \id\rangle\circ \langle(132),\id\rangle=\langle(123),\id\rangle\\
&\langle\id,\id\rangle=\langle(123),(123)\rangle\circ \langle(132),(132)\rangle \sim\langle(132), (132)\rangle\circ \langle(132),(132)\rangle=\langle(123),(123)\rangle\\
&\langle\id,\id\rangle=\langle(123),(132)\rangle\circ \langle(132),(123)\rangle \sim\langle(132), (123)\rangle\circ \langle(132),(123)\rangle=\langle(123),(132)\rangle.
\end{align*}
As all automorphisms of $(\id,\id)$ in $\mathcal M^{coH}$ are identified,  $C_1$ has a single automorphism in $\mathcal M_{inv}$. As in Example \ref{ex:minv}, all morphisms between objects in $C_5$ are identified and likewise for $C'_5$, $C_6$.
 
 In contrast, the automorphism group of each element of $C_2$, $C'_2$, $C_3$, $C_4$ in $\mathcal M$ and $\mathcal M^{coH}$  is $A_3\times A_3$. Automorphisms of these objects in $\mathcal M$ coincide with their automorphisms in $\mathcal M^{coH}$.
 Each automorphism of an object  in one of these conjugacy classes is conjugate  only to itself and to  automorphisms of different objects  in the same conjugacy class. As any composable sequence of morphisms in $\mathcal M$ involves only automorphisms of the same object,  the automorphism groups of these conjugacy classes in $\mathcal M_{inv}$ are given by $A_3\times A_3$. 
Thus,  the groupoid $\mathcal M_{inv}$ is as  in Figure \ref{fig:conjpic2}.
\end{example}

\section{Mapping class group actions}
\label{sec:mapclass}

In this section, we describe the mapping class group actions on the protected objects for connected closed surfaces. 
In the following $\Sigma$ is a   surface of genus $g\geq 1$ and $\Sigma\setminus D$ the associated surface with a disc removed and fundamental group $\pi_1(\Sigma\setminus D)=F_{2g}$. 

The mapping class group of $\Sigma$ is the quotient  of the group $\mathrm{Homeo}_+(\Sigma)$ of orientation preserving homeomorphisms of $\Sigma$ by the normal subgroup $\mathrm{Homeo}_0(\Sigma)$ of homeomorphisms homotopic to the identity.  It is isomorphic to the group of outer automorphisms of the fundamental group $\pi_1(\Sigma)$
$$\mathrm{Map}(\Sigma)=\mathrm{Homeo}_+(\Sigma)/\mathrm{Homeo}_0(\Sigma)\cong\mathrm{Out}(\pi_1(\Sigma))=\mathrm{Aut}(\pi_1(\Sigma))/\mathrm{Inn}(\pi_1(\Sigma)).$$
The mapping class group of $\Sigma\setminus D$ is defined analogously with the additional condition that all homeomorphisms fix the boundary of $D$ pointwise, see  Farb and Margalit \cite[Sec.~2.1]{FM}. The mapping class groups $\mathrm{Map}(\Sigma)$ and $\mathrm{Map}(\Sigma\setminus D)$ can be presented with the same generators but with additional relations for $\mathrm{Map}(\Sigma)$, see for instance the presentation by Gervais \cite{Ge}.

Mapping class group actions associated with protected objects for involutive and, more generally,  pivotal Hopf monoids in a finitely complete and cocomplete symmetric monoidal category  are constructed in  \cite{MV}. It is shown in \cite[Th.~9.2]{MV} that the mapping class group $\mathrm{Map}(\Sigma\setminus D)$ acts on the Yetter-Drinfeld module in Example \ref{ex:modcomodpi_1gen}  by automorphisms. 
By \cite[Th.~9.5]{MV}  this induces an action of 
$\mathrm{Map}(\Sigma)$ by automorphisms of its biinvariants. 

The mapping class group actions in \cite{MV} are obtained from a concrete presentation of the mapping class groups $\mathrm{Map}(\Sigma)$ and $\mathrm{Map}(\Sigma\setminus D)$   in terms of generating Dehn twists and relations. They associate to each generating Dehn twist a finite sequence of edge slides and prove that
resulting automorphisms of  $H^{\oo E}$ from Definition \ref{def:EdgeSlide}  satisfy the relations for $\mathrm{Map}(\Sigma\setminus D)$  in \cite{Ge}. 
The induced automorphisms of the protected object  then satisfy the additional  relations of   $\mathrm{Map}(\Sigma)$  in \cite{Ge}.  

As we established in Theorem \ref{th:topinv} that
the protected object is independent of the choice of the underlying graph, we can reformulate \cite[Th.~9.5]{MV} as follows. 

\begin{theorem}
\label{th:mcgact} Let $H$ be an involutive Hopf monoid in $\mac$ and $\Sigma$ an oriented surface of genus $g\geq 1$. Then the edge slides from Definition \ref{def:EdgeSlide}  induce an action of the mapping class group $\mathrm{Map}(\Sigma)$  by automorphisms of the protected object.
\end{theorem}

For  group objects in cartesian monoidal categories such as simplicial groups and crossed modules  this mapping class group action admits a   concrete description in terms of mapping class group actions on representation varieties.  
For this, recall that for any group  $G$  the group $\mathrm{Aut}(\pi_1(\Sigma))$ acts   on the set of group homomorphisms $\rho:\pi_1(\Sigma)\to G$  via $(\phi \rhd \rho)(\lambda)=\rho(\phi^\inv(\lambda))$ for all $\lambda\in \pi_1(\Sigma)$ and $\phi\in \mathrm{Aut}(\pi_1(\Sigma))$.  This induces an action of  $\mathrm{Map}(\Sigma)=\mathrm{Out}(\pi_1(\Sigma))=\mathrm{Aut}(\pi_1(\Sigma))/\mathrm{Inn}(\pi_1(\Sigma))$  on the representation variety  $\mathrm{Hom}(\pi_1(\Sigma), G)/G$.

To relate this to the mapping class group actions from  \cite[Th.~9.5]{MV} note that for a group object $H$ in a cartesian monoidal category the formulas for the edge slides in Definition \ref{def:EdgeSlide} and Example \ref{ex:edgeslide} reduce to left and right multiplication with $H$, sometimes composed with inversions.

It follows that any finite  sequence of edge slides from the standard graph to itself  induces an automorphism of $H^{\oo 2g}$ that arises from an automorphism of $F_{2g}=\pi_1(\Sigma\setminus D)$. As it preserves the Yetter-Drinfeld module structure in Example \ref{ex:modcomodpi_1gen}, it induces  automorphisms of $\mathcal M^{coH}$, $\mathcal M^H$ and $\mathcal M_{inv}$. Inner automorphism of $\pi_1(\Sigma)$ induce trivial automorphisms of $\mathcal M_{inv}$. 
For a group $H$ as a group object in $\Set$ it is then directly apparent that the induced action of $\mathrm{Map}(\Sigma)$ on $\mathcal M_{inv}$ is the one on the representation variety $\Hom(\pi_1(\Sigma), H)/H$, see also Examples 9.6 and 9.7  in \cite{MV}.  This result can be applied  to determine the mapping class group action for a simplicial group.

\begin{corollary}\label{prop:mapsimplicial}
Let $H=(H_n)_{n\in\N_0}$ be a simplicial group as a Hopf monoid in $\SSet$.  
Then the 
action of $\mathrm{Map}(\Sigma)$ on the representation varieties $\Hom(\pi_1(\Sigma), H_n)/H_n$ induces an action of $\mathrm{Map}(\Sigma)$ 
 on  $\mathcal M_{inv}$ by simplicial maps, and this coincides with the action in \cite[Th.~9.5]{MV}.
\end{corollary}
\begin{proof}
The induced $\mathrm{Map}(\Sigma)$-action on $\mathcal M_{inv}$ is by simplicial maps,  because
 the face maps and degeneracies of $\mathcal M_{inv}$ act elements of the representation varieties  $\Hom(\pi_1(\Sigma), H_n)/H_n$ by post-composition with the face maps and degeneracies $d_i: H_n\to H_{n-1}$ and $s_i: H_n\to H_{n+1}$, whereas $\mathrm{Map}(\Sigma)$ acts by pre-composition. 
This coincides with the action from \cite[Th.~9.5]{MV}, because the latter reduces to the $\mathrm{Map}(\Sigma)$-action on $\Hom(\pi_1(\Sigma), H_n)/H_n$ for the group $H_n$ as an involutive Hopf monoid in Set, and all (co)limits, images and (co)actions in $\SSet$ are degreewise.
\end{proof}

In the case of a crossed module as a Hopf monoid in $\Cat$, the mapping class group action on the protected object is  induced by the mapping class group action on the representation variety for the associated semidirect product group.

\begin{corollary}\label{cor:mapclasscross} Let $H=(B,A,\brhd,\partial)$ be a crossed module. 
Then the $\mathrm{Map}(\Sigma)$-action  on $\mathcal M_{inv}$ from Theorem \ref{th:mcgact}  is induced by the $\mathrm{Map}(\Sigma)$-action on  $\Hom(\pi_1(\Sigma), A\rtimes B)/A\rtimes B$.
\end{corollary}

\begin{proof} As the group structure of $H$ as a group object in $\Cat$ is the one of the semidirect product $A\rtimes B$, the 
$\mathrm{Map}(\Sigma\setminus D)$-action on  $\mathcal M=H^{\times 2g}$ for the standard graph \eqref{eq:standardgraph}  can be identified with the $\mathrm{Map}(\Sigma\setminus D)$-action on  $\mathcal M=(A\rtimes B)^{\times 2g}$ one for the group $A\rtimes B$ as a group object in $\Set$. The crossed module structure ensures that this  $\mathrm{Map}(\Sigma\setminus D)$-action respects the category structure of $(A\rtimes B)^{\times 2g}$ and defines a  $\mathrm{Map}(\Sigma\setminus D)$-action by invertible endofunctors.

The  $\mathrm{Map}(\Sigma\setminus D)$-action on $\mathcal M$ induces the $\mathrm{Map}(\Sigma)$-action on the protected object $\mathcal M_{inv}$ for both,  the group $A\rtimes B$ as a group object in $\Set$ and for $H$ as  a group object in $\Cat$. The former is the action on the representation variety $\Hom(\pi_1(\Sigma), A\rtimes B)/A\rtimes B$. As the protected object $\mathcal M_{inv}$  is a quotient of this representation variety by Theorem \ref{th:kitaevexplicit}, its $\mathrm{Map}(\Sigma)$-action  is induced by the $\mathrm{Map}(\Sigma)$-action on the representation  variety.
\end{proof}

\begin{example} We consider the mapping class group action  on the groupoids $\mathcal M_{inv}$ from Example \ref{ex:minv}, \ref{ex:minv2} for the crossed module $(S_3,A_3,\brhd,\partial)$ and the torus.

The mapping class group of the torus $T$ is the group 
\begin{align}
&\mathrm{Map}(T)=\mathrm{SL}(2,\Z)=\langle D_a,D_b\mid D_aD_bD_a=D_bD_aD_b, (D_aD_bD_a)^4=1\rangle.
\end{align}
It is generated by the Dehn twists $D_a, D_b$ along the $a$- and $b$-cycle, which act on $\pi_1(T)=\Z\times\Z$ by 
\begin{align}\label{eq:dtform}
&D_a: a\mapsto a, b\mapsto b-a & &D_b: a\mapsto a+b, b\mapsto b.
\end{align}

In both, Example \ref{ex:minv} and \ref{ex:minv2}, the $\mathrm{SL}(2,\Z)$-action on the objects of $\mathcal M_{inv}$ is 
 the $\mathrm{SL}(2,\mathbb Z)$-action on the representation variety $\Hom(\Z\times\Z,  S_3)/S_3$ with orbits $\{C_1\}$, $\{C_2, C'_2, C_3,C_4\}$ and $\{C_5, C'_5, C_6\}$.
 
In Example \ref{ex:minv} the $\mathrm{SL}(2,\mathbb Z)$-action on $\mathcal M_{inv}$ is determined uniquely by the action on the objects. This follows, because for all choices of  objects $s,t\in \mathrm{Ob}\,\mathcal M_{inv}$ the groupoid
$\mathcal M_{inv}$  has at most one morphism $f: s\to t$.  In Example \ref{ex:minv2} an analogous statement holds for morphisms between the objects  $C_1,C_5,C'_5, C_6$, since all of them are identity morphisms.

In contrast, the $\mathrm{SL}(2,\mathbb Z)$-action on the automorphisms of $C_2, C'_2, C_3, C_4$ in Example \ref{ex:minv2} is non-trivial and can be identified with an orbit of the $\mathrm{SL}(2,\Z)$-action on $\mathrm{Mat}(2\times 2,\Z_3)$ by left multiplication. In this action, $D_a$ and $D_b$ correspond to left-multiplication with the generators
\begin{align*}
A= \begin{pmatrix} 1 & 0\\ 1 & 1 \end{pmatrix}, \quad \quad B=\begin{pmatrix} 1 & -1\\  0 & 1 \end{pmatrix}.
\end{align*}
 Automorphisms of $C_2, C'_2, C_3, C_4$ in $\mathcal M_{inv}$ are given by group homomorphisms  $\tau: \Z\times\Z\to A_3^{\times 2}\cong \Z_3^{\times 2}$,  which are determined by the images $\tau(1,0), \tau(0,1)\in \Z_3\times \Z_3$. Interpreting an element $(c,d)\in \Z_3\times\Z_3$ as an automorphism $c: d\to d$ and taking $\tau(0,1)$ as the first and $\tau(1,0)$ as the second row of a matrix, we find that the $\mathrm{SL}(2,\mathbb Z)$-action induced by \eqref{eq:dtform}
coincides with the $\mathrm{SL}(2,\Z)$-orbit containing those matrices whose second column is non-trivial.
\end{example}

As our construction yields objects equipped with mapping class group actions and assigns the tensor unit to the sphere  $S^2$, it is natural to ask if the protected objects satisfy the axioms of a modular functor from  \cite[Def 5.1.1]{BK}. Although the latter are formulated for categories of vector spaces, they have obvious generalisations to other symmetric monoidal categories.  

 However, the assignment of  protected objects to surfaces can in general not be expected to satisfy these axioms. 
The problem is axiom (iii) in \cite[Def 5.1.1]{BK}, which requires that the object assigned to a disjoint union $\Sigma_1\amalg \Sigma_2$ of surfaces is the tensor product of the objects assigned to  $\Sigma_1, \Sigma_2$. This does in general not hold for the protected objects from Definition \ref{def:KitaevModel}, as they are constructed by taking equalisers and coequalisers.  The tensor product of two (co)equalisers in a symmetric monoidal category $\mac$ is in general not a (co)equaliser of their tensor product. 
This is already apparent in the symmetric monoidal category $\mathrm{Ab}=\Z\mathrm{-Mod}$ with the usual tensor  product that does  not preserve equalisers.
 Nevertheless, the construction satisfies this axiom, if the underlying symmetric monoidal category is $\Set$, $\SSet$ or $\Cat$.

\begin{proposition}\label{Lemma:ProductOfConnectedComponents}
Let $H$ be a group object in $\mathcal C=\Set$, $\SSet$ or $\Cat$ and $\Sigma$ an oriented surface  with connected components $\Sigma_1, \ldots, \Sigma_k$. Then the protected object for $\Sigma$ is the product of the protected objects for $\Sigma_1,\ldots,\Sigma_k$.
\end{proposition}
\begin{proof}  The claim follows by induction over $k$, and it is sufficient to consider $k=2$. 
By Theorem \ref{th:topinv} we can compute the protected object for $\Sigma$ by choosing a standard graph $\Gamma_i$ from \eqref{eq:standardgraph} on each connected component $\Sigma_i$.  We denote by $E_i$ the edge set of $\Gamma_i$ and by $E=E_1\cup E_2$ the edge set of $\Gamma$.

The (co)actions
$\rhd$ and  $\delta$ for $\Gamma$ are then given by formula \eqref{eq:ComposedVertexAction}, and it follows directly that they are the  products of the (co)actions $\rhd_i$, $\delta_i$ for $\Gamma_i$, up to braidings. Lemma \ref{lemma:VertexAndFaceOperatorsCommmute} and some simple computations imply that $(H^{\times E}, \rhd, \delta)$ is a Yetter-Drinfeld module  over  $H\times H$. 

We denote by $F_\Gamma: H^{\times E}\to H\times H$ the morphism  from Example \ref{Ex:YDAndHopfModulesInCat} for $\Gamma$ and  by $F_{\Gamma_i}: H^{\times E_i}\to H$ the corresponding morphisms for $\Gamma_i$. 

 $\bullet \quad \mathcal C=\mathrm{Set}$: As in Example \ref{Ex:ImageInSet,Top} we obtain
\begin{align*}
M^{coH}&= F_\Gamma^\inv(1)=\{(x, y)\in H^{\times E_1}\times H^{\times E_2}: (F_{\Gamma_1}(x), F_{\Gamma_2}(y))=(1,1)\}\, = \, M^{coH}_1\times M^{coH}_2,\\
M^H&= \{H^{\times 2} \rhd (m_1, m_2): m_1\in H^{\times E_1}, m_2\in H^{\times E_2}\}= M^H_1 \times M^H_2
\end{align*}
with inclusions $\iota=(\iota_1, \iota_2): M^{coH}\to H^{\times E}$ and canonical surjections $\pi=(\pi_1, \pi_2): H^{\times E}\to M^H$. As the image of a morphism $f:A \to B$ in $\mathrm{Set}$ is the usual image of a map, we have  $\mathrm{im}((f_1, f_2))=(\mathrm{im}(f_1), \mathrm{im}(f_2))$ and $M_{inv}=M_{inv,1}\times M_{inv,2}$.

$\bullet \quad \mathcal C=\mathrm{SSet}$:  By Proposition \ref{prop:binvsimplicial} 
the  coinvariants are given by the sets
\begin{align*}
M^{coH}_n =\{(m_1, m_2)\in (H^{\times E_1}\times H^{\times E_2})_n: (F_{\Gamma_1,n}(m_1), F_{\Gamma_2,n}(m_2))=(1, 1)\} \, = \, M^{coH}_{1,n} \times M^{coH}_{2,n}.
\end{align*}
Face maps and degeneracies are induced by the ones of $H^{\times E_1} \times H^{\times E_2}$ and the simplicial map $\iota: M^{coH}\to H^{\times E}$ is given by the maps $\iota_n=(\iota_{1,n}, \iota_{2,n})$. As the product in $\mathrm{SSet}$ is objectwise, this yields $M^{coH}=M^{coH}_1 \times M^{coH}_2$. 
An analogous argument shows that the sets $M^H_n$, $(M_{inv})_n$ from Proposition \ref{prop:binvsimplicial} are given by
 $M^H_n=M^H_{1,n}\times M^H_{2,n}$ and  $(M_{inv})_n=(M_{inv,1})_n\times (M_{inv,2})_n$. 

$\bullet \quad \mathcal C=\mathrm{Cat}$: By Lemma \ref{Lemma:CoinvariantsInCat} the coinvariants $M^{coH}$  for the comodule $M=H^{\times E}\cong H^{E_1}\times H^{\times E_2}$ are the subcategory with objects 
\begin{align*}
&\mathrm{Ob}(M^{coH})= \{(A_1, A_2)\mid  A_i\in \text{Ob}(H^{\times E_i}), F_{\Gamma_i}(A_i)=e\}\\
&\mathrm{Hom}_{M^{coH}}((A_1,A_2),(A'_1,A'_2))=\{ (f_1,f_2)\mid    f_i\in \mathrm{Hom}_{M}(A_i, A'_i),  F_{\Gamma_i}(f_i)=1_e\}
\end{align*}
and hence  $M^{coH}$ is the product category $M^{coH}_1 \times M^{coH}_2$. For the invariants we can apply Lemma \ref{def:coequCat} and the results from $\mathrm{SSet}$. 
As a right adjoint, the nerve $N$ preserves products,  and hence $N(\rhd)=N(\rhd_1)\times N(\rhd_2)$, up to braidings,  and the same holds for the trivial actions $\epsilon\oo 1_{H^{\oo E}}$ and $\epsilon\oo 1_{H^{\oo E_i}}$.  It follows that the coequaliser of $N(\rhd)$ and $N(\epsilon^{\times 2} \times 1_{H^{\times E}})$  is the product of the coequalisers of  $N(\rhd_i)$ and  $N(\epsilon\times 1_{H^{\times E_i}})$. As the homotopy functor preserves finite products, see for instance \cite[Prop. 1.3]{Jo}, this yields $M^H=M^H_1 \times M^H_2$ and $M_{inv}=M_{inv,1}\times M_{inv,2}$. 
\end{proof}

\subsection*{Acknowledgements} A.-K.~Hirmer gratefully acknowledges a PhD fellowship of the Erika Giehrl foundation, Friedrich-Alexander-Universit\"at Erlangen-N\"urnberg.

\end{document}